\documentclass{article}

\input{preamble.sty}
\graphicspath{ {figs/} }

\usepackage{lmodern}

\newcommand{\ttE}{\texttt{E}}
\newcommand{\ttV}{\texttt{V}}
\newcommand{\onevec}{\vec{1}}
\newcommand{\Divergence}{\mathrm{div}}
\newcommand{\Unif}{\mathrm{Unif}}
\newcommand{\ind}{\perp\!\!\!\!\perp}

\begin{document}
	
% Title
\title{ {\bf Two-Sample Testing with a Graph-Based Total Variation Integral Probability Metric} }
\author{Alden Green, Sivaraman Balakrishnan, Ryan J. Tibshirani}
\maketitle
\RaggedRight

\begin{abstract}
	We consider a novel multivariate nonparametric two-sample testing problem where, under the alternative, distributions $P$ and $Q$ are separated in an integral probability metric over functions of bounded total variation (TV IPM). We propose a new test, the graph TV test, which uses a graph-based approximation to the TV IPM as its test statistic. We show that this test, computed with an $\varepsilon$-neighborhood graph and calibrated by permutation, is minimax rate-optimal for detecting alternatives separated in the TV IPM. As an important special case, we show that this implies the graph TV test is optimal for detecting spatially localized alternatives, whereas the $\chi^2$ test is provably suboptimal. Our theory is supported with numerical experiments on simulated and real data.
\end{abstract}

\section{Introduction}

In nonparametric two-sample testing, one observes independent samples $X_1,\ldots,X_{n_1} \sim P$ and $Y_1,\ldots,Y_{n_2} \sim Q$, all belonging to $\Rd$, and uses these as evidence to determine whether or not to reject the null hypothesis that $P = Q$. This is a classical statistical problem with many applications, and the problem has also received renewed interest in the machine learning community.

In this last context, a good deal of recent attention has been paid to test statistics involving \emph{integral probability metrics} (IPMs), also sometimes referred to as maximum mean discrepancies (MMDs)~\citep{gretton2012kernel}. Originally introduced in the probability theory community~\citep{muller1997integral}, an IPM is a distance between probability measures $P$ and $Q$ of the form
\begin{equation*}
d_{\mc{F}}(P,Q) = \sup_{f \in \mc{F}} \; \mathbb{E}_{P}[f(X)] - \mathbb{E}_{Q}[f(Y)],
\end{equation*}
i.e. it measures the maximum difference of means over all functions $f$ in a function class $\mc{F}$.

% AJG: Kept around in case wanted to go back to the old way of writing it.
%Originally introduced in the probability theory community~\citep{muller1997integral}, an IPM is a distance between probability measures $P$ and $Q$ of the form
%\begin{equation*}
%d_{\mc{F}}(P,Q) = \sup_{f: \|f\|_{\mc{F}} \leq 1} \; \mathbb{E}_{P}[f(X)] - \mathbb{E}_{Q}[f(Y)],
%\end{equation*}
% where $\{f:\|f\|_{\mc{F}} \leq 1\}$ is a collection of norm-bounded functions in a function class $\mc{F}$.

The statistical properties of an IPM-based test statistic are (obviously) determined by $\mc{F}$. For univariate distributions, different choices of $\mc{F}$ recover a number of fundamental probability metrics such as the total variation, Cramer-von-Mises~\citep{cramer1928composition,vonmises1933wahrscheinlichkeit}, Wasserstein 1-~\citep{kantorovich1942translocation,vaserstein1969markov}, and Kolmogorov-Smirnov~\citep{kolmogorov1933sulla,smirnov1948table} distances. For multivariate data, a popular IPM takes $\mc{F}$ to be a reproducing kernel Hilbert space~\citep{gretton2012kernel}. Consideration of each of these distances leads to tests with minimax-optimal power against different classes of alternatives. Indeed, a principal advantage of IPMs is that it is possible to design tests with high power against specific kinds of alternatives, simply by changing the collection of functions $\mc{F}$.

In this paper we introduce a new multivariate nonparametric two-sample test based on an IPM. Specifically, we consider the distance obtained by taking $\mc{F}$ to be the space of functions with finite \emph{total variation}. The total variation of an integrable function $f \in L^1(\Rd)$ is 
\begin{equation}
\label{eqn:tv}
\mathrm{TV}(f) := \sup \; \Bigl\{\int f \cdot \mathrm{div}\phi: \phi = (\phi_1,\ldots,\phi_d) \in C_c^{1}(\Rd;\Rd), ~~ \|\phi(x)\|_2 \leq 1 ~~\textrm{for all $x \in \Reals^d$}\Bigr\},
\end{equation}
where $\Divergence(\phi) := \sum_{i = 1}^{d} \partial \phi_i/\partial x_i$ is the divergence of a smooth vector field $\phi$, and $C_c^{1}(\Rd;\Rd)$ is the set of compactly supported smooth vector fields $\phi: \Rd \to \Rd$. The \emph{total variation IPM} we consider is
\begin{equation}
	\label{eqn:tv-ipm}
	d_{\mathrm{BV}}(P,Q) := \sup_{f:\TV(f) \leq 1} \; \E_{P}[f(X)] - \E_{Q}[f(Y)].
\end{equation} 
The subscript $\BV$ refers to the fact that functions with $\TV(f) < \infty$ are commonly referred to as functions of \emph{bounded variation}, denoted $\BV(\Rd)$. The notation is also intended to make clear that the IPM defined in~\eqref{eqn:tv-ipm} is distinct from the total variation distance between two probability measures, which is $d_{\TV}(P,Q) := \sup |P(A) - Q(A)|$.

Models based on TV smoothness are widely used in fields like image processing~\citep{rudin1992nonlinear,vogel1996iterative,chambolle1997image,chan2000high}, and nonparametric regression~\citep{koenker1994quantile,mammen1997locally,tibshirani2014adaptive}. In part, this is because TV is a heterogeneous notion of smoothness: speaking loosely, it allows functions to be wiggly or even discontinuous in certain parts of their input space as long as they are sufficiently smooth over the rest of the domain. Additionally, TV gives a reasonable notion of regularity in many instances. For example, in one dimension the total variation of a step function is simply the sum of the heights of the steps. In fact, one can use this property, along with the relationship between a CDF and the expectation of a step function, to show that when $d = 1$ the TV IPM~\eqref{eqn:tv} is equal to the Kolmogorov-Smirnov distance~\citep{muller1997integral}. This means the univariate TV IPM is equivalent to a fundamental nonparametric distance. However, we are not aware of any work investigating the TV IPM for $d \geq 2$. 

The focus of our article is a multivariate hypothesis testing problem where, under the alternative $P \neq Q$, the distance $d_{\BV}(P,Q)$ between $P$ and $Q$ in the TV IPM is sufficiently large. To come up with a test statistic for this problem, we need a way of estimating the TV IPM from samples $X_1,\ldots,X_{n_1} \sim P$, $Y_{1},\ldots,Y_{n_2} \sim Q$, and this turns out to be somewhat subtle. A seemingly natural statistic is the plug-in estimate
$$
d_{\BV}(P_{n_1},Q_{n_2}) = \sup_{\TV(f) \leq 1}\Big\{\frac{1}{n_1}\sum_{i = 1}^{n_1} f(X_i) - \frac{1}{n_2}\sum_{j = 1}^{n_2} f(Y_j)\Big\}.
$$
Indeed these kinds of empirical IPMs are commonly used as estimates and test statistics, for instance when $\mc{F}$ is an RKHS, or $\mc{F}$ is the collection of univariate functions of bounded variation. For our problem, however, the plug-in estimate is not suitable, since when $d \geq 2$ the TV IPM between two empirical measures is infinite: $d_{\BV}(P_{n_1},Q_{n_2}) = \infty$.\footnote{To see this explicitly, consider a suitable sequence of bump functions $f_r(x) = r^{-(d - 1)}\1(x \in B(X_i,r))$, centered around some $X_i$ that is distinct from any $Y_j$. Each $f_r$ has unit TV and so is feasible for the optimization problem in~\eqref{eqn:tv-ipm}, but by driving $r \to 0$ we can blow the criterion up to $\infty$.} This is fundamentally because $\BV(\Rd)$ is a quite rich function class -- compared to (say) an RKHS -- for which point evaluation is not continuous.

Instead we propose a new test, the \emph{graph TV test}, involving a test statistic that uses a graph-based approximation to the TV IPM. In more detail, let $G = (\texttt{V},\texttt{E})$ be an undirected, unweighted graph with $n = n_1 + n_2$ vertices at the combined samples $\texttt{V} = \mc{X}_{n_1} \cup \mc{Y}_{n_2}$, where $\mc{X}_{n_1} = \{X_1,\ldots,X_{n_1}\}$ and $\mc{Y}_{n_2} = \{Y_1,\ldots,Y_{n_2}\}$. Then the \emph{graph TV IPM} is the solution to the finite-dimensional optimization problem
\begin{equation}
\begin{aligned}
\label{eqn:graph-tv-ipm}
& \max \Big\{\frac{1}{n_1}\sum_{i = 1}^{n_1} f(X_i) - \frac{1}{n_2}\sum_{j = 1}^{n_2} f(Y_j)\Big\} \\
& \textrm{subject to}~~ \sum_{i,j = 1}^{n} \big|f(Z_i) - f(Z_j)\big| \cdot \1(\{Z_i,Z_j\} \in \texttt{E}) \leq 1,
\end{aligned}
\end{equation}
where $\mc{Z}_n = \{Z_1,\ldots,Z_n\} = \mc{X}_{n_1} \cup \mc{Y}_{n_2}$ is the combined set of samples, and $\ttE \subseteq \mc{Z}_n \times \mc{Z}_n$. Defining the optimization domain in terms of a functional defined over a graph $G$ leads to a test statistic that is practically reasonable to compute, and that is finite when $G$ is connected.

Theoretically, we study our hypothesis testing problem from the perspective of the \emph{detection boundary}: informally, this is the minimum distance  $d_{\BV}(P,Q)$ between $P$ and $Q$ required for some level-$\alpha$ test to have non-trivial power, meaning power of at least (say) $\frac{1}{2}$. Our main results characterize the minimax-optimal rate of convergence of the detection boundary under suitable regularity conditions, and show that the graph TV test, suitably tuned and calibrated by permutation, achieves this optimal rate. A more detailed summary of these theoretical results is given in Section~\ref{subsec:main-results}. First, we demonstrate several properties of the graph TV test in a simple but revealing empirical example.

\subsection{Illustrative example}

For our illustrative example, we sample $n_1 = n_2 = 1000$ observations from the two-dimensional mixture distributions
\begin{align*}
	P & = (1 - \pi) P_0 + \pi\Unif(B(x_P,\eta)) \\
	Q & = (1 - \pi) P_0 + \pi \Unif(B(x_Q,\eta)), \\
	x_P & = (1 ~ 5)^{\top}, x_Q = (5 ~ 1)^{\top}, \eta = 0.5, \pi = 0.02.
\end{align*}
Here $P_0$ is the product distribution of two independent Laplace random variables, and $\Unif(B(x,\eta))$ is the uniform distribution over a ball centered at $x \in \R^2$, with radius $\eta$. The important thing to note about these distributions is that the level sets of $P - Q$ have sparse support compared to the level sets of $P$ or $Q$. We will see later that these kinds of \emph{spatially localized} departures from the null (or just spatially localized alternatives, for short) play an important role in understanding the hypothesis testing problem where $P,Q$ are separated in TV IPM.

Figure~\ref{fig:illustrative-example} measures the power of the graph TV test using a $10$-nearest neighbors graph. The number of neighbors is chosen to make the graph sparse while still being connected with high probability; theoretical support for this choice is given in Section~\ref{sec:optimality}. As a benchmark, we compare to the popular kernel MMD test, computed using a Gaussian kernel with various bandwidths. The receiver-operator characteristic curve of each test is plotted, to examine power at many different levels of empirical type I error. We can see that the graph TV test has better power than the kernel test, across various levels of type I error, regardless of the choice of bandwidth.

\begin{figure}[htb]
	\centering
	\begin{minipage}{.55\textwidth}
		\begin{subfigure}[t]{.49\textwidth}
			\includegraphics[width=\linewidth]{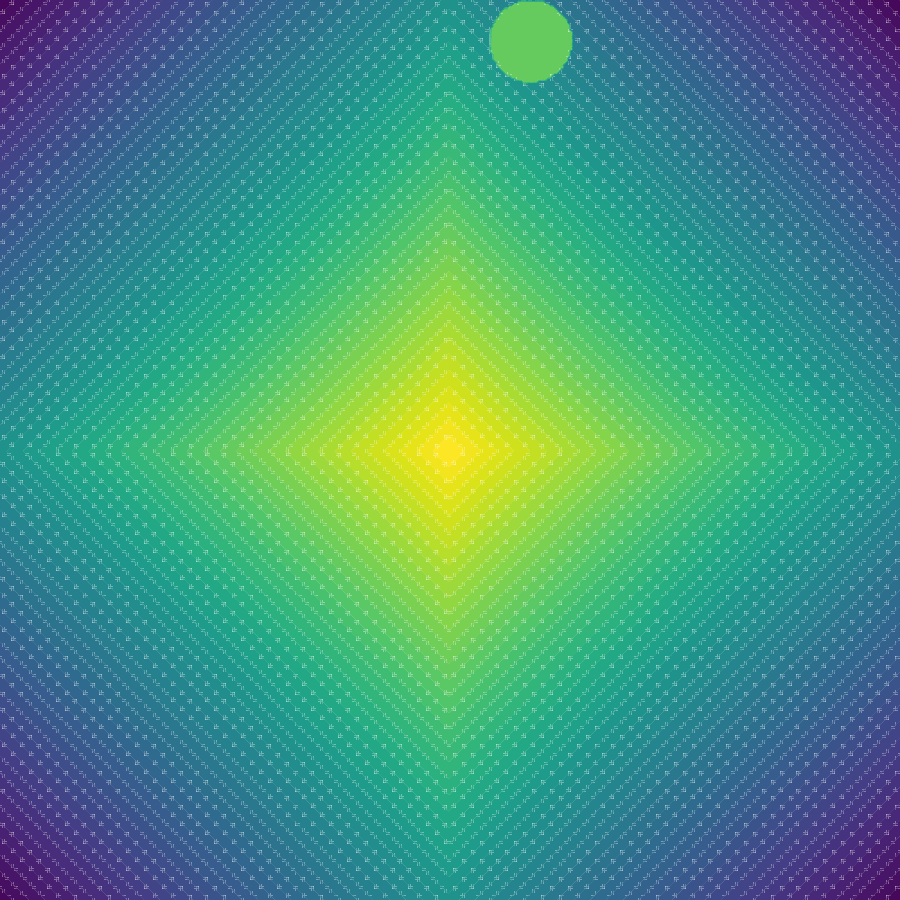}
			\caption{Level sets of $P$}
		\end{subfigure}
		\hfill
		\begin{subfigure}[t]{.49\textwidth}
			\includegraphics[width=\linewidth]{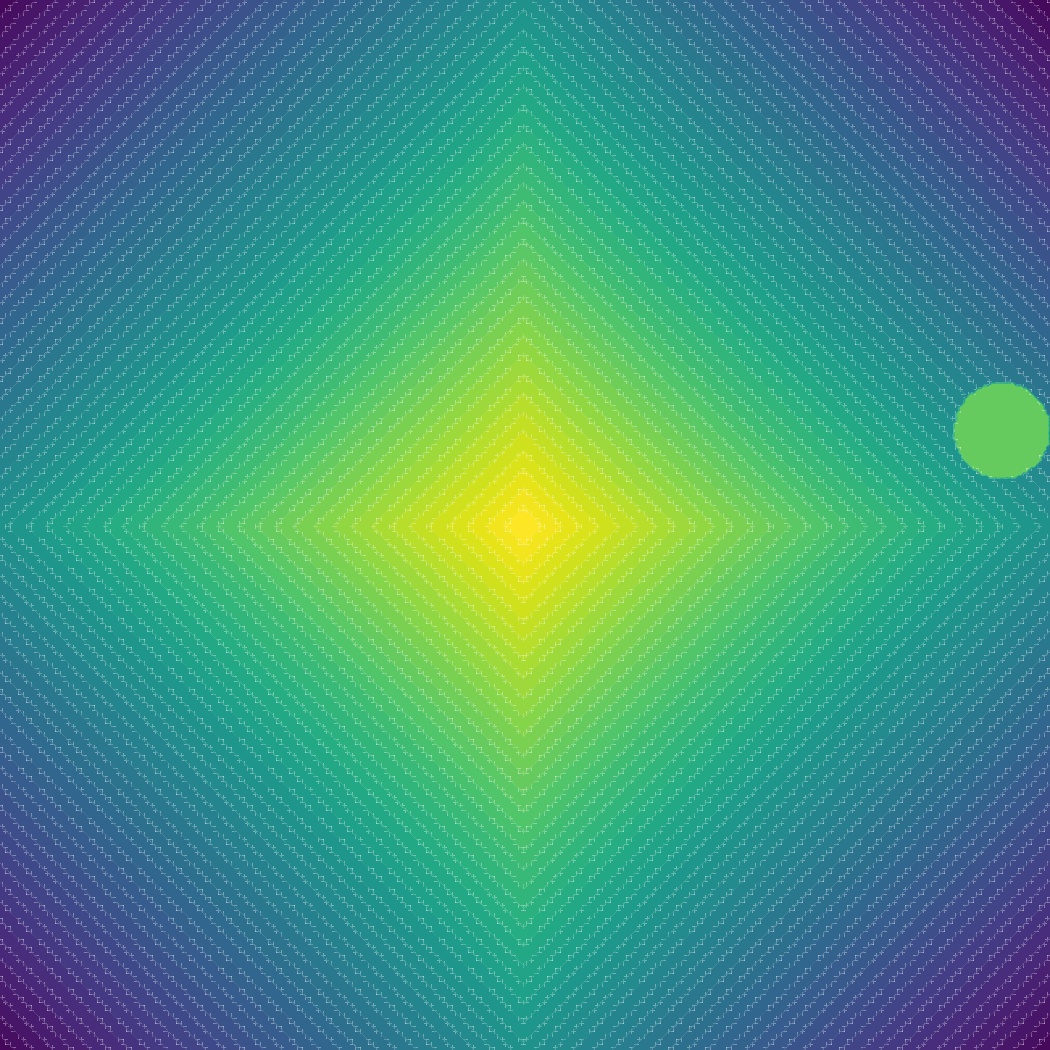}
			\caption{Level sets of $Q$}
		\end{subfigure} \\
		\begin{subfigure}[t]{.49\textwidth}
			\includegraphics[width=\linewidth]{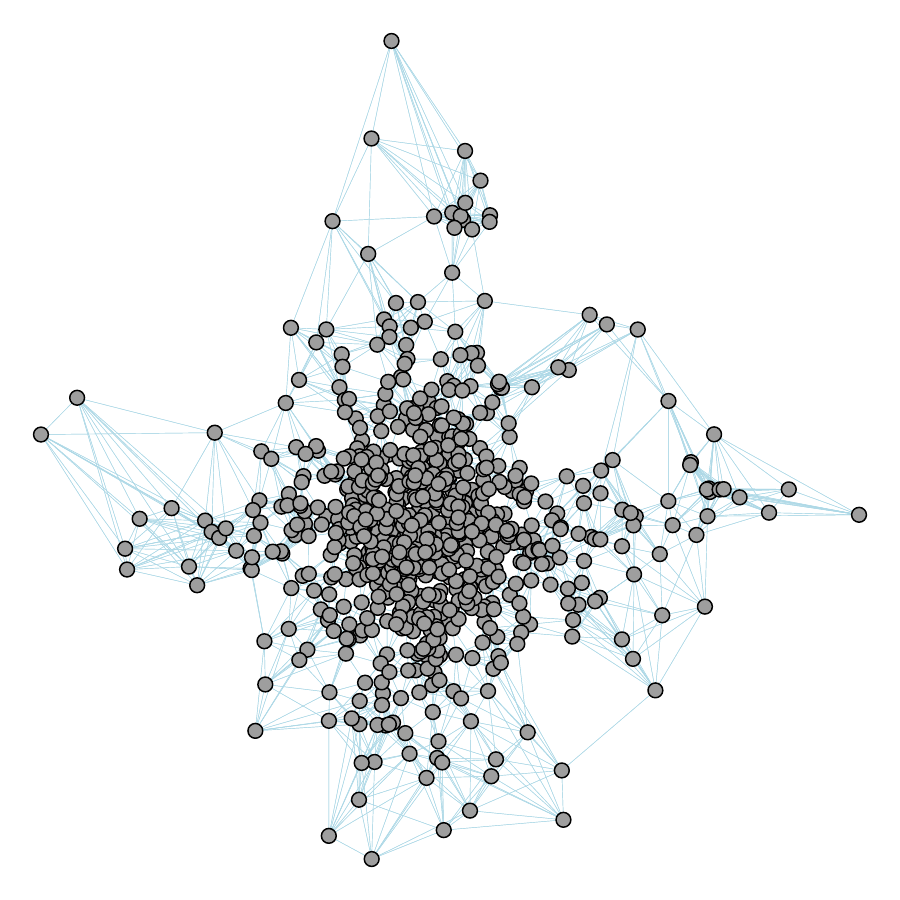}
			\caption{$10$-nearest neighbors graph}
		\end{subfigure}
		\hfill
		\begin{subfigure}[t]{.49\textwidth}
			\includegraphics[width=\linewidth]{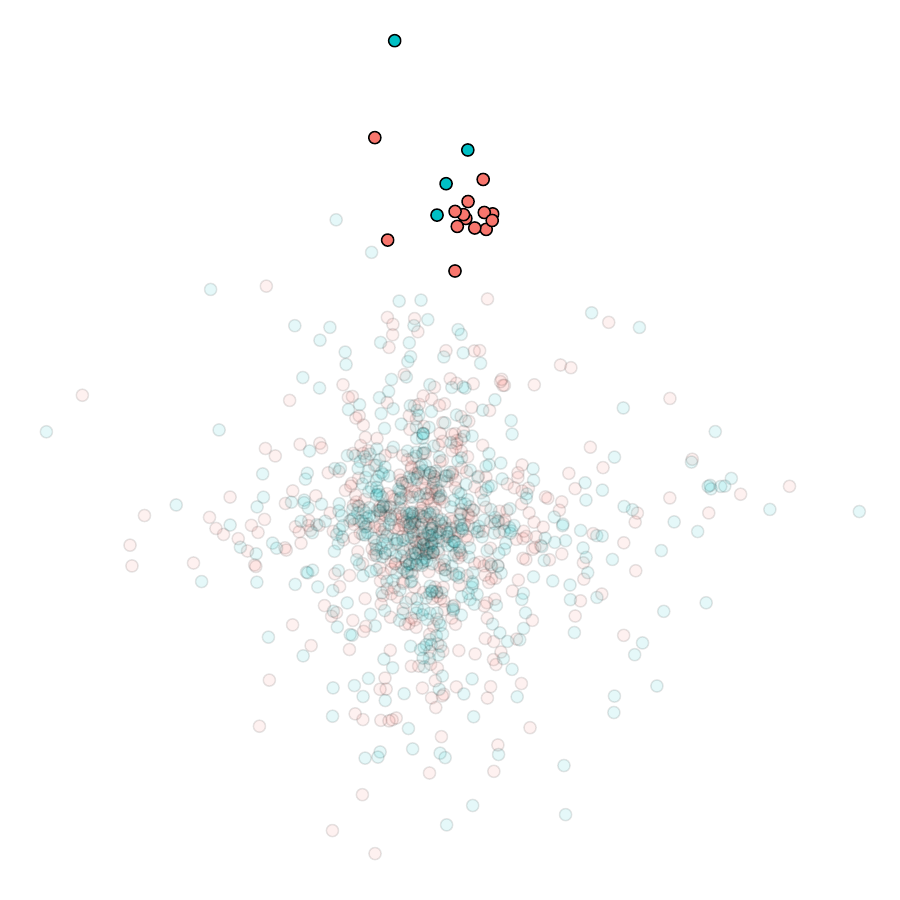}
			\caption{Graph TV IPM witness}
		\end{subfigure}
	\end{minipage}
	\hfill
	\begin{minipage}{.4\textwidth}
		\begin{subfigure}[t]{.99\textwidth}
			\includegraphics[width=\linewidth]{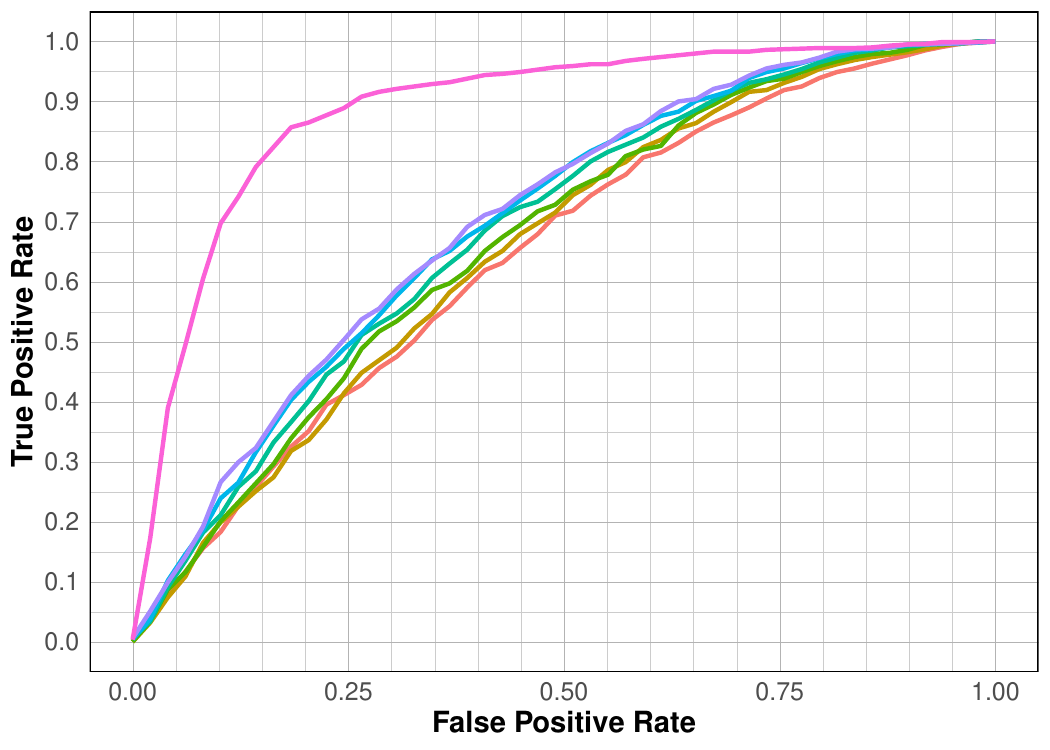}
			\caption{ROC curve of graph TV test (\textcolor{pink}{pink}) and Gaussian kernel MMD tests.}
		\end{subfigure}
	\end{minipage}
	\caption{Illustrative example.}
	\label{fig:illustrative-example}
\end{figure}

To better understand why the graph TV test effectively solves this problem, we can look at its \emph{witness} function. A witness of an IPM is a function $f^{\ast} \in \mc{F}$ that achieves the maximum difference in means. Qualitatively speaking, examining regions where the witness of an IPM is ``large'' can thus be an effective way of interpreting why an IPM-based test has rejected, or failed to reject, the null. This can be particularly useful when the witness is just the indicator function of a set, as is the case for the total variation distance or univariate Kolmogorov-Smirnov distance. Later, we show that the witness function of the graph TV IPM is also always an indicator function of a set, and so can be interpreted in the same way. 

% TODO: Previous paragraph needs more work.

In this specific example, Figure~\ref{fig:illustrative-example}d show that the witness of the graph TV IPM puts all of its mass on a small region around $x_P$. This is exactly where the density of $P$ is much larger than the density of $Q$. Put simply, the IPM correctly identifies a region where $P$ and $Q$ are significantly different. 

Intuitively, a hypothesis test designed to detect spatially localized alternatives should ``hone in'' on the area where ``the action is happening'', and then test for the presence or absence of signal in this area. The results of the simulation indicate that the graph TV IPM is doing something along these lines. The theory developed in this paper will show that the graph TV test is essentially optimal for a broad class of problems that include spatially localized alternatives as an important special case.

\subsection{Outline and summary}
\label{subsec:main-results}
Here is an outline of our paper, with a brief description of some our main results. 

Section~\ref{sec:graph-tv-ipm} describes our proposed test statistic and test in more detail, including discussion of the choice of graph $G$, and calibration of the test. 

Section~\ref{sec:representation} is about representation. A series of equivalences establish that the graph TV IPM is \emph{always} witnessed by a binary-valued function. We use this fact to develop a strategy for  computing the test statistic by solving a series of max-flow problems, as well as provide several interpretations of the graph TV IPM in terms of classification and clustering.

Section~\ref{sec:optimality} contains our main theoretical results, characterizing the detection boundary of the hypothesis testing problem where, under the alternative, $P \neq Q$ are separated in the TV IPM. For simplicity, suppose that $d \geq 3$ and that $n_1 = n_2 = \frac{n}{2}$. (The formal theorem statements make no assumption regarding balanced class sizes.) Then these results can be summarized as follows:
\begin{itemize}
	\item Theorem~\ref{thm:graph-tv} shows that when $P,Q$ have densities bounded away from $0$ and $\infty$ and $d_{\BV}(P,Q) \gtrsim (\log n/n)^{1/d}$, then a graph TV test, using an $\varepsilon$-neighborhood graph and calibrated via permutation, has high power.
	\item Theorem~\ref{thm:lower-bound} shows that under the same conditions, no test can have high power if $d_{\BV}(P,Q) \ll (\log n/n)^{1/d}$. The upper and lower bounds match up to constant factors, establishing the rate of convergence for the detection boundary in this problem, and showing that the graph TV test is rate-optimal for detecting differences in $d_{\BV}(P,Q)$.
\end{itemize}
We also examine the implications of this theory for detecting a class of spatially localized alternatives $\mc{P}_{\eta}$, where $\eta$ is the diameter of the support of $P - Q$, and thus determines the degree of spatial localization.
\begin{itemize}
	\item Corollaries~\ref{cor:upper-bound-spatially-localized} and~\ref{cor:lower-bound-spatially-localized} show that the graph TV test is rate-optimal for this class, in the sense that it has non-trivial power so long as $\eta \gtrsim (\log n/n)^{1/d}$, while no test can be have non-trivial power if $\eta \ll (\log n/n)^{1/d}$.
	\item  We also consider a chi-squared test based on binning the domain, and show in Theorem~\ref{thm:lower-bound-chi-squared} that this test is suboptimal for the same problem: no matter how well the number of bins is chosen, the chi-squared test will have trivial power if $\eta \ll n^{-2/3d}$.
\end{itemize}
All of our theory extends to the bivariate setting $d = 2$, but in this case our upper and lower bounds differ by a $\sqrt{\log n}$ factor. We do not consider the univariate setting $d = 1$, since in that case the TV IPM is simply the usual Kolmogorov-Smirnov distance, and the minimax optimal rates for detecting alternatives separated in Kolmogorov-Smirnov distance are well-understood~\citep{ingster2003nonparametric}. 

 Section~\ref{sec:continuum-limit} describes the exact asymptotic behavior of the graph TV IPM in the continuum limit, showing that it converges to a ``density-weighted'' TV IPM rather than the unweighted TV IPM of~\eqref{eqn:tv-ipm}. We discuss situations in which this density-weighting might be useful, and on the other hand, ways of eliminating the density-weighting when it is not desired.

Section~\ref{sec:experiments} contains some numerical experiments with synthetic data and a real data application. Extensions of the graph TV test to goodness-of-fit and specification testing are described in Section~\ref{sec:discussion}.

\subsection{Related work}
\label{subsec:related-work}

As we have already mentioned, this paper proposes a novel multivariate nonparametric distance and hypothesis test based on an IPM. The majority of work on statistical inference in multiple dimensions with IPMs concerns the kernel MMD -- an early reference is \citet{gretton2012kernel} -- but the class of functions of bounded variation is not an RKHS and thus the TV IPM is not a kernel MMD. There has also been some recent interest in using Wasserstein $p$-distances for multivariate statistical inference~\citep{chernkozhukov2017monge,hallin2021distribution,hallin2021multivariate}. Wasserstein $p$-distances for $p > 1$ are not IPMs, but the special case of the Wasserstein $1$-distance between two distributions $P,Q$ with bounded support corresponds to the IPM where $\mc{F}$ contains all continuous functions with Lipschitz constant of at most $1$. 

In recent work two of us (the authors) and collaborators have proposed another class of multivariate IPMs which we called the Radon KS distances~\citep{paik2023maximum}, based on a different notion of multivariate variation called Radon (total) variation. Both TV and (degree-$0$) Radon variation are multivariate generalizations of univariate total variation, and thus both IPMs reduce to the same metric -- the Kolmogorov-Smirnov distance --  when $d = 1$; however, they are not the same when $d \geq 2$. For instance, (a special case of) the Radon KS distance is always witnessed by an indicator of a halfspace, $\1(w^{\top}x \leq b)$, while the TV IPM can be witnessed by indicator functions of a much richer class of sets.  

There are also several multivariate nonparametric two-sample tests involving graphs.~\cite{friedman1979multivariate} use minimum spanning trees to generalize the Wald-Wolfowitz runs test and Kolmogorov-Smirnov test to higher dimensions.~\cite{schilling1986multivariate,henze1988multivariate} propose tests involving $k$-nearest neighbors graphs. Tests based on data-depth~\citep{liu1993quality} can be formulated as graph-based tests~\citep{bhattacharya2019general}. More recent graph-based testing proposals include~\cite{rosenbaum2005exact,chen2017new}. Graph-based tests have been studied theoretically by~\cite{henze1999multivariate,bhattacharya2020asymptotic}, among others. In general, these tests are designed to be distribution-free under the null, and do not involve graph-based IPMs; thus the motivation and formulation of these tests is different than our own.

The graph TV IPM is an example of a graph-based learning method~(see \citet{belkin2003laplacian,zhu2003semi} for some early and fundamental references). In graph-based learning the idea is to use a graph built over observed samples as a tool for organizing and analyzing data. Over time it has been shown~\citep{koltchinskii2000random,belkin2007convergence,vonluxburg2008consistency,garciatrillos2016continuum,garciatrillos2016consistency} that many graph-based functionals approximate a continuum notion of regularity, with more recent attention focusing on rates of convergence~\citep{garciatrillos2020error,garciatrillos2020graph,padilla2020adaptive} and minimax optimality~\citep{green2021minimaxa,green2021minimaxb,hu2022voronoigram}. Part of our work can be viewed as establishing some of the same kinds of guarantees for the graph TV test.

Our main theoretical results concern the minimax optimality of our nonparametric test. The minimax perspective on nonparametric testing was developed by Ingster in a series of pioneering papers~\citep{ingster1987minimax,ingster2003nonparametric}, with more recent work extending these results~\citep{ariascastro2018remember,balakrishnan2019hypothesis,balasubramanian2021optimality,kim2022minimax}. However, the assumptions made in these works are different than our own. Typically, the distributions $P$ and $Q$ are assumed to have smooth densities, for example, densities which belong to H\"{o}lder or Sobolev spaces, whereas we will not assume that densities of $P$ and $Q$ are smooth or even continuous. On the other hand, these papers typically study the detection boundary when distributions are separated in an $L^p$-norm, whereas we assume the distributions are separated in the TV IPM. For these testing problems, $\chi^2$-type tests are typically optimal, whereas we will show that a $\chi^2$-type test is suboptimal for the testing problem considered in this work. Our theory thus complements these works. 

% AG 7/29/23: Use some of this language to comment on the illustrative example, and the rest in the summary of results.
% Unlike an RKHS, $\BV(\Rd)$ contains functions with certain kinds of discontinuities, and one consequence is that the TV IPM is not a kernel IPM. Later we will see that that kernel tests are inappropriate for detecting distributions that differ in TV IPM. Notably, this includes $P,Q$ for which the support of $P - Q$ is \emph{spatially localized}, meaning $P = Q$ except on a set of small (Lebesgue) measure. Intuitively, for these alternatives one can do better by ``honing in'' on the area where the ``action is happening'', and then testing for the presence/absence of signal in this area. Our results indicate that the graph TV test is doing something along these lines.  

\subsection{Notation}
We will use $B$ to refer to an upper bound on the densities of $P,Q$. (See Section~\ref{subsec:setup}.) We use $C_1,C_2,\ldots$ ($c_1,c_2,\ldots$) to refer to large (small) constants that may depend only on $B$ and $d$, and let $C,c$ and $N$ denote constants that may depend only on $B$ and $d$ and may change from line to line. For sequences $(a_n),(b_n)$ we use the asymptotic notation $a_n \lesssim b_n$ to mean that there exists $C$ such that $a_n \leq C b_n$ for all $n \in \mathbb{N}$, the notation $a_n \ll b_n$ to mean that for every $c > 0$, $a_n < c b_n$ for all sufficiently large $n \in \mathbb{N}$, and $a_n \asymp b_n$ to mean $a_n \lesssim b_n$ and $b_n \lesssim a_n$.

We write $\mathbb{P}_{P,Q}$ and $\mathbb{E}_{P,Q}$ for probability and expectation when $\mc{X}_{n_1} \overset{i.i.d}{\sim} P$ and  $\mc{Y}_{n_2} \overset{i.i.d}{\sim} Q$. The notation $P_{n_1}(f) := \frac{1}{n_1}\sum_{i = 1}^{n_1} f(X_i)$ will be used for the sample mean of a function $f: \Rd \to \R$  and $P_{n_1}(\theta) := \frac{1}{n_1}\sum_{i = 1}^{n_1} \theta_i$ for a vector $\theta \in \Rn$; likewise for $Q_{n_2}(f)$ and $Q_{n_2}(\theta)$.

\section{Graph Total Variation Test}
\label{sec:graph-tv-ipm}

We begin this section by introducing some relevant notation involving graphs, before moving on to formally define the graph TV test.

Throughout, $G = (\ttV,\ttE)$ will be an unweighted, undirected graph with vertices $\ttV = \mc{Z}_n$ at the combined set of samples. Arbitrarily orient and enumerate the edges $e_1,\ldots,e_{m}$ where $m = |\ttE|$. We define the incidence matrix $D_G \in \Reals^{n \times 2m}$ to have rows $(D_G)_{\ell} = (0,\ldots,1,0,\ldots,0,-1,0,\ldots,0)$ for $1 \leq \ell \leq m$ -- where if $e_{\ell} = (i,j)$ then there is a $1$ in position $i$ and a $-1$ in position $j$ -- and $(D_G)_{\ell + m} = -(D_G)_{\ell}$. The graph TV of $\theta \in \Reals^n$ is
$$
\|D_G\theta\|_1 = \sum_{i,j = 1}^{n} |\theta_i - \theta_j| \cdot \1\big( \{Z_i,Z_j\} \in \ttE\big).
$$
Let $Z_i = X_i$ for $i = 1,\ldots,n$ and $Z_{n_1 + j} = Y_{j}$ for $j = 1,\ldots,n_2$, and set $a = (1/n_1,\ldots,1/n_1,-1/n_2,\ldots,-1/n_2)$ to be the assignment vector denoting whether each $Z_i$ belongs to $\mc{X}_{n_1}$ or $\mc{Y}_{n_2}$. Then the graph TV IPM~\eqref{eqn:graph-tv-ipm} can be written in terms of $D_G$ and $a$:
\begin{equation}
	\label{eqn:graph-tv-ipm-matrix}
	d_{\DTV(G)}(\mc{X}_{n_1},\mc{Y}_{n_2}) = \sup_{\theta: \|D_{G}\theta\|_1 \leq 1} \; a^{\top} \theta.
\end{equation}
If $G$ is a connected graph then the graph TV IPM will be finite, since in this case the null space of $D_G$ is spanned by $\onevec = (1,\ldots,1)$ and by construction $a^{\top} \onevec = 0$. 

The optimization problem in~\eqref{eqn:graph-tv-ipm-matrix} can be reformulated as a linear program:
\begin{equation}
	\label{eqn:lp}
	\begin{aligned}
		\max_{\theta \in \R^n,u \in \R^m} \; \theta^{\top} a, \quad \textrm{subject to} \quad u^{\top} \onevec \leq 1, \quad D_{G}\theta - u \leq 0, \quad -D_{G}\theta - u \leq 0.
	\end{aligned}
\end{equation}
This means we can take advantage of highly-optimized LP solvers to compute the graph TV IPM. Later, in Section~\ref{subsec:computation}, we discuss an alternative approach to computing the statistic which takes more advantage of the graph structure.  

\subsection{$\varepsilon$-neighborhood graph}
We have now described a way to compute the graph TV IPM that is valid for any connected graph $G$. However, the behavior of the statistic and test will strongly depend on the choice of graph $G$. While the appropriate choice of graph depends on the particular application, for our theoretical results we will mostly focus on the \emph{$\varepsilon$-neighborhood graph} $G_{n,\varepsilon}$, which puts an edge between $Z_i, Z_j \in \mc{Z}_n$ if $\|Z_i - Z_j\|_2 \leq \varepsilon$. The resulting $\varepsilon$-neighborhood graph TV is defined for a function $f: \Rd \to \Reals$ as 
\begin{equation}
	\label{eqn:dtv}
\DTV_{n,\varepsilon}(f) := \sum_{i,j = 1}^{n} |f(z_i) - f(z_j)| \cdot \1\{\|z_i - z_j\|_2 \leq \varepsilon \}.
\end{equation}
The test statistic we will analyze is the $\varepsilon$-graph TV IPM $d_{\DTV(G_{n,\varepsilon})}(\mc{X}_{n_1},\mc{Y}_{n_2})$. When the graph is clear from context (and it almost always will be) we will abbreviate this to $d_{\DTV}(\mc{X}_{n_1},\mc{Y}_{n_2})$.

Let us give some insight into the connection between $\DTV_{n,\varepsilon}(f)$ and the continuum TV defined in~\eqref{eqn:tv}. Suppose $\mc{Z}_{n}$ is uniformly spread over an open domain $\Omega$ with compact closure, and $f$ is a smooth function compactly supported in $\Omega$. In this case, it is known that in the continuum limit as $n \to \infty, \varepsilon \to 0$, 
\begin{equation*}
	(\sigma n^2 \varepsilon^{d + 1})^{-1} \DTV_{n,\varepsilon}(f) \overset{P}{\to} \TV(f).
\end{equation*}
(Here $\sigma = \int_{B(0,1)} |x_1| \,dx$ is a constant pre-factor; see for instance~\citet{garciatrillos2016continuum}.) One might expect that under similar conditions, the graph TV IPM, suitably rescaled, will converge to the continuum TV IPM~\eqref{eqn:tv-ipm}, and we present a result of this kind in Section~\ref{sec:continuum-limit}. This is one reason why it might make sense to use the graph TV IPM to detect alternatives which differ, at the level of the population, according to the continuum TV IPM.

%\begin{remark}
%	At a first glance the graph TV IPM appears rather unusual. As opposed to a more traditional IPM, the constraint set depends on the data, and in particular is finite dimensional, albeit growing with $n$. Nevertheless, when $G$ is a suitably constructed neighborhood graph the resulting IPM is a bonafide nonparametric measure of distance. This can be seen quite clearly by looking at the large-sample limit. Namely, if $\mc{Z}_{n}$ is uniformly spread over a compact domain $\Omega$ and $f$ is smooth, then $\DTV_{n,\varepsilon}(f)$ is nearly equal to $\TV(f)$, up to a scaling factor that does not depend on $f$ \agcomment{ref Nicholas}. In this case, under certain conditions
%	$$
%	d_{\DTV(G_{n,\varepsilon})}(P_{n_1},Q_{n_2}) \overset{P}{\to} d_{\BV}(P,Q),
%	$$
%	in the limit as $n \to \infty$ and $\varepsilon \to 0$. We delay a formal statement until Section~\ref{sec:extensions}, as it is a special case of the more general Theorem~\ref{}, which also handles non-uniform $\mc{Z}_n$. For now, the important takeaway is that $d_{\BV}(P,Q)$ is indexed by an infinite dimensional function class, and thus $d_{\DTV(G_{n,\varepsilon})}$ is an estimate of a genuine nonparametric distance between two unknown distributions.
% \end{remark}

\begin{remark}
	\label{rmk:ill-posed}
	One reason to define a test statistic using a graph is that it can be computed by solving a finite-dimensional optimization problem, i.e. the linear program in~\eqref{eqn:lp}. As we have already pointed out, another reason to use a graph-based statistic is that the naive approach of plugging in empirical distributions $P_{n_1}, Q_{n_2}$ to $d_{\BV}$ cannot be used since $d_{\BV}(P_{n_1}, Q_{n_2}) = \infty$. It seems likely that there are other ways of fixing this degeneracy. One idea is to smooth the empirical distributions $P_{n_1}$ and $Q_{n_2}$ prior to applying the TV IPM. We suspect this would result in a test with similar theoretical properties to the graph TV IPM, but which would be hard to compute. In contrast, the graph TV IPM is a computationally tractable procedure.
\end{remark}

\subsection{Calibration by permutation}
\label{subsec:graph-tv-test}
Since the graph TV IPM measures distances between two sets of samples, a two-sample test should reject the null hypothesis if the IPM is sufficiently large. We will calibrate such a test by permutation. Let $S_n$ be the set of permutations over $[n] = \{1,\ldots,n\}$, and for each $\pi \in S_n$ define $\mc{X}_{n_1}^{\pi} = (Z_{\pi(1)},\ldots,Z_{\pi(n_1)})$ and $\mc{Y}_{n_2}^{\pi} = (Z_{\pi(n_1 + 1)},\ldots,Z_{\pi(n)})$. For a level $\alpha \in (0,1)$, the permutation critical value is the $(1 - \alpha)$ quantile of the permutation distribution:
$$
t_{\DTV}(\mc{X}_{n_1},\mc{Y}_{n_2}) = \inf\Bigl\{t: \frac{1}{n!}\sum_{\pi \in S_n} \1\bigl(d_{\DTV}(\mc{X}_{n_1}^{\pi},\mc{Y}_{n_2}^{\pi}) \leq t\bigr) \geq 1 - \alpha  \Bigr\}.
$$
We adopt the convention of taking a hypothesis test to be a measurable function $\varphi: \Reals^n \to \{0,1\}$. The graph TV test (calibrated by permutation) is thus
$$
\varphi_{\DTV}(\mc{X}_{n_1},\mc{Y}_{n_2}) =
\begin{dcases*}
	0, \quad & \textrm{$d_{\DTV}(\mc{X}_{n_1},\mc{Y}_{n_2}) \leq t_{\DTV}(\mc{X}_{n_1},\mc{Y}_{n_2})$} \\
	1, \quad & \textrm{$d_{\DTV}(\mc{X}_{n_1},\mc{Y}_{n_2}) > t_{\DTV}(\mc{X}_{n_1},\mc{Y}_{n_2})$}.
\end{dcases*}
$$
This is a level-$\alpha$ test for equality of distributions $P = Q$, i.e. $\mathbb{P}_{P,P}(\varphi_{\DTV}(\mc{X}_{n_1},\mc{Y}_{n_2}) = 1) \leq \alpha$. By randomizing the decision at the critical value $d_{\DTV}(\mc{X}_{n_1},\mc{Y}_{n_1}) = t_{\DTV}(\mc{X}_{n_1},\mc{Y}_{n_1})$, it can be made to be a size-$\alpha$ test~\citep{lehmann2005testing}, i.e $\mathbb{P}_{P,P}(\varphi_{\DTV}(\mc{X}_{n_1},\mc{Y}_{n_2}) = 1)$ would be exactly $\alpha$. For computational reasons, one typically approximates the permutation critical value by Monte Carlo, i.e. by uniformly sampling $\pi_1,\ldots,\pi_B$ from $S_n$; so long as the identity permutation is included, the resulting test is still level-$\alpha$.

\section{Representation of the TV IPM Witness}
\label{sec:representation}

\iffalse
\begin{itemize}
	\item State the representer theorem: graph TV IPM is an exact convex relaxation of an integer program. State a proposition saying that the same is true for continuum TV IPM. Explain why (or at least say that) this is a useful technical result for deriving upper bounds on the the risk of the graph TV test.
\end{itemize}
\fi

% \ag{Siva mentions that depending on where we want to submit, we might want to put this section before the minimax theory section.}

It is often possible to show that the witness of an IPM must belong to a subset of the original optimization domain $\mc{F}$. Such representations, when they exist, are useful both for computational reasons and because they help us understand ``how'' the IPM is measuring differences between distributions. We now derive a representation result for the witness of the graph TV IPM in terms of binary-valued vectors.

As written in~\eqref{eqn:graph-tv-ipm} the graph TV IPM is the solution to a constrained optimization problem, but it is equivalent to the unconstrained ratio optimization problem
\begin{equation}
	\label{eqn:graph-tv-ipm-ratio}
	\max_{\theta \in \Rn} \; R_{G}(\theta), \quad R_{G}(\theta) :=
	\begin{dcases}
		\frac{a^{\top} \theta}{\|D_G\theta\|_1},& \quad \|D_G\theta\|_1 > 0 \\
		- \infty,& \quad \|D_G\theta\|_1 = 0.
	\end{dcases} 
\end{equation}
The equivalence between~\eqref{eqn:graph-tv-ipm} and~\eqref{eqn:graph-tv-ipm-ratio} is immediate: if $\theta^{\ast}$ achieves the maximum in~\eqref{eqn:graph-tv-ipm-ratio}, then $d_{\DTV(G)}(\mc{X}_{n_1},\mc{Y}_{n_2}) = R_G(\theta^{\ast}) \cdot \|D_G\theta^{\ast}\|_1$. A less obvious result -- though one that will be familiar to readers versed in the theory of linear programming -- is that while the domain in~\eqref{eqn:graph-tv-ipm-ratio} is all of $\Reals^n$, in fact the maximum is achieved by a binary vector:
\begin{equation}
	\label{eqn:representation-graph-tv-ipm}
	\max_{\theta \in \Rn} \; R_G(\theta) = \max_{\theta \in \{0,1\}^n} \; R_G(\theta).
\end{equation}
% \begin{remark}
%	A similar representation exists for the continuum TV IPM~\eqref{eqn:tv-ipm}, showing that it too has a binary-valued witness function. See Proposition~\ref{prop:representation-population-tv-ipm} in the Appendix. This fact plays an important technical role in the proof of Theorem~\ref{thm:graph-tv}.
% \end{remark}

This equivalence follows from known facts about submodular optimization~\citep{hein2011beyond,bach2013learning}, but for completeness we give a proof of~\eqref{eqn:representation-graph-tv-ipm} in Appendix~\ref{sec:pf-representation-graph-tv-ipm}. We now discuss some computational and conceptual implications that follow from this relationship.

\begin{remark}
	A representation analogous to~\eqref{eqn:representation-graph-tv-ipm} also exists for the continuum TV IPM~\eqref{eqn:tv-ipm}, showing that it too has a binary-valued witness function. See Proposition~\ref{prop:representation-population-tv-ipm} in the Appendix. This fact plays an important technical role in the proof of Theorem~\ref{thm:graph-tv}.
\end{remark}

\subsection{Computation via parametric max flow}
\label{subsec:computation}
Equation~\eqref{eqn:representation-graph-tv-ipm} says that the graph TV IPM is an exact convex relaxation of the non-convex problem on the right hand side of~\eqref{eqn:representation-graph-tv-ipm}. Interestingly, this is a case where there are attractive computational approaches to directly solving the non-convex problem. We review one such approach based on \emph{parametric max flow}, a combinatorial approach to solving discrete ratio optimization problems such as the right hand side of~\eqref{eqn:representation-graph-tv-ipm}. Accepted wisdom in fields such as computer vision is that, although the algorithm has poor worst-case complexity, it often works well in practice~\citep{kolmogorov2007applications,wang2016trend}. 

To see how parametric max flow can be used to compute the graph TV IPM, consider
\begin{equation}
	\label{eqn:ratio-optimization-1}
	M(\lambda) := \max_{\theta \in \{0,1\}^n } \; \theta^{\top} a - \lambda\|D_G\theta\|_1, \quad \textrm{for $\lambda \in (0,\infty)$.}
\end{equation}
The function $M(\lambda)$ is non-negative, continuous and piecewise linear, and is strictly positive if and only if $\lambda < \lambda^{\ast}$ where $\lambda^{\ast} = \max_{\theta \in \{0,1\}^n} R_G(\theta)$.  In other words, we have (yet) another representation of the graph TV IPM, as the smallest value of $\lambda$ for which the solution to~\eqref{eqn:ratio-optimization-1} is equal to $0$:
\begin{equation*}
	d_{\DTV(G)}(\mc{X}_{n_1},\mc{Y}_{n_2}) = \inf\Big\{\lambda > 0: \max_{\theta \in \{0,1\}^n } \; \theta^{\top} a - \lambda\|D_G\theta\|_1 = 0\Big\}.
\end{equation*} 
This suggests an iterative strategy for (approximately) computing the graph TV IPM by binary search, suggested (in a broader context) by~\cite{kolmogorov2007applications}:
\begin{enumerate}
	\item Initialize lower and upper bounds $\lambda_{l} = 0, \lambda_{u} = \|D^{\dagger} a\|_{\infty}$, which satisfy $\lambda_{l} \leq \lambda^{\ast} \leq \lambda_u$. 
	\item Compute $M(\lambda)$ at $\lambda = (\lambda_{u} + \lambda_{l})/2$.  
	\item If $M(\lambda) = 0$, set $\lambda_{u} = \lambda$. Otherwise, set $\lambda_{l} = \lambda$. 
	\item Repeat Steps 2 and 3 until convergence. Output $\lambda$.
\end{enumerate}
This is appealing because in Step 2 $M(\lambda)$ can be computed by solving a (standard) max-flow min-cut problem, for which there exist fast algorithms both in theory and practice~\citep{boykov2004experimental}. 

% For instance: given some starting values $\lambda_0 < \lambda^{\ast} < \lambda_1$, we would proceed as follows
% \begin{enumerate}
%   \item Initialize lower and upper bounds $\lambda_{l} = 0, \lambda_{u} = \|D^{\dagger} a\|_{\infty}$, which satisfy $\lambda_{l} \leq \lambda^{\ast} \leq \lambda_u$.
%	\item Compute $M(\lambda)$ at $\lambda = (\lambda_{u} + \lambda_{l})/2$.  
%   \item If $M(\lambda) = 0$, set $\lambda = (\lambda + \lambda_{l})/2$. Otherwise, set $\lambda = (\lambda_u + \lambda)/2$. 
%   \item Repeat Steps 2 and 3 until convergence.
% \end{enumerate}

% This discussion assumes a strategy for effectively choosing the sequence $\lambda_0,\lambda_1,\ldots$. One way to do this is simply to use binary search. Alternatively, one can use Dinkelbach's method~\citep{dinkelbach1967nonlinear} as described in~\citep{kolmogorov2007applications}. For a formal algorithm and more discussion of parametric max flow see~\citep{kolmogorov2007applications}.

\subsection{Interpretations}
\label{subsec:interpretations}
We now discuss two ways to interpret the graph TV IPM, leveraging the equivalences between~\eqref{eqn:graph-tv-ipm},~\eqref{eqn:graph-tv-ipm-ratio} and~\eqref{eqn:representation-graph-tv-ipm}. One interpretation has to do with binary classification, while the other is more graph-theoretic in nature.

\paragraph{Classification.} We first show that the criterion of the graph TV IPM can be interpreted as measuring the accuracy of a binary classifier divided by a measure of the complexity of its decision boundary. Let $\mc{A}:\Rd \to \{0,1\}$ be a binary classifier used to distinguish samples from $\mc{X}_{n_1}$ and samples from $\mc{Y}_{n_2}$. The classifier outputs 1 if it thinks a given sample belongs to $\mc{X}_{n_1}$ and 0 if it thinks the sample belongs to $\mc{Y}_{n_2}$. Then the classification accuracy (above baseline) of $\mc{A}$ is
$$
\mathrm{acc}(\mc{A}) := \frac{1}{2}\Bigl(\frac{1}{n_1} \sum_{i = 1}^{n_1} \1(\mc{A}(X_i) = 1) + \frac{1}{n_2} \sum_{j = 1}^{n_2} \1(\mc{A}(Y_j) = 0)\Bigr) -\frac{1}{2}.
$$
Subtraction by $1/2$ means that expected classification accuracy of a randomized base classifier $\mc{A}_{\mathrm{base}}$ -- which assigns a $1$ to each point with probability $n_1/n$, and a $0$ otherwise -- is $0$. On the other hand, one way to measure the ``complexity'' of $\mc{A}$ is through some data-dependent measure of its decision boundary. For example, one could count the number of pairs $(Z_i,Z_j)$ within $\varepsilon$ of one another for which $\mc{A}(Z_i) \neq \mc{A}(Z_j)$:
$$
\mathrm{plex}(\mc{A}) := \frac{1}{2}\sum_{i,j = 1}^{n} \1\big\{\mc{A}(Z_i) \neq \mc{A}(Z_j),\|Z_i - Z_j\|_2 \leq \varepsilon\big\}
$$
With classification accuracy and complexity thus defined, we have $R_{G_{n,\varepsilon}}(\1_{\mc{A}}) = \mathrm{acc}(\mc{A})/\mathrm{plex}(\mc{A})$, where $\1_{\mc{A}} = (\mc{A}(Z_1),\ldots,\mc{A}(Z_n))$; in other words the ratio is simply the classification accuracy of $\mc{A}$, normalized by its complexity. Therefore,\footnote{In \eqref{eqn:classification-interpretation} the maximum is over all $\mc{A}$ with $\mathrm{plex}(\mc{A}) > 0$.}
\begin{equation}
	\label{eqn:classification-interpretation}
	d_{\DTV}(\mc{X}_{n_1},\mc{Y}_{n_2}) = \max_{\theta \in \{0,1\}^n} R_{G_{n,\varepsilon}}(\theta) = \max_{\mc{A}} \frac{\mathrm{acc}(\mc{A})}{\mathrm{plex}(\mc{A})}.
\end{equation}
Various authors (e.g. \citet{friedman2003multivariate,lopez-paz2017revisiting,kim2021classification} and others) suggest using classification accuracy of specific classifiers as a two-sample test statistic. The equivalences in~\eqref{eqn:classification-interpretation} show that the graph TV test is different: it (implicitly) considers \emph{all} classifiers $\mc{A}$, rejecting the null when there is an $\mc{A}$ that achieves sufficiently high classification accuracy, relative to the complexity of its decision boundary. In this sense, it is similar to the classical learning-theoretic idea of structural risk minimization. We note that there exists a somewhat analogous representation of kernel MMDs~\citep{fukumizu2009kernel}.

\paragraph{Graph clustering.} One can equivalently write the representation in~\eqref{eqn:representation-graph-tv-ipm} in terms of finding a particular kind of graph cluster. For a subset $S \subseteq \mc{Z}_n$, define the graph \emph{cut} and \emph{balance} functionals
\begin{equation}
	\label{eqn:cut-balance}
	\cut_G(S) := \sum_{i,j = 1}^{n} \1\big(Z_i \in S,Z_j \in S^c, \{Z_i,Z_j\} \in \ttE\big), \quad\textrm{and}\quad \mathrm{bal}_G(S) := \frac{1}{2}\Big|\sum_{i = 1}^{n} a_i \1(Z_i \in S)\Big|,
\end{equation}
The graph TV IPM can be viewed as solving a particular kind of balanced cut problem:
\begin{equation}
	\label{eqn:representation-graph-tv-ipm-cut}
	d_{\DTV(G)}(\mc{X}_{n_1},\mc{Y}_{n_1}) = \max_{\theta \in \{0,1\}^n} R_G(\theta) = \max_{S \subset \mc{Z}_n} \frac{\min\{\mathrm{bal}_G(S), \mathrm{bal}_G(S^c)\}}{\cut_G(S)}.
\end{equation}
This resembles (one over) the \emph{Cheeger cut} of $S$~\citep{cheeger1970lower}, where we recall that the Cheeger cut is the solution to
\begin{equation}
	\label{eqn:cheeger}
	 \min_{S} \Phi(S), \quad \Phi(S) := \frac{\cut_{G}(S)}{\min\{|S|,|S^c|\}}.
\end{equation}
The crucial difference between the balanced cut problems~\eqref{eqn:representation-graph-tv-ipm-cut} and~\eqref{eqn:cheeger} is that in the former, the balance term depends on the assignment vector $a$, and encourages picking a set $S$ that mostly belongs to $\mc{X}_{n_1}$ or $\mc{Y}_{n_2}$.  Minimizing the Cheeger cut -- or equivalently, maximizing one over the Cheeger cut -- is a popular technique for graph clustering~\citep{kannan2004clusterings,garciatrillos2016consistency}. At a high level, then, we can view the optimization underlying the graph TV IPM as searching for a cluster of points that primarily belong either to $\mc{X}_{n_1}$ or $\mc{Y}_{n_2}$.

\section{Minimax TV IPM Testing}
\label{sec:optimality}

Like any nonparametric test the graph TV test can only have non-trivial power against alternatives in a few ``directions''~\citep{janssen2000global}, so to get meaningful theoretical results we must place some conditions on the alternative hypotheses. In Section~\ref{subsec:setup} we state these conditions and formalize the hypothesis testing problem under consideration. Our upper and lower bounds on power are given in Sections~\ref{subsec:upper-bounds} and~\ref{subsec:lower-bounds}, and implications for spatially localized alternatives are explored in Sections~\ref{subsec:spatially-localized-alternatives-graphtvipm} and~\ref{subsec:spatially-localized-alternatives-chisquared}.

\subsection{Problem setup and background}
\label{subsec:setup}

\paragraph{Regularity conditions.}
For all of our theoretical results, we will assume that $P$ and $Q$ satisfy the following regularity conditions.
\begin{enumerate}[label=(A\arabic*)]
	\item 
	\label{asmp:domain}
	Both $P$ and $Q$ are supported on $\Omega = (0,1)^d$.
	\item
	\label{asmp:density}
	Both $P$ and $Q$ are absolutely continuous with respect to Lebesgue measure, with densities $p$ and $q$ respectively.
	\item
	\label{asmp:bounded-density}
	The difference in densities $p - q$ is bounded from above, and the mixture of densities $\mu := \frac{n_1}{n}p + \frac{n_2}{n}q$ is bounded from below: there exists $B \in [2,\infty)$ such that
	\begin{equation}
		\label{eqn:bounded-density}
		\big|p(x) - q(x)\big| \leq B, \quad \textrm{and} \quad \frac{1}{B} \leq \mu(x) \leq B, \quad \textrm{for all $x \in \Omega$.} 
	\end{equation} 
\end{enumerate}

Hereafter we let $\mc{P}^{\infty}(d,B)$ be the collection of all $(P,Q)$ that satisfy~\ref{asmp:domain}-\ref{asmp:bounded-density}. We will treat $B$ as a potentially large but fixed (in $n_1,n_2$) constant, and allow constants $c, c_1,c_2,\ldots,$ and $C,C_1,C_2,\ldots$ to depend on $B$. We abbreviate $\mc{P}^{\infty}(d) = \mc{P}^{\infty}(d,B)$ for convenience.

\paragraph{Detection boundary and minimax optimality.}
Formally speaking, the hypothesis testing problem we consider is to distinguish
\begin{equation}
	\label{eqn:dual-tv-testing}
	H_0: P = Q, \quad \textrm{from} \quad H_1: (P,Q) \in \mc{P}^{\infty}(d), ~~ d_{\BV}(P,Q) \geq \rho.
\end{equation}
We are interested in the {detection boundary} of this testing problem, meaning the minimum value of $\rho$ required for some test to have power of at least (say) $1/2$, uniformly over all alternatives. Mathematically, letting 
$$
\mathrm{Risk}_{n_1,n_2}(\varphi,\rho) = \sup\Bigl\{\E_{P,Q}[1 - \varphi]: (P,Q) \in \mc{P}^{\infty}(d),  d_{\BV}(P,Q) \geq \rho\Bigr\}
$$
be the maximum probability of type II error over the alternatives in~\eqref{eqn:dual-tv-testing}, the \emph{detection boundary} is 
$$
\rho_{n_1,n_2}^{\ast} := \inf\Bigl\{\rho: \inf_{\varphi} \mathrm{Risk}_{n_1,n_2}(\varphi,\rho) \leq \frac{1}{2}\Bigr\},
$$
where the infimum is over all tests which have size at most $1/4$ for any $P = Q$. 

Although all of our results are non-asymptotic, our focus will be on understanding the minimax rate, meaning the rate at which $\rho_{n_1,n_2}^{\ast}$ converges to $0$ as $n_1,n_2 \to \infty$. We will refer to any test that has risk less than $1/2$ for some sequence $\rho_{n_1,n_2} \lesssim \rho_{n_1,n_2}^{\ast}$ as a rate-optimal test. (We note that the constants $1/2$ and $1/4$ in the definition of $\rho_{n_1,n_2}^{\ast}$ are arbitrary, and could be replaced by any fixed $\alpha,\beta \in (0,1)$ for which $\beta < 1 - \alpha$ without changing the rate of convergence of the detection boundary.) 

\begin{remark}
	Some assumptions on $P$ and $Q$ are necessary in order for~\eqref{eqn:dual-tv-testing} to be a well-posed problem. For instance if $P$ or $Q$ has an atom then the TV IPM between them will be infinite unless they place the same weight on the atom.\footnote{For a constructive argument, take $P$ to have an atom at some $x_0 \in \Rd$, and suppose $Q(\{x_0\}) < P(\{x_0\})$. Consider $f_r(x) = \frac{\1\{x \in B(x_0,r)\}}{r^{d - 1}}$. The total variation of $f_r$ is fixed in $r$, but taking $r \to 0$ will blow the criterion of the TV IPM up to $\infty$.} 
	The specific condition $(P,Q) \in \mc{P}^{\infty}(d)$ is sufficient to ensure that $d_{\BV}(P,Q)$ is finite and that the detection boundary converges to $0$ as $n_1,n_2 \to \infty$, but it may not be necessary. For example, note that the TV IPM is simply the dual norm of total variation and so is sensibly defined whenever the signed measure $P - Q$ belong to the space that is dual to $\BV(\Rd)$: this is a weaker condition than $(P,Q) \in \mc{P}^{\infty}(d)$~\citep{meyers1977integral,phuc2015characterizations}. Extending our theory to hold under the minimum possible assumptions on $(P,Q)$ would be an interesting direction for follow up work.
\end{remark}

\subsection{Risk of graph TV test}
\label{subsec:upper-bounds}
\iffalse
\begin{itemize}
	\item Present upper bound on the risk of the graph TV test. Explain that, along with the metric property of the TV IPM, this implies the test is consistent against fixed alternatives.
\end{itemize}
\fi

We now state our first major result: an upper bound on the risk of the $\varepsilon$-graph TV test when computed with graph radius $
\varepsilon_n := C_1(\frac{\log n}{n})^{1/d}$, where $C_1 := (24B)^{1/d}2\sqrt{d}$.

\begin{theorem}
	\label{thm:graph-tv}
	Consider the hypothesis testing problem in~\eqref{eqn:dual-tv-testing}, under assumptions~\ref{asmp:domain}-\ref{asmp:bounded-density}. For any $\alpha \in (0,1)$, the $\varepsilon_n$-graph TV test is size-$\alpha$: $\mathbb{E}_{P,P}[\varphi_{\DTV}] \leq \alpha.$ If furthermore $\alpha,\beta \geq 1/n$, then there is a constant $C$ such that if
	\begin{equation}
		\label{eqn:detection-boundary-ub}
		d_{\BV}(P,Q) \geq \frac{C}{\beta}\biggl(\frac{\log \min(n_1,n_2)}{\min(n_1,n_2)}\biggr)^{1/d} \cdot
		\begin{dcases}
			\sqrt{\frac{\log \min(n_1,n_2)}{\beta}}, & \quad \textrm{when $d = 2$} \\
			1, & \quad \textrm{when $d \geq 3$,}
		\end{dcases}
	\end{equation}
	then the $\varepsilon_n$-graph TV test has power of at least $1 - \beta$: $\mathbb{E}_{P,Q}[\varphi_{\DTV}] \geq 1 - \beta$.
\end{theorem}
At this point a few remarks are in order.

\begin{remark}
	Although they are not strictly necessary, the conditions that $\alpha, \beta \geq 1/n$ significantly ease both the presentation and derivation of Theorem~\ref{thm:graph-tv}. These conditions permit $\alpha \leq 1/4$ when $n \geq 4$ and $1 - \beta \geq 1/2$ when $n \geq 2$. Therefore, Theorem~\ref{thm:graph-tv} implies an upper bound on the detection boundary $\rho_{n_1,n_2}^{\ast}$.
\end{remark}

\begin{remark}
	It is not hard to show that $d_{\BV}(P,Q)$ is a metric over $\mc{P}^{\infty}(d)$, meaning $d_{\BV}(P,Q) = 0 \Longleftrightarrow P = Q$.  Thus, for any fixed pair of distributions $(P,Q) \in \mc{P}^{\infty}(d)$, Theorem~\ref{thm:graph-tv} implies that the graph TV test has asymptotic power tending to $1$ as $n_1,n_2 \to \infty$. 
\end{remark} 

\begin{remark}
	The choice of neighborhood graph radius $\varepsilon = \varepsilon_n \asymp (\log n/n)^{1/d}$ is (up to constants) the smallest choice of $\varepsilon$ that ensures that with high probability the graph $G_{n,\varepsilon}$ is connected. Our theory can handle larger $\varepsilon \gg \varepsilon_n$, but the resulting test will only be powerful when $d_{\BV}(P,Q)$ is substantially larger than in~\eqref{eqn:detection-boundary-ub}. From the computational perspective, we would like to choose $\varepsilon$ to be as small as possible, since computing the graph TV IPM takes longer with denser graphs.
\end{remark}

\begin{remark}
	The result of Theorem~\ref{thm:graph-tv} applies to a test that is calibrated by permutation, as in~\citet{albert2015tests,kim2022minimax}. This is in contrast to the more common approach to minimax analysis of nonparametric two-sample testing, which is to consider a hypothesis test with rejection region determined by a concentration inequality. In practice, the latter kinds of tests are typically quite conservative, whereas calibration by permutation leads to (nearly) exact type I error control.
\end{remark}

Let us very briefly describe the way we prove Theorem~\ref{thm:graph-tv}. As the $\varepsilon_n$-graph TV test is calibrated by permutation the upper bound on its size is a standard fact~\citep{lehmann2005testing}. To lower bound the power, we show that there exists a threshold $t_{n_1,n_2}$ that does not depend on the data, such that under the conditions of the theorem both
\begin{equation}
	\label{pf:graph-tv-0}
	\mathbb{P}_{P,Q}\Bigl(t_{\DTV}(\mc{X}_{n_1},\mc{Y}_{n_2}) \geq t_{n_1,n_2}\Bigr) \leq \frac{\beta}{2}, \quad \textrm{and} \quad \mathbb{P}_{P,Q}\Bigl(d_{\DTV}(\mc{X}_{n_1},\mc{Y}_{n_2}) < t_{n_1,n_2}\Bigr) \leq \frac{\beta}{2}.
\end{equation}
Taken together, the inequalities in~\eqref{pf:graph-tv-0} imply the stated upper bound on the risk of the graph TV test. The proof of~\eqref{pf:graph-tv-0}, and thus Theorem~\ref{thm:graph-tv}, can be found in the appendix, as for all the rest of the results of this paper.

\subsection{Lower bound on detection boundary}
\label{subsec:lower-bounds}

% Outline
\iffalse
\begin{itemize}
	\item Present lower bounds on the risk of any test. Be explicit -- the lower bounds match the upper bounds, so the critical radius converges at a given rate, and the graph TV test is rate-optimal.
\end{itemize}
\fi
% End outline

% Text
Theorem~\ref{thm:graph-tv} gives an upper bound on the detection boundary $\rho_{n_1,n_2}^{\ast}$. The following result is a lower bound.
\begin{theorem}
	\label{thm:lower-bound}
	Consider the hypothesis testing problem in~\eqref{eqn:dual-tv-testing}, under assumptions~\ref{asmp:domain}-\ref{asmp:bounded-density}. For any $0 < \beta < 1 - \alpha$, there are constants $N = N(\alpha,\beta)$ and $c$ such that the following holds: if $\min(n_1,n_2) \geq N$, then for any level-$\alpha$ test $\varphi$ there exist distributions $P,Q$ satisfying
	\begin{equation*}
		d_{\BV}(P,Q) \geq c \bigg(\frac{\log \min(n_1,n_2)}{\min(n_1,n_2)}\bigg)^{1/d},
	\end{equation*}
	but for which $\varphi$ has power at most $1 - \beta$: $\mathbb{E}_{P,Q}[\varphi] \leq 1 - \beta$. 
\end{theorem}
Combined, Theorems~\ref{thm:graph-tv} and~\ref{thm:lower-bound} show that the detection boundary converges at rate
\begin{equation}
	\label{eqn:detection-boundary}
	\rho_{n_1,n_2}^{\ast} \asymp \bigg(\frac{\log \min(n_1,n_2)}{\min(n_1,n_2)}\bigg)^{1/d},
\end{equation}
and further imply that the graph TV test is rate-optimal for detecting differences in $d_{\BV}$. (Except when $d = 2$, where the upper and lower bounds differ by a $\sqrt{\log \min(n_1,n_2)}$ factor.)

\begin{remark}
	Typical work on nonparametric testing -- see e.g. ~\citet{ingster1987minimax,lepski1999minimax,ariascastro2018remember} -- focuses on characterizing the detection boundary of hypothesis testing problems where the metric between distributions is an $L^p$ distance between densities, and assumes regularity conditions such as the densities $p,q$ being H\"{o}lder smooth or belong to a certain Sobolev or Besov space. The hypothesis testing problem in~\eqref{eqn:dual-tv-testing} is different: the assumption $(P,Q) \in \mc{P}^{\infty}(d)$ does not restrict the densities $p,q$ to be differentiable or even continuous, and the distance is measured using the TV IPM which is a strictly weaker metric than the $L^{\infty}$ distance between densities. Likewise, the minimax rate of convergence we obtain is different than the ``standard'' rate of convergence in nonparametric testing problems. For instance, assuming balanced sample sizes $n_1 = n_2 = n/2$ for convenience, if $P,Q$ have densities $p,q$ that are Lipschitz continuous, and the distance between $P,Q$ is measured by the $L^2$ norm of $p - q$, then the detection boundary converges at rate $(1/n)^{2/(4 + d)}$~\citep{ariascastro2018remember}. This is a slower rate of convergence than $(\log n/n)^{1/d}$ when $d < 4$, and a faster rate of convergence when $d \geq 4$.

\end{remark}

\subsection{Spatially localized alternatives and the graph TV test}
\label{subsec:spatially-localized-alternatives-graphtvipm}
Let us examine the implications of our upper and lower bounds for the type of problem motivated by our illustrative example: detecting spatially localized departures from the null. 

We begin by constructing a collection of parametric families $\mc{P}_{\eta}$ of varying degrees of spatial localization. Let $\Xi_{\eta}$ be a partition of $\Omega = (0,1)^d$ into cubes $\Delta \in \Xi_{\eta}$ of side-length $\eta$: that is, for $N = \frac{1}{\eta}$ and each $j \in [N]^d$, let $\Delta_j = \frac{j}{N} + (-\eta,0]^d.$ Further subdivide each cube $\Delta_j$ into bottom-left and upper-right halves $\Delta_j^L$ and $\Delta_j^R$, where $ \Delta_j^L = \frac{j}{N} + (-\eta,\frac{\eta}{2}]^d$ and $\Delta_j^R = \frac{j}{N} + (-\frac{\eta}{2},0]^d$. Then put $\phi_{\Delta_j}(x) = \1(x \in \Delta_j^L) - \1(x \in \Delta_j^R)$ and $p_{\Delta_j}(x) = 1 + \frac{1}{2}\phi_{\Delta_j}(x)$, let $P_{\Delta_j}$ be the distribution with density $p_{\Delta_j}$, and set $\mc{P}_{\eta} = \{(P_0,P_{\Delta_j}): \Delta_j \in \Xi_{\eta}\}$, where $P_0 = \Unif(\Omega)$. Though this construction is a bit technical, there are two important things to keep in mind. First, for all $\eta \in (0,1)$, $\mc{P}_{\eta} \subset \mc{P}^{\infty}(d)$. Second, each family $\mc{P}_{\eta}$ represents a collection of alternatives that differ in a region with area on the order of $\eta^d$. Thus, the smaller $\eta$, the more spatially localized the departures from the null in $\mc{P}_{\eta}$.

Consider the two-sample testing problem of distinguishing
\begin{equation}
	\label{eqn:spatially-localized-testing}
	H_0: P = Q, \quad \textrm{versus} \quad H_1: (P,Q) \in \mc{P}_{\eta}.
\end{equation}
As $\eta \to 0$, this hypothesis testing problem becomes more challenging, and we are interested in the smallest value of $\eta$ for which some test has risk of at most $\frac{1}{2}$. It turns out that this can be determined (up to constant factors) from the theory we have already developed.
\begin{corollary}
	\label{cor:upper-bound-spatially-localized}
	Consider the hypothesis testing problem in~\eqref{eqn:spatially-localized-testing}. For any $\alpha \in (0,1)$, the $\varepsilon_n$-graph TV test is size-$\alpha$: $\mathbb{E}_{P,P}[\varphi_{\DTV}] \leq \alpha$. For any $\alpha \in (0,1)$, the $\varepsilon_n$-graph TV test is size-$\alpha$: $\mathbb{E}_{P,P}[\varphi_{\DTV}] \leq \alpha.$ If furthermore $\alpha,\beta \geq 1/n$, then there is a constant $C$ such that if
	\begin{equation}
		\label{eqn:spatial-localization-ub}
		\eta \geq C \Big(\frac{\log \min(n_1,n_2)}{\min(n_1,n_2)}\Big)^{1/d} \cdot
		\begin{dcases}
			\sqrt{\log \min(n_1,n_2)}, & \quad \textrm{when $d = 2$} \\
			1, & \quad \textrm{when $d \geq 3$,}
		\end{dcases}
	\end{equation}
	then the $\varepsilon_n$-graph TV test has power of at least $1/2$: $\mathbb{E}_{P,Q}[\varphi_{\DTV}] \geq \frac{1}{2}$.
\end{corollary}
\begin{corollary}
	\label{cor:lower-bound-spatially-localized}
	Consider the hypothesis testing problem in~\eqref{eqn:spatially-localized-testing}, and let $\varphi$ be any level-$\alpha$ test. For any $0 < \beta < 1 - \alpha$, there are constants $N = N(\alpha,\beta)$ and $c$ such that the following holds: if $\min(n_1,n_2) \geq N$, then for any level-$\alpha$ test $\varphi$ there exist distributions $(P,Q) \in \mc{P}_{\eta}$ for some 
	\begin{equation}
	\label{eqn:spatial-localization-lb}
		\eta \geq c \bigg(\frac{\log \min(n_1,n_2)}{\min(n_1,n_2)}\bigg)^{1/d},
	\end{equation}
	for which $\varphi$ has power at most $1 - \beta$: $\mathbb{E}_{P,Q}[\varphi] \leq 1 - \beta$. 
\end{corollary}
Corollaries~\ref{cor:upper-bound-spatially-localized} and~\ref{cor:lower-bound-spatially-localized} imply that the graph TV test is rate-optimal for the hypothesis testing problem in~\eqref{eqn:spatially-localized-testing}, without needing knowledge of $\eta$. (Again, up to a $\sqrt{\log \min(n_1,n_2)}$ factor when $d = 2$.) The corollaries follow directly from Theorems~\ref{thm:graph-tv} and~\ref{thm:lower-bound}, and the upper and lower bounds
\begin{equation}
	\label{eqn:spatially-localized-tvipm}
	c \eta \leq d_{\BV}(P_{\Delta},P_0) \leq C \eta, \quad \textrm{for each $(P_0,P_{\Delta}) \in \mc{P}_{\eta}$}.
\end{equation}
Specifically, keeping in mind that $\mc{P}_{\eta} \subset \mc{P}^{\infty}(d)$, the upper bound on $\Risk_{n_1,n_2}(\varphi_{\DTV},\mc{P}_{\eta})$ in Corollary~\ref{cor:lower-bound-spatially-localized} follows from Theorem~\ref{thm:graph-tv} and the lower bound in~\eqref{eqn:spatially-localized-tvipm}. On the other hand, the families $\mc{P}_{\eta}$ are exactly those used to construct the ``hard'' distributions in the proof of Theorem~\ref{thm:lower-bound}, meaning the conclusions of Theorem~\ref{thm:lower-bound} are unchanged if we only consider $(P,Q) \in \mc{P}_{\eta}$. Using this observation along with the upper bound in~\eqref{eqn:spatially-localized-tvipm} leads to Corollary~\ref{cor:lower-bound-spatially-localized}. Finally, the lower bound in~\eqref{eqn:spatially-localized-tvipm} is a calculation given in the proof of Theorem~\ref{thm:lower-bound}, whereas the upper bound follows from~\eqref{eqn:tv-ipm-finite}.

\subsection{Spatially localized alternatives and a $\chi^2$-type test}
\label{subsec:spatially-localized-alternatives-chisquared}
% Outline
\iffalse
\begin{itemize}
	\item Define a class of alternatives parameterized by their degree of spatial localization. Remark that there is general and ongoing interest in spatially localized alternatives (Friedman,Johnstone,Kim). Allude to scan statistics, postponing a deeper discussion later.
	\item Present results on the fundamental limitations of the $L^2$-histogram test.
	\item Read off from our previous results cases where the graph TV test has non-trivial risk whereas the kernel test does not.
\end{itemize}
\fi
% End outline

Now we consider the ability of a $\chi^2$-type test to detect alternatives in $\mc{P}_{\eta}$. Specifically we consider a test based on the following $\chi^2$-type statistic, 
\begin{equation}
\label{eqn:chi-squared}
\mc{K}_{\varepsilon}(\mc{X}_{n_1},\mc{X}_{n_2}) := \sum_{\Delta \in \Xi_\varepsilon}\Big(n_1P_{n_1}(\Delta) - n_2Q_{n_2}(\Delta)\Big)^2.
\end{equation}
Here $\varepsilon > 0$ is a tuning parameter that determines the width of the cubes $\Delta \in \Xi_\varepsilon$. In words, the statistic in~\eqref{eqn:chi-squared} is computed by partitioning the domain into bins, taking the squared difference between the number of points $X_1,\ldots,X_{n_1}$ and $Y_{1},\ldots,Y_{n_2}$ in each bin, and adding across bins. This is not exactly the same as Pearson's $\chi^2$ test as it does not normalize the counts in each bin. Nevertheless, when the number of samples $n_1 = n_2$ is equal and the distributions $P,Q$ have densities bounded away from $0$ and $\infty$, the operating characteristics of a test based on~\eqref{eqn:chi-squared} should be similar to the $\chi^2$ test, and in many cases the test will be quite powerful. For instance,~\citet{ariascastro2018remember} analyze a test based on~\eqref{eqn:chi-squared} and show that it is optimal for detecting H\"{o}lder smooth alternatives.

A test $\varphi_{\csq}$ based on~\eqref{eqn:chi-squared} will reject the null hypothesis if $\mc{K}_{\varepsilon}(\mc{X}_{n_1},\mc{Y}_{n_2})$ is greater than some threshold $t_{\alpha}$ chosen to control type I error. We analyze the risk of $\varphi_{\csq}$ in a variant of our two-sample framework in which the sample sizes are independent and identically distributed Poisson random variables: that is,
\begin{equation}
\label{eqn:poissonized}
\begin{aligned}
n_1,n_2 & \overset{ind}{\sim} \mathrm{Pois}(n), \quad X_1,\ldots,X_{n_1} \overset{\mathrm{i.i.d}}{\sim} P, \quad Y_1,\ldots,Y_{n_2} \overset{\mathrm{i.i.d}}{\sim} Q.
\end{aligned}
\end{equation}
We assume independent Poisson sample sizes because it makes the analysis much easier. In the following theorem, $z^{\alpha}$ denotes the $(1 - \alpha)$ quantile of the standard Normal distribution.
\begin{theorem}
	\label{thm:lower-bound-chi-squared}
	Consider the Poisson observation model~\eqref{eqn:poissonized}. Suppose $t_{\alpha}$ is chosen such that $\mathbb{E}_{P_0,P_0}[\varphi_{\csq}]  \leq \alpha$ for some $\alpha \in (0,1)$. There exist constants $c,N$ such that if $n \geq N$ and 
	\begin{equation}
	\label{eqn:lower-bound-chi-squared}
	\eta \leq c n^{-2/3d},
	\end{equation}
	then there exist alternatives $(P,Q) \in \mc{P}_{\eta}$ against which the $\varepsilon$-chi-squared test has power of at most $2/(z^{\alpha})^2$: $\mathbb{E}_{P,Q}[\varphi] \leq 2/(z^{\alpha})^2$. 
\end{theorem}

Theorem~\ref{thm:lower-bound-chi-squared} shows that for alternatives in $\mc{P}_{\eta}$ the chi-squared test fails to have power unless the parameter $\eta$ controlling spatial localization is much larger than $(\log n/n)^{1/d}$. This also means that the $\chi^2$-type test is suboptimal for detecting distributions separated in the TV IPM. Specifically, using the fact that $\mc{P}_{\eta} \subset \mc{P}^{\infty}(d)$ along with the lower bound in~\eqref{eqn:spatially-localized-tvipm}, we conclude that if $\alpha < 0.01$ (say) then $\varphi_{\csq}$ will not have power of at least $1/2$ over all alternatives in~\eqref{eqn:dual-tv-testing} unless $\rho_{n_1,n_2} \gtrsim n^{-2/3d}$. The bottom line is that the graph TV test has better worst-case power than the $\chi^2$-type test for detecting spatially localized alternatives specifically, and distributions separated in the TV IPM more generally.

Furthermore, we note that the $\chi^2$-type statistic in~\eqref{eqn:chi-squared} is (up to a normalizing constant) actually an empirical kernel MMD, with kernel 
\begin{equation*}
	k_{\varepsilon}(x,y) = \sum_{\Delta \in \Xi_\varepsilon} \1(x \in \Delta, y \in \Delta).
\end{equation*}
Indeed, we believe the conclusion of Theorem~\ref{thm:lower-bound-chi-squared} should hold for kernel MMDs using other ``nice'' kernels such as the Gaussian kernel, though we do not pursue the details further. % Thus, the results of Theorems~\ref{thm:graph-tv}-\ref{thm:lower-bound-chi-squared} suggest that for spatially localized departures from the null, the graph TV IPM may have better power than typical kernel MMD tests.

At this point a few remarks are in order.
\begin{remark}
	Our results should not be interpreted as saying that the graph TV test is generally superior to a kernel test. It is a fact of life in nonparametric testing that no test can be universally optimal; indeed, designing a test with \emph{complementary} operating characteristics to the kernel MMD is one of our basic motivations. In particular, we think that the kernel test is likely to be superior to the graph TV test when alternatives are more globally supported; although empirically the story appears more nuanced, see Section~\ref{sec:experiments}.
\end{remark}
\begin{remark}
	 There have been several proposals for improving the power of kernel tests by carefully tuning the bandwidth of the kernel in a data-dependent manner. However, Theorem~\ref{thm:lower-bound-chi-squared} implies that the $\chi^2$-type test $\varphi_{\csq}$ is suboptimal for detecting spatially localized alternatives \emph{no matter how} the binwidth $\varepsilon$ is chosen. On the other hand, it is possible that other alterations to kernel tests may be more effective. For example, the bandwidth of the kernel could be chosen in a locally adaptive manner as proposed in~\citet{lepski1999minimax} (albeit for an entirely different testing problem). Their modification seems more appropriate for detecting spatially localized alternatives.
\end{remark}

Finally, the limitations of $\chi^2$ statistics for detecting sparse alternatives are well-understood in Normal means~\citep{donoho2004higher} and regression testing~\citep{ariascastro2011global} problems. Relatedly, there are two-sample tests besides the graph TV IPM, not based on kernels or IPMs, which we would expect to successfully detect spatially localized alternatives. For instance, we believe the \emph{max test} -- which replaces the sum of squared differences in~\eqref{eqn:chi-squared} by the maximum --  should be optimal for the testing problem~\eqref{eqn:spatially-localized-testing}, at least for an appropriate choice of $\varepsilon$. However we see several advantages of the graph TV test over the max test: the graph TV test is also optimal for the more general testing problem~\eqref{eqn:dual-tv-testing}, and it can be applied more generally to problems where binning the domain is either computationally expensive, or impossible (because the domain is unknown).

% \ag{Any comment about whether this could compete with graph TV test?}

\section{Continuum limit of Graph TV IPM}
\label{sec:continuum-limit}

In this section we examine the large sample limit of the graph TV IPM. In general, the statistic will not converge to $d_{\BV}(P,Q)$, but rather to a weighted TV IPM, where the weight function depends on how the graph is formed. % Theorem~\ref{thm:graph-ks-consistency}, below, shows the correct weighted continuum limit for the $\varepsilon$-neighborhood graph. % We discuss on the one hand ways to eliminate this weighting and on the other hand situations in which it might be desirable. 
Let $\omega$ be a positive weight function supported on an open set $\Omega \subset \Rd$. Then \emph{$\omega$-weighted TV} is defined as 
\begin{equation}
	\label{eqn:weighted-tv}
	\mathrm{TV}(f; \omega) := \sup\Big\{ \int_{\Omega} f \cdot \mathrm{div}(\phi): \phi \in C_c^1(\Omega;\Rd), \|\phi(x)\|_{2} \leq \omega(x)~\textrm{for all $x \in \Omega$} \Big\}. 
\end{equation}
The corresponding $\omega$-weighted TV IPM is
\begin{equation*}
	d_{\BV}(P,Q; \omega) := \sup_{\TV(f;\omega) \leq 1} \mathbb{E}_{P}[f(X)] - \mathbb{E}_{Q}[f(Y)].
\end{equation*}
Theorem~\ref{thm:graph-ks-consistency}, below, shows that the large sample limit of the $\varepsilon$-neighborhood graph TV IPM -- as $n_1,n_2 \to \infty$ and $\frac{n_1}{n} \to \lambda \in (0,1)$ -- is a particular weighted TV IPM involving the square of the limiting mixture density, which is
$$
\omega_{P,Q}(x) := \lambda p(x) + (1 - \lambda) q(x).
$$ 
This will be true under appropriate regularity conditions, and for graph radius $\varepsilon \to 0$ satisfying
\begin{equation}
	\label{eqn:graph-radius-continuum-limit}
	\varepsilon \geq C_2\Big(\frac{(\log n)^{q_d}}{n}\Big)^{1/d},
\end{equation}
where $q_d = 3/2$ for $d = 2$ and $q_d = 1$ for $d \geq 3$. In what follows, recall that $\sigma = \int_{B(0,1)} |x_1| \,dx$. 
\begin{theorem}
	\label{thm:graph-ks-consistency}
	Suppose $(P,Q) \in \mc{P}^{\infty}(d)$ have continuous densities $p,q$. As $n_1,n_2 \to \infty$, $n_1/n \to \lambda \in (0,1)$, if  $\varepsilon$-neighborhood graph TV IPM is computed with graph radius $\varepsilon \to 0$ additionally satisfying~\eqref{eqn:graph-radius-continuum-limit}, then with probability one,
	\begin{equation}
		\label{eqn:graph-ks-consistency}
		\sigma n^2 \varepsilon^{d + 1} d_{\DTV}(\mc{X}_{n_1},\mc{Y}_{n_2}) \to d_{\BV}(P,Q; \omega_{P,Q}^2).
	\end{equation}
\end{theorem}
The proof of Theorem~\ref{thm:graph-ks-consistency} relies on establishing a variational mode of convergence known as \emph{$\Gamma$-convergence}. $\Gamma$-convergence has been previously applied to establish asymptotic convergence of various graph-based functionals~\citep{garciatrillos2016consistency,garciatrillos2018variational}, and we apply and extend some of these results to prove~\eqref{eqn:graph-ks-consistency}. Compared to the theory of Section~\ref{sec:optimality}, Theorem~\ref{thm:graph-ks-consistency} is more precise as it identifies the exact asymptotic limit of $d_{\DTV}(\mc{X}_{n_1},\mc{Y}_{n_2})$. On the other hand, $\Gamma$-convergence has no non-asymptotic implications, and so Theorem~\ref{thm:graph-ks-consistency} says nothing about rates of convergence or minimax risk.

The discrepancy $d_{\BV}(P,Q; \omega_{P,Q}^2)$ is unusual in that the distributions $P,Q$ determine the constraint set underlying the IPM. Compared to the usual definition of total variation, the constraint $\TV(f;\omega_{P,Q}^2) \leq 1$ requires functions $f$ to be smoother in high-density regions, but allows them to be rougher in low-density regions. Additionally, since the test functions $\phi$ used in the definition of~\eqref{eqn:weighted-tv} must be compactly supported in $\Omega$, the constraint allows functions $f$ to be rougher at the boundary $\partial \Omega$. In some cases it may be quite useful to have an IPM where the constraint depends on the distributions. For example, suppose $P,Q$ are concentrated around or on a lower-dimensional subspace of $\Rd$, e.g. a manifold. In this case, it seems more natural to measure smoothness on this lower-dimensional support, rather than measuring smoothness over all of $\Rd$ as in the usual definition of TV.

\begin{remark}
	For certain problems there may be reason to prefer a differently-weighted TV IPM $d_{\BV}(P,Q;\omega)$, or even the unweighted TV IPM $d_{\BV}(P,Q)$, over $d_{\BV}(P,Q;\omega_{P,Q}^2)$. We know of two ways to determine the influence of $P,Q$ in the asymptotic weighting. The first is to explicitly reweight the graph TV functional; this is essentially the recommendation of~\citet{coifman2006diffusion}, although they consider a different graph operator. The other approach is to use a different graph. For example, in previous work, two of us (the authors) considered a graph derived from a Voronoi tesselation of the data, and showed that the resulting Voronoi graph TV converges to unweighted TV in the continuum limit~\citep{hu2022voronoigram}. An advantage of this approach, as opposed to explicitly re-weighting the edges in the graph, is that it will also mitigate the boundary effects.
\end{remark}

\section{Numerical Experiments}
\label{sec:experiments}

\subsection{Graph TV test versus the chi-squared test}
\label{subsec:simulations}
We illustrate the relevance of our theory in finite samples by comparing the power of the graph TV test to the chi-squared test in a small simulation study. In each simulation, $n_1 = n_2 = 100,000$ samples are drawn from an alternative $(P_0,P_{s,\Delta}) \in \mc{P}_{s,\eta}$, where $P_{s,\Delta}$ has density $p_{s,\Delta} = 1 + \frac{s}{2}\phi_{\Delta}(x)$, and $\phi_{\Delta}$ is defined as in Section~\ref{subsec:spatially-localized-alternatives-graphtvipm}. The parameters $\eta$ and $s$ control the spatial localization and signal strength, respectively. The point of this experiment is to see, at various levels of spatial localization, how strong the signal strength needs to be in order for the graph TV and chi-squared tests to have high power.

% In each case we draw equal numbers of samples $\mc{X}_{n_1} \overset{i.i.d}{\sim} P_0$ and $\mc{Y}_{n_2} \overset{i.i.d}{\sim} P_{\Delta}$. We measure performance by area under the ROC curve, which averages power across many different levels of type I error. 

For these experiments the graph TV test does \emph{not} use the $\varepsilon$-neighborhood graph. Instead the graph TV IPM is computed by binning the domain and forming a lattice graph over the bins. (For more details see Section~\ref{sec:numerical-experiments-additional-details}).  This is done computational reasons: with $n = 200,000$ samples, the lattice graph (for a reasonable choice of binwidth) will typically have many fewer edges than the $\varepsilon$-neighborhood graph. Of course, in this case the alternatives are also defined in terms of bins and there so is a statistical advantage to constructing the graph by using bins. But the chi-squared test shares this advantage and so we feel the comparison is a fair one.

Figure~\ref{fig:conceptual-example} compares the area under the ROC curve for the ``binned graph'' TV test, with binwidth $\varepsilon = 0.02$, to the chi-squared test with binwidths $\varepsilon = 0.02, 0.1, 0.5$, across different choices of spatial localization $\eta$. For small $\eta$, the graph TV test outperforms each of the chi-squared tests, as our theory predicts. When there is less spatial localization, the optimally tuned chi-squared test performs slightly better than the graph TV test -- but this is actually somewhat impressive for the graph TV test, since it has \emph{not} been optimally tuned. On this evidence, it seems that in addition to being highly sensitive to spatially localized alternatives, the graph TV test may also have reasonable power against spatially contiguous alternatives that differ at larger scales. % We revisit this possibility in Section~\ref{sec:discussion}.

\begin{figure}[htb]
	\centering
	\begin{subfigure}[t]{.32\textwidth}
		\includegraphics[width=\textwidth]{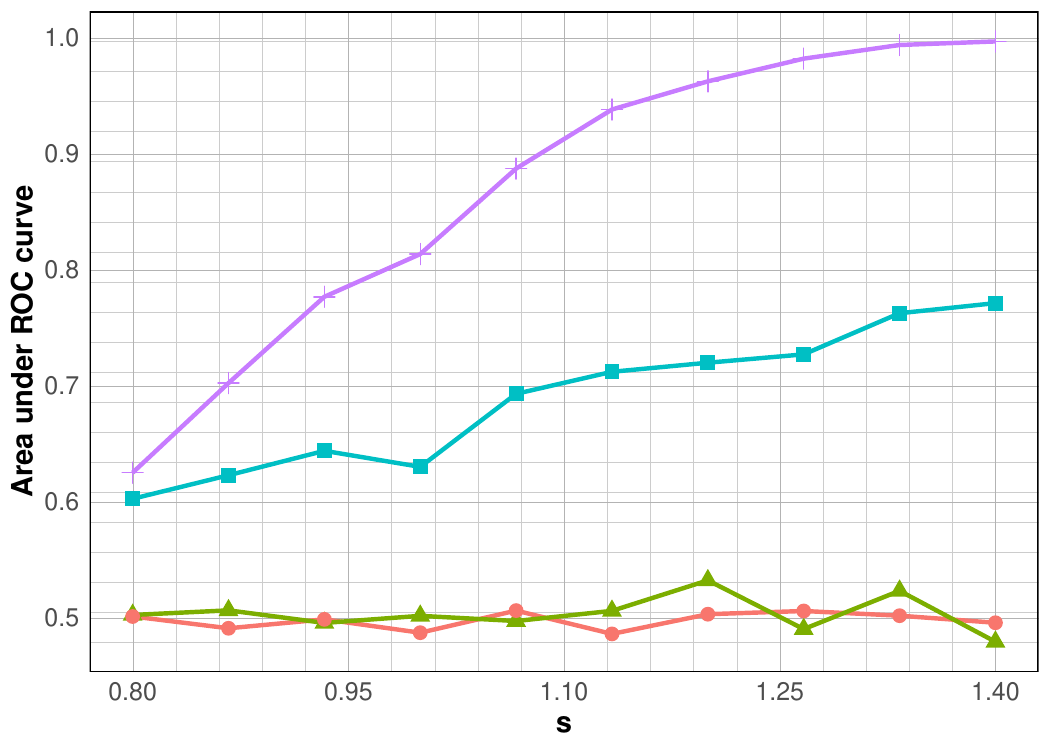}
		\caption{$\eta = 0.02$.}
	\end{subfigure}
	\begin{subfigure}[t]{.32\textwidth}
		\includegraphics[width=\textwidth]{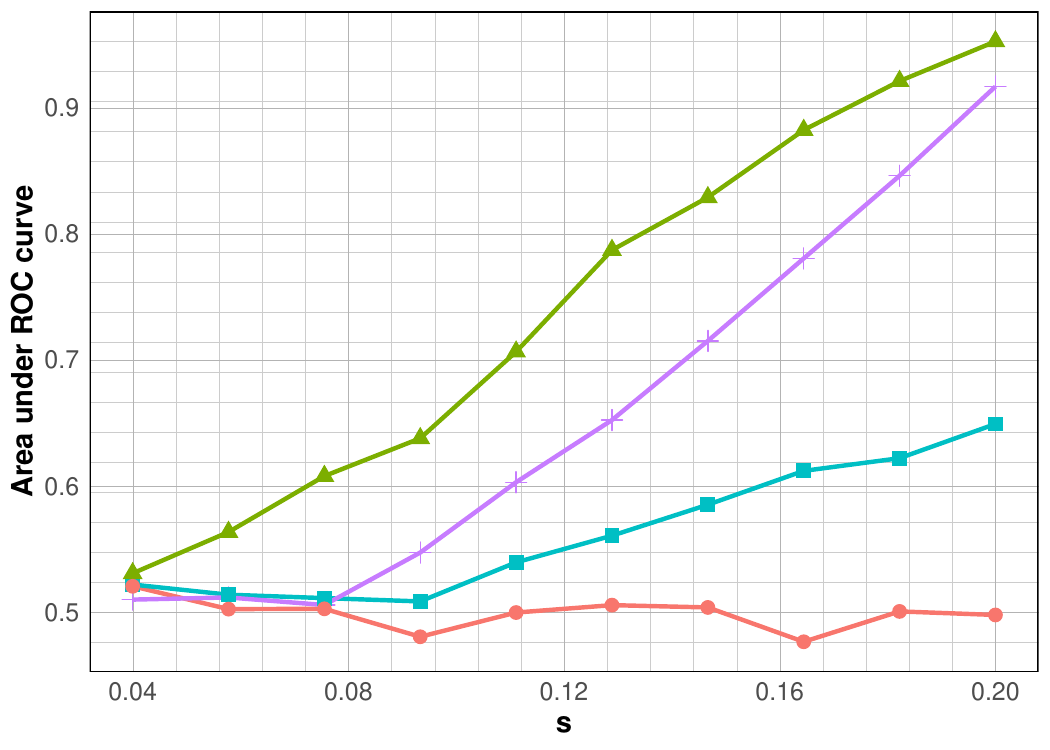}
		\caption{$\eta = 0.1$.}
	\end{subfigure}
	\begin{subfigure}[t]{.32\textwidth}
		\includegraphics[width=\textwidth]{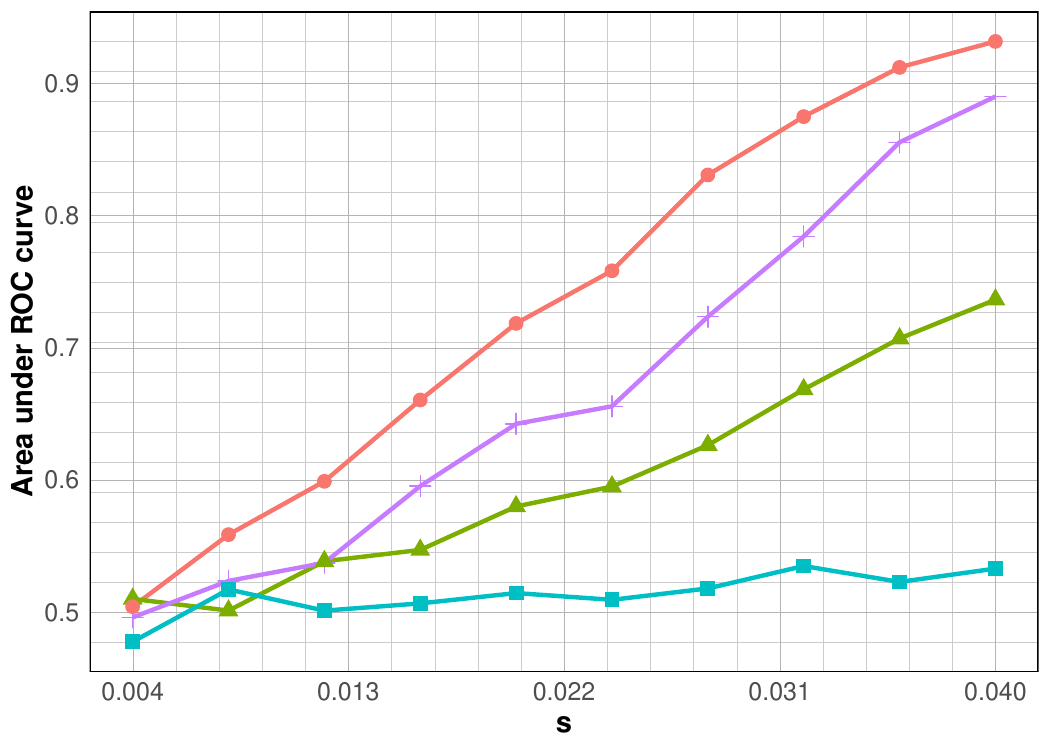}
		\caption{$\eta = 0.5$.}
	\end{subfigure}
	\caption{Area under the ROC curve (AUC) across different alternatives. Each plot corresponds to a different $\eta$. The signal strength $s$ is varied on $x$-axis; note that the scale of the $x$-axis differs in each plot. Tests compared are the graph TV test (\textcolor{Purple}{purple}) and the chi-squared test with binwidth $\varepsilon = 0.02$ (\textcolor{Aquamarine}{light blue}), $\varepsilon = 0.1$ (\textcolor{SpringGreen}{green}) and $\varepsilon = 0.5$ (\textcolor{orange}{orange}).}
	\label{fig:conceptual-example}
\end{figure}

\subsection{Chicago crime data}
For a real data example, we look at a dataset of robberies in Chicago in 2022, and use the graph TV IPM to distinguish between a subset of the true data, and synthetic samples drawn from a Gaussian mixture model fit to some separate training data. The goal is to identify whether the synthetic and real data are drawn from statistically indistinguishable distributions, and secondarily, to determine where the distributions differ. 
% The idea for using this data comes from~\red{jittkritum2017linear}, who implement a kernel test with the kernel function evaluated at locations strategically chosen to maximize power; \red{jittkritum2017linear} argue that their test is interpretable because the selected locations identify regions where $P$ and $Q$ are very different 

Figure~\ref{fig:real-data} visualizes the results. As pointed out by~\citet{jitkrittum2017linear}, a $2$-component Gaussian mixture model (GMM) is a poor fit to the data due to the geography of Chicago and Lake Michigan, and the graph TV test is significant at the $5\%$ level. Examining the witness function shows that the graph TV IPM has identified a small region near the lake where there are many more robberies than anticipated by the mixture model. Indeed, when a GMM is fit with three or more components, one component is always specifically dedicated to modeling the presence of an anomalous high-density ``cluster'' in this region. The fit is improved when $10$ components are used, although the graph TV test is still significant at the $10\%$ level.

\begin{figure}[htbp]
	\centering
	\begin{subfigure}{0.45\textwidth}
		\centering
		\includegraphics[width=\linewidth]{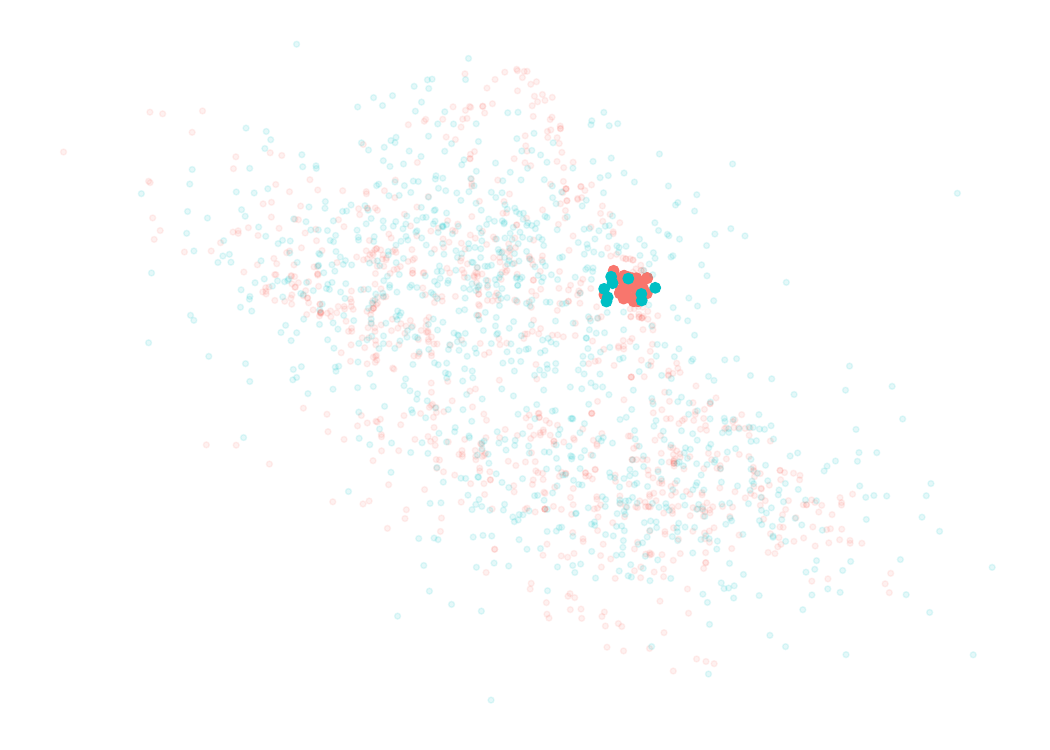} % First image file
		\caption{\textcolor{orange}{Real} and \textcolor{Aquamarine}{simulated} data.}
	\end{subfigure}\hfill
	\begin{subfigure}{0.45\textwidth}
		\centering
		\includegraphics[width=\linewidth]{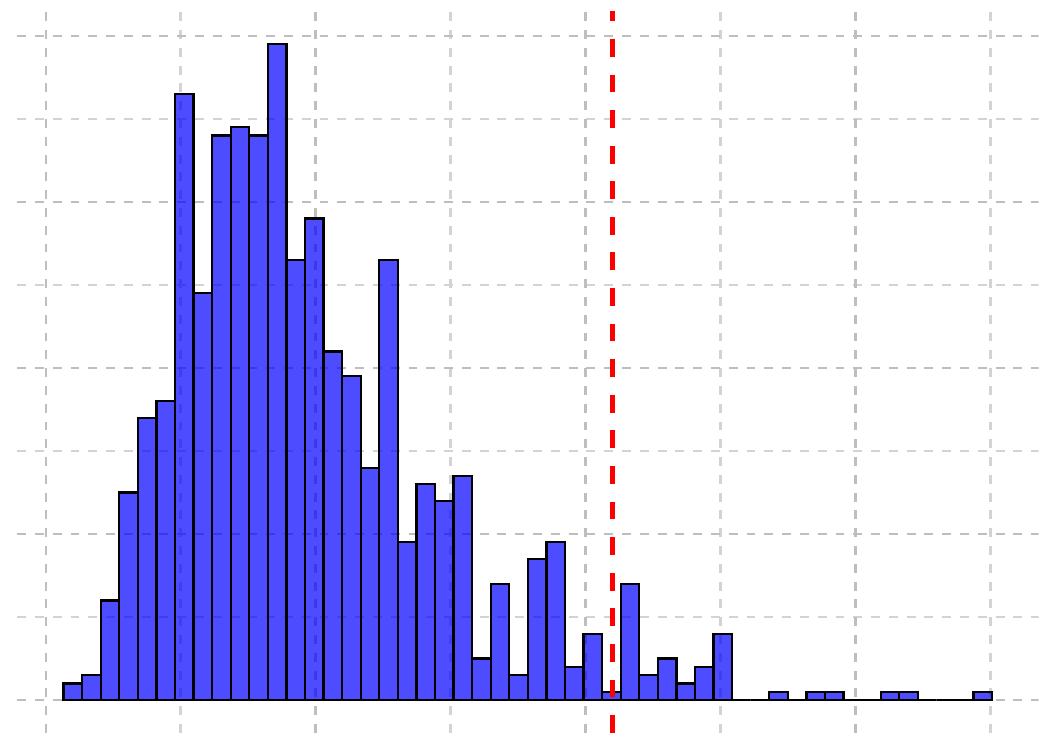} % Second image file
		\caption{$p =.043$.}
	\end{subfigure}
	\begin{subfigure}{0.45\textwidth}
		\centering
		\includegraphics[width=\linewidth]{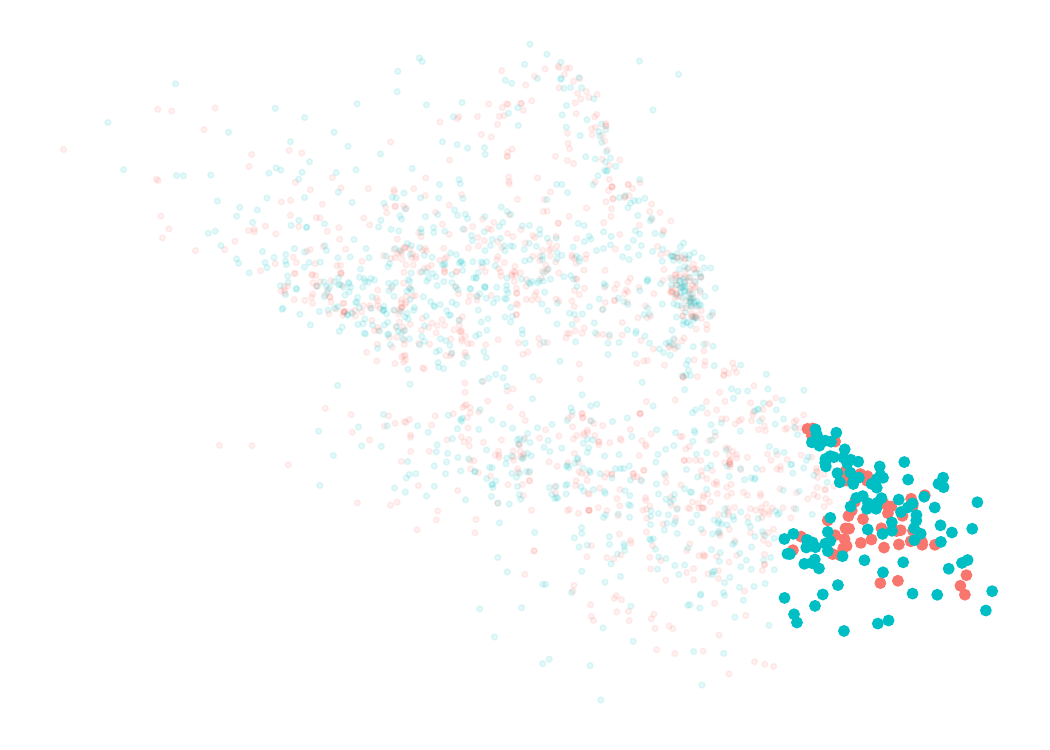} % First image file
		\caption{\textcolor{orange}{Real} and \textcolor{Aquamarine}{simulated} data.}
	\end{subfigure}\hfill
	\begin{subfigure}{0.45\textwidth}
		\centering
		\includegraphics[width=\linewidth]{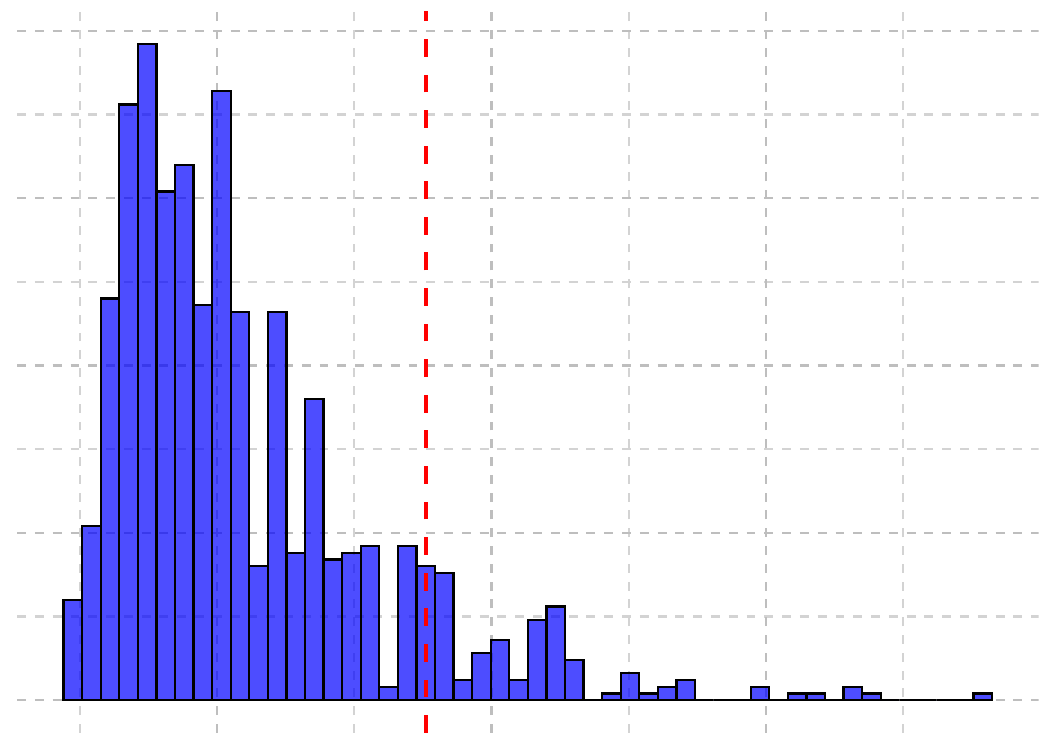} % Second image file
		\caption{$p =.093$.}
	\end{subfigure}
	\caption{Chicago crime data. Left column: Real data compared with data sampled from the Gaussian mixture model. Bolded points represent the support of the witness function. Right column: permutation distribution, with test statistic marked by vertical red dashed line. Top row: $2$-components; bottom row: $10$-components.}
	\label{fig:real-data}
\end{figure}

\section{Extensions}
\label{sec:discussion}

In this paper we have proposed a novel multivariate nonparametric test using a graph TV IPM, and shown that it has optimal power against a class of alternatives -- distributions with densities bounded from above and below that are well-separated according to a continuum TV IPM -- that include spatially localized alternatives as a special case. 

One advantage of using IPMs for statistical inference is that is typically straightforward to adapt a method designed for one problem -- such as two-sample testing -- to test other nonparametric hypotheses. (See e.g Section 7 of~\citet{gretton2012kernel}, which discusses several such adaptations of a two-sample kernel test.) We now explain how to alter two-sample graph TV test to test two different canonical nonparametric hypotheses.

\subsection{Goodness-of-fit testing}

In goodness-of-fit testing -- also known as one-sample testing -- one observes $X_1,\ldots,X_{n} \sim P$ and tests the null hypothesis $H_0: P = P_0$ where $P_0$ is a known distribution. As currently formulated, we cannot measure the distance between $\mc{X}_{n}$ and $P_0$ using the graph TV IPM, since the statistic is not defined for continuous distributions. However this is not really an issue, so long as we can sample from $P_0$. In that case, we would simply draw $n_0$ independent samples $Y_1,\ldots,Y_{n_0}$ from $P_0$, and use $d_{\DTV(G)}(\mc{X}_{n_1},\mc{Y}_{n_0})$ to empirically assess goodness-of-fit. A corresponding test could be calibrated by permutation exactly as in the two-sample case. In principle $n_0$ would be another user-specified tuning parameter, though in practice we would expect that setting $n_0 = n$ would often be a reasonable choice. Indeed for $n_0 = n$ our theory implies that if $(P,P_0) \in \mc{P}^{\infty}(d)$, then the resulting graph TV test, calibrated by permutation, will have non-trivial power if $d_{\BV}(P,P_0) \gtrsim (\log n/n)^{1/d}$ for $d \geq 3$, and if $d_{\BV}(P,P_0) \gtrsim \log n/n^{1/2}$ when $d = 2$. 

\subsection{Regression testing}
In regression testing -- also known as specification testing --  we observe independent pairs $(Z_i,U_i) \sim P,i = 1,\ldots,n$ with $Z_i \in \Reals^d$ being the (random) covariates and $U_i \in \Reals$ being the response. We suppose that $(Z,U) \sim P$ can be modeled as $U = \mu(Z) + W$ with errors $W = U - \mu(Z)$ satisfying $\mathbb{E}[U|Z] = 0$ and $W \ind Z$. The goal is to test the null hypothesis $H_0: \mu \in \mc{M}_0$, where $\mc{M}_0$ is typically a finite-dimensional parametric model for the conditional mean. In this discussion we will take $\mc{M}_0 = \{\mu_0\}$ to be a point null: a noteworthy example is $\mu_0 = 0$, in which case the problem is to detect whether any signal is present. 

The following defines a TV IPM between $\mu_0$ and the true conditional mean $\mu$:
$$
d_{\BV}(\mu_0,\mu) = \sup_{f: \TV(f) \leq 1} \int f(x) (\mu(z) - \mu_0(z)) \,dP_Z(z),
$$
where $P_Z$ is the marginal distribution of $Z$. Under suitable regularity conditions the functional $d_{\BV}(\mu_0,\mu) = 0$ if and only if $\mu = \mu_0$ $P_Z$-almost everywhere. The point null $H_0: \mu = \mu_0$ can be tested using the following adaptation of the graph TV IPM:
$$
d_{\DTV(G)}(e) := \sup_{\theta: \|D_{G}\theta\|_1 \leq 1} \; \sum_{i = 1}^{n} (Y_i  - \mu_0(Z_i))\theta_i,
$$
where $e = (Y_1  - \mu_0(Z_1),\ldots,Y_n - \mu_0(Z_n)) \in \Reals^n$. If one is willing to specify a particular distribution for the errors -- for instance $W \sim N(0,\sigma^2)$ -- then a test using $d_{\DTV(G)}(e)$ can be calibrated by Monte Carlo. Otherwise, we can calibrate by permuting the errors $e$. This results in a correctly calibrated test -- conditional on $\mc{Z}_n$ and hence marginally -- because under $H_0$, both $e$ and $e_{\pi} := (e_{\pi(1)},\ldots,e_{\pi(n)})$ have the same distribution conditionally on $\mc{Z}_n$.

\bibliographystyle{plainnat}
\bibliography{bibliography}

\appendix

\section{Proof of Theorem~\ref{thm:graph-tv}}
\label{sec:pf-graph-tv}
For convenience take $\lambda = \frac{n_1}{n}$. In this section we complete the proof of Theorem~\ref{thm:graph-tv} by establishing that both inequalities in~\eqref{pf:graph-tv-0} hold for the threshold
$$
t_{n_1,n_2} := \frac{C_3}{n \log n\min(\lambda, 1 - \lambda)} \times 
\begin{dcases}
	\sqrt{\log n}, & \quad \textrm{$d = 2$} \\
	1, & \quad \textrm{$d \geq 3$.}
\end{dcases}
$$
Sections~\ref{subsec:pf-graph-tv-ub}-\ref{subsec:step-3} establish the claimed upper bound on the permutation threshold $t_{\DTV}(\mc{X}_{n_1}, \mc{Y}_{n_2})$. Sections~\ref{subsec:pf-graph-tv-lb}-\ref{subsec:pf-representation-population-tv-ipm-2} establish the claimed lower bound on the $\varepsilon_n$-graph TV IPM $d_{\DTV}(\mc{X}_{n_1}, \mc{Y}_{n_2})$.

\subsection{Upper bound on permutation critical value: proof outline}
\label{subsec:pf-graph-tv-ub}
We begin by writing $t_{\DTV}(\mc{X}_{n_1},\mc{Y}_{n_2})$ as the $(1 - \alpha)$th quantile of an empirical process:
\begin{equation}
	\label{eqn:empirical-process}
	t_{\DTV}(\mc{X}_{n_1}, \mc{Y}_{n_2}) = \inf\biggl\{t: \mathbb{P}_{\Pi \sim \mathrm{Unif}(S_n)}\Bigl(\max\Bigl\{\sum_{i = 1}^{n} \theta_i a_{\Pi(i)}: \|D_{\varepsilon_n}\theta\|_1 \leq 1\Bigr\} \leq t|\mc{Z}_n\Bigr) \geq 1 - \alpha \biggr\}.
\end{equation}
In the probability in~\eqref{eqn:empirical-process}, all that is random is $\Pi \sim \mathrm{Unif}(S_n)$, which is a randomly chosen permutation of the assignment vector $a_i$ that is independent of $\mc{X}_{n_1}$ and $\mc{Y}_{n_2}$. Our strategy will be to exhibit a set $Z^n$  such that for every $(P,Q) \in \mc{P}^{\infty}(d)$,
\begin{equation}
\label{pf:graph-tv-ub-1}
\begin{aligned}
	& \mathbb{P}_{P,Q}(\mc{Z}_n \in Z^n) \geq 1 - \frac{\beta}{2}, \quad \textrm{and} \\
	& \mathbb{P}_{\Pi \sim \mathrm{Unif}(S_n)}\biggl(\max\Bigl\{\frac{1}{n}\sum_{i = 1}^{n} \theta_i a_{\Pi(i)}: \|D_{\varepsilon_n}\theta\|_1 \leq 1\Bigr\} \leq t_{n_1,n_2}|\mc{Z}_n\biggr) \geq 1 - \alpha ~~ \textrm{for every $\mc{Z}_n \in Z^n$.}
\end{aligned}
\end{equation}
The second statement in~\eqref{pf:graph-tv-ub-1} implies that $t_{\DTV}(\cX_{n_1},\cY_{n_2}) \leq t_{n_1,n_2}$ whenever $\mc{Z}_n \in Z^n$. The first statement in~\eqref{pf:graph-tv-ub-1} implies that this event occurs with probability at least $1 - \beta/2$, which is the claim of~\eqref{pf:graph-tv-0}.

In what follows, recall the data-independent piecewise cubic partition $\Xi_{\varepsilon}$ introduced in Section~\ref{subsec:spatially-localized-alternatives-graphtvipm}. The proof of~\eqref{pf:graph-tv-ub-1} will use the partition $\Xi_{\varepsilon_0}$ where $\varepsilon_0 = \frac{\varepsilon}{2\sqrt{d}}$. To reduce notational overhead, we will use the abbreviation $\Xi = \Xi_{\varepsilon_0}$ in the rest of the proof. The proof of~\eqref{pf:graph-tv-ub-1} is long and we begin with a high-level outline. 
\begin{enumerate}
	\item We approximate $\theta \in \Rn$ by averaging over cubes $\Delta \in \Xi$. To that end we establish some deterministic upper bounds on (i) the approximation error of $\theta - P_{\Xi}\theta$ and (ii) a grid-based discrete total variation (grid TV) of $P^{\Xi}\theta$. This idea is inspired by~\citet{padilla2020adaptive}.
	\item We apply the estimates of Step 1 to upper bound the permutation empirical process in~\eqref{eqn:empirical-process}. Here we extend known upper bounds on the Gaussian complexity of a constraint set involving the TV of a grid graph~\citep{hutter2016optimal} to apply to (what we refer to as) a \emph{permutation complexity} that is the relevant functional in our case. To bound this permutation complexity we rely on a Bernstein-type inequality for randomly permuted sums due to~\citet{albert2019concentration}.
	\item The upper bounds in Step 2 are tight so long as $\mc{Z}_n$ is spread out over $\Omega$, in the sense that each cube $\Delta \in \Xi$ contains sufficiently many points and no cube contains too many points. We conclude the proof by showing that when $\mc{X}_{n_1},\mc{Y}_{n_2}$ are randomly sampled from $(P,Q) \in \mc{P}^{\infty}(d)$ this happens with high probability.
\end{enumerate}

\subsection{Step 1: Piecewise constant approximation and associated estimates}
\label{subsec:piecewise-constant-approximation}
As mentioned, to upper bound~\eqref{eqn:empirical-process} we take a piecewise constant approximation of each $\theta \in \Rn$. In this section we define this approximation scheme and give some upper bounds on approximation error and discrete TV of the approximant.

\paragraph{Piecewise constant approximation.}
Let $N_0 = \floor{1/\varepsilon_0}$ and $n_0 := N_0^d$. For each cube $\Delta \in \Xi$ let $n(\Delta)$ be the number of points in $\mc{Z}_n$ that fall in $\Delta$, and define
$$n_{\min}(\Xi) := \min_{\Delta \in \Xi} n(\Delta), \quad n_{\max}(\Xi) := \max_{\Delta \in \Xi} n(\Delta).$$
Assume in what follows that $n_{\min}(\Xi)$ is strictly positive. [The choice of $\varepsilon = \varepsilon_n$ and $\varepsilon_0 = \frac{\varepsilon}{2\sqrt{d}}$ will mean that this is true with high probability so long as $Z_{1},\ldots,Z_n$ are independent draws from pairs $(P,Q) \in \mc{P}_{\infty}(\Omega)$, see Section~\ref{subsec:step-3}.]  

Define two approximants $P_{\Xi}\theta, P^{\Xi}\theta$ by averaging $\theta \in \Reals^n$ over each cube in the partition:
$$
P_{\Delta}\theta := \frac{1}{n(\Delta)}\sum_{i = 1}^{n} \theta_i \1\{Z_i \in \Delta\}, \quad (P^{\Xi}\theta)_\Delta := P_{\Delta}\theta, \quad (P_{\Xi}\theta)_i := \sum_{\Delta \in \Xi} (\P_{\Delta}\theta)  \1\{Z_i \in \Delta\}.
$$
The difference between $P_{\Xi}\theta$ and $P^{\Xi}\theta$ is that the former is a vector in $\Reals^n$ whereas the latter is a vector in $\R^{\Xi}$.

\paragraph{Upper bound on approximation error.}
We provide a bound on the $\ell^1$ approximation error $\|\theta - P_{\Xi}\theta\|_1$ in terms of the $\varepsilon$-graph TV that holds whenever $\varepsilon \geq \sqrt{d}\varepsilon_0$. For a given $Z_i \in \Delta$, we have
\begin{align*}
	|\theta_i - P_{\Delta} \theta| & = \biggl|\frac{1}{n(\Delta)}\sum_{j = 1}^{n} (\theta_i - \theta_j) \1\{Z_j \in \Delta\}\biggr| \\
	& \leq\frac{1}{n(\Delta)}\sum_{j = 1}^{n} |\theta_i - \theta_j| \1\{Z_j \in \Delta\} \\
	& \leq \frac{1}{n(\Delta)}\sum_{j = 1}^{n} |\theta_i - \theta_j| \1\{\|Z_i - Z_j\|_2 \leq \varepsilon\}, \tag{$\varepsilon \geq \sqrt{d}\varepsilon_0$}
\end{align*}
Summing over all $i = 1,\ldots,n$ gives
\begin{equation}
	\label{pf:graph-tv-typeI-step2-1}
	\|\theta - P_{\Xi}\theta\|_1 \leq \frac{1}{n_{\min}(\Xi)} \sum_{i,j = 1}^{n} |\theta_i - \theta_j| \1\{\|Z_i - Z_j\|_2 \leq \varepsilon\} = \frac{\DTV_{\varepsilon}(\theta)}{n_{\min}(\Xi)}.
\end{equation}

\paragraph{Upper bound on grid discrete TV.}
The partition $\Xi$ induces an unweighted undirected geometric graph $G_{\Xi} = (\Xi,\ttE_{\Xi})$ isomorphic to the $d$-dimensional grid graph. A given pair $\{\Delta,\Delta'\} \in \ttE_{\Xi}$ if and only if $\Delta = \Delta_j, \Delta' = \Delta_{j \pm e_i}$ for some $i = 1,\ldots,d$. The associated grid discrete TV is
$$
\DTV_{\Xi}(\gamma) := \sum_{\Delta,\Delta'} |\gamma_{\Delta} - \gamma_{\Delta'}| \times \1\bigl(\{\Delta,\Delta'\} \in \ttE_{\Xi}\bigr).
$$
Consider now the grid discrete TV of $P^{\Xi}\theta$. For any $\Delta \sim \Delta'$ in $G_{\Xi}$,
\begin{align*}
	|P_{\Delta}\theta - P_{\Delta'}\theta| & = \Bigl|\frac{1}{n(\Delta)}\sum_{i = 1}^n\theta_i \1(Z_i \in \Delta) - \frac{1}{n(\Delta')}\sum_{j = 1}^{n} \theta_j \1(Z_j \in \Delta') \Bigr| \\
	& \leq \frac{1}{n(\Delta) n(\Delta')}\sum_{i,j = 1}^{n} |\theta_i - \theta_j| \times \1(Z_i \in \Delta, Z_j \in \Delta').
\end{align*}
Now use the fact that $\varepsilon = 2\sqrt{d}\varepsilon_0$ so that if $Z_i,Z_j$ belong to adjoining grid cells then $\|Z_i - Z_j\|_2 \leq \varepsilon$, and so there is an edge between them in the $\varepsilon$-graph. Then summing over all $\Delta \sim \Delta'$ gives
\begin{equation}
	\label{pf:graph-tv-typeI-step2-3}
	\DTV_{\Xi}(P^{\Xi}\theta) \leq \frac{1}{\big[n_{\min}(\Xi)\bigr]^2}\sum_{i,j =1}^{n} |\theta_i - \theta_j| \times \1(\|Z_i - Z_j\|_2 \leq \varepsilon) = \frac{1}{\bigl[n_{\min}(\Xi)\bigr]^2} \DTV_{\varepsilon}(\theta).
\end{equation}

\subsection{Step 2: Upper bound on empirical process}
Let $\xi = n \times a_{\Pi}$ and $W_{\Delta} := \sum_{Z_i \in \Delta} \xi_i$ for each $\Delta \in \Xi$. For a given $\theta \in \Reals^n$ we have the decomposition,
\begin{equation*}
	\begin{aligned}
		\frac{1}{n}\sum_{i = 1}^{n} \theta_i \xi_i
		& = \frac{1}{n}\sum_{i = 1}^{n} (P_{\Xi}\theta)_i \xi_i + \frac{1}{n}\sum_{i = 1}^{n}\bigl(\theta_i - (P_{\Xi}\theta)_i\bigr)  \xi_i \\
		& = \frac{1}{n}\sum_{\Delta \in \Xi} (P^{\Xi}\theta)_\Delta W_\Delta + \frac{1}{n}\sum_{i = 1}^{n}\bigl(\theta_i - (P_{\Xi}\theta)_i\bigr) \xi_i \quad \textrm{(by the piecewise constant structure of $P_{\Xi}\theta$)}.
	\end{aligned}
\end{equation*}
Using the deterministic results of Step 1, we give separate high probability upper bounds on the two terms in the decomposition, which we call the ``main term'' and ``truncation term''. These results then imply an upper bound on the permutation process of~\eqref{eqn:empirical-process}.

Throughout this section \emph{all probabilistic statements are conditional on $\mc{Z}_{n}$}, and quantify only the randomness in $\xi_i, i = 1,\ldots,n$ due to the randomly chosen permutation $\Pi \sim \mathrm{Unif}(S_n)$. Deriving an upper bound on the truncation term is simple and we begin with this.

\paragraph{Truncation term.}
By H\"{o}lder's inequality and then the upper bound on approximation error in~\eqref{pf:graph-tv-typeI-step2-1},
\begin{equation}
	\label{pf:graph-tv-typeI-step2-2}
	\begin{aligned}
		\frac{1}{n} \sum_{i = 1}^{n}(\theta_i - P_{\Xi}\theta(Z_i)) \xi_i 
		& \leq \frac{\max_{i = 1,\ldots,n}|\xi_i|}{n} \sum_{i = 1}^{n}|\theta_i - (P_{\Xi}\theta)_i|  \\
		& \leq \frac{1}{\min(\lambda, 1 - \lambda) n} \times \frac{\DTV_{\varepsilon}(\theta)}{n_{\min}(\Xi)}.
	\end{aligned}
\end{equation}

\paragraph{Review: upper bound on Gaussian complexity of unit grid TV ball.}
Recall that $G_{\Xi}$ is a $d$-dimensional grid graph with $n_0$ vertices. The following upper bound on the Gaussian complexity of mean-zero vectors in the unit grid TV ball is due to~\citet{hutter2016optimal}: in the notation of this paper, for independent Normal random variables $w_\Delta \sim N(0,\sigma^2), \Delta \in \Xi$,
\begin{equation*}
	\sup\Bigl\{\sum_{\Delta \in \Xi} \gamma_{\Delta} w_\Delta: \bar{\gamma} = 0, \DTV_{\Xi}(\gamma) \leq 1 \Bigr\} = 
	\begin{dcases}
		O_{P}\bigl(\sigma \log n_0 \bigr), & \quad \textrm{$d = 2$} \\
		O_{P}\bigl(\sigma \sqrt{\log n_0}\bigr), & \quad \textrm{$d \geq 3$.}
	\end{dcases}
\end{equation*}
To derive this result~\citet{hutter2016optimal} rely on asymptotic properties of the eigenvalues $\lambda_k$ and eigenvectors $u_k$ of the grid Laplacian $L_{\Xi} = D_{\Xi}^{\top} D_{\Xi}$ and the left singular vectors of $D_{\Xi}$. Specifically, for a given triplet $(\lambda_k,u_k,v_k), k = 2,\ldots,n_0$, there exist positive constants $c,C$ depending only on $d$ (and not on $n_0$) for which
$$
\lambda_k \geq c \biggl(\frac{k}{n_0}\biggr)^{2/d}, \quad \|v_k\|_{\infty},\|u_k\|_{\infty} \leq \sqrt{\frac{C}{n_0}}, \quad \textrm{for all $k = 2,\ldots,n_0$.}
$$

\paragraph{Main term: permutation complexity of unit grid TV ball.}
Lemma~\ref{lem:grid-binomial-complexity} shows that a similar upper bound holds with respect to a functional that we call ``permutation complexity'', in which the Gaussian random variables $w_{\Delta}$ are replaced by the randomly permuted sums $W_{\Delta}$. To see the correspondence between the upper bound in the permutation case and the (previously known) upper bound in the Gaussian case, examine the second term below and note that $(2\min(\lambda,1-\lambda))^{-1/2}n_{\max}(\Xi)$ is an upper bound on $\sqrt{\Var(W_{\Delta})}$.
\begin{lemma}
	\label{lem:grid-binomial-complexity}
	There exists a constant $C$ depending only on $d$ such that with probability at least $1 - \delta$,
	$$
	\sup\Bigl\{\sum_{\Delta \in \Xi} \gamma_{\Delta} W_\Delta: \DTV_{\Xi}(\gamma) \leq 1 \Bigr\} 
	\leq 
	\frac{C \log(n_0/\delta)}{\min(\lambda,1 - \lambda)} + \frac{C}{\sqrt{\min(\lambda,1 - \lambda)}} \cdot
	\begin{dcases}
		\sqrt{n_{\max}(\Xi)} \log(n_0/\delta),& \quad \textrm{$d = 2$} \\
		\sqrt{n_{\max}(\Xi)} \sqrt{\log(n_0/\delta)},& \quad \textrm{$d \geq 3$.} 
	\end{dcases}
	$$
\end{lemma}

\emph{Proof of Lemma 1.} To begin note that $W \in \row(D_{\Xi})$ as $\sum_{\Delta} W_{\Delta} = \sum_{i = 1}^{n} \xi_i = 0$. Therefore $W = D_{\Xi}^{\top} (D_{\Xi}^{\top})^{\dagger} W$ where $(D_{\Xi}^{\top})^{\dagger}$ is the Moore-Penrose pseudoinverse of $D_{\Xi}^{\top}$. 
It follows from Holder's inequality that
$$
\sum_{\Delta \in \Xi} \gamma_{\Delta} W_{\Delta} = \dotp{\gamma}{W}_2 = \dotp{D_{\Xi}\gamma}{(D_{\Xi}^{\top})^{\dagger}W}_2 \leq \|(D_{\Xi}^{\top})^{\dagger}W\|_{\infty}.
$$
Writing $(D_{\Xi}^{\top})^{\dagger} = \sum_{k = 2}^{n_0} \frac{u_k v_k^{\top}}{\sqrt{\lambda_k}}$ in terms of its singular value decomposition, the quantity we are interested in upper bounding is
$$
\|(D_{\Xi}^{\top})^{\dagger}W\|_{\infty} = \max_{j = 1,\ldots,|E_{\Xi}|} \biggl|\sum_{k = 2}^{n_0} \frac{u_{k,j}}{\sqrt{\lambda_k}} \dotp{v_k}{W}_2 \biggr| = \max_{j = 1,\ldots,|E_{\Xi}|} \biggl|\sum_{i = 1}^{n} \xi_i \sum_{k = 2}^{n_0} \frac{u_{k,j} v_{k,i}}{\sqrt{\lambda_k}}\biggr| =: \max_{j = 1,\ldots,|E_{\Xi}|} \biggl|\sum_{i = 1}^{n} U_{ij}\biggr|
$$
where $U_{ij} = \sum_{i = 1}^{n} \xi_i \sum_{k = 2}^{n_0} \frac{u_{k,j} v_{k,i}}{\sqrt{\lambda_k}}$, and in a slight abuse we have defined $v_{k,i} := (v_{k})_{\Delta}$ for each $Z_i \in \Delta$. 

Now we upper bound $|\sum_{i = 1}^{n}U_{ij}|$, using a Bernstein's inequality for randomly permuted sums due to~\citet{albert2019concentration}. For ease of reference, this inequality is recorded in~\eqref{eqn:bernstein-permuted} in Section~\ref{subsec:concentration}. To apply the result, note that for any $j = 1,\ldots,E_{\Xi}$:

(i) $\mathbb{E}[U_{ij}] = 0$ for $i = 1,\ldots,n$; \\
(ii) The variance of the sum can be decomposed into sum of variances and covariances:
$$
\mathrm{Var}\Bigl[\sum_{i = 1}^{n}U_{ij}\Bigr] = \sum_{i = 1}^{n} \mathrm{Var}\Bigl[U_{ij}\Bigr] + \sum_{i_1 \neq i_2}^{n} \Cov\Bigl[U_{i_1,j}, U_{i_2,j}\Bigr].
$$
(iii) The sum of the variances satisfies the upper bound
\begin{align*}
	\sum_{i = 1}^{n}\mathrm{Var}\Bigl[U_{ij}\Bigr] 
	& = \frac{2}{\min(\lambda,1 - \lambda)} \sum_{i = 1}^{n} \Bigl(\sum_{k = 2}^{n_0}\frac{v_{k,i} u_{k,j}}{\sqrt{\lambda_k}} \Bigr)^2 \tag{since $\Var[\xi_i] \leq \frac{2}{\min(\lambda,1 - \lambda)}$} \\
	& = \frac{2}{\min(\lambda,1 - \lambda)} \sum_{\Delta \in \Xi} n(\Delta) \Bigl(\sum_{k = 2}^{n_0}\frac{v_{k,\Delta} u_{k,j}}{\sqrt{\lambda_k}} \Bigr)^2 \\
	& \leq \frac{2 \times n_{\max}(\Xi)}{\min(\lambda,1 - \lambda)} \sum_{\Delta \in \Xi} \Bigl(\sum_{k = 2}^{n_0}\frac{v_{k,\Delta} u_{k,j}}{\sqrt{\lambda_k}} \Bigr)^2 \\
	& = \frac{2 \times n_{\max}(\Xi)}{\min(\lambda,1 - \lambda)} \sum_{k = 2}^{n_0}\frac{u_{k,j}^2}{\lambda_k} \tag{using the $\ell^2$ orthogonality of $v$} \\
	& \leq \frac{C \times n_{\max}(\Xi)}{n_0\min(\lambda,1 - \lambda)} \sum_{k = 2}^{n_0}\frac{1}{\lambda_k} \tag{using $\|u_k\|_{\infty} \leq \sqrt{\frac{C}{n_0}}$} \\
	& \leq \frac{C \times n_{\max}(\Xi)}{n_0^{(1 - 2/d)}\min(\lambda,1 - \lambda)} \sum_{k = 2}^{n_0}\frac{1}{k^{2/d}} \tag{using $\lambda_k \geq c(k/n_0)^{2/d}$} \\
	& \leq \frac{C \times n_{\max}(\Xi)}{\min(\lambda,1 - \lambda)} \times
	\begin{dcases}
		\log(n_0), & \quad \textrm{$d = 2$} \\
		1, & \quad \textrm{$d \geq 3$.}
	\end{dcases}
\end{align*}
(iv) The sum of the covariances satisfies the upper bound
\begin{align*}
	\sum_{i_1 \neq i_2}^{n} \Cov\Bigl[U_{i_1,j}, U_{i_2,j}\Bigr] 
	& = \sum_{i_1 \neq i_2}^{n} \biggl\{\Cov\Bigl[\xi_{i_1}, \xi_{i_2}\Bigr] \times \Bigl(\sum_{k = 2}^{n_0}\frac{v_{k,i_1} u_{k,j}}{\sqrt{\lambda_k}} \Bigr) \times \Bigl(\sum_{k = 2}^{n_0}\frac{v_{k,i_2} u_{k,j}}{\sqrt{\lambda_k}} \Bigr)\biggr\} \\
	& \leq \frac{1}{n - 1}\Bigl(\frac{1}{\lambda} + \frac{1}{1 - \lambda}\Bigr) \times \sum_{i_1 \neq i_2}^{n} \biggl\{ \Bigl|\sum_{k = 2}^{n_0} \frac{v_{k,i_1} u_{k,j}}{\sqrt{\lambda_k}} \Bigr| \times \Bigl|\sum_{k = 2}^{n_0}\frac{v_{k,i_2} u_{k,j}}{\sqrt{\lambda_k}} \Bigr|\biggr\} \tag{since $\Cov[\xi_{i_1},\xi_{i_2}] = \frac{-1}{n - 1}\Bigl(\frac{1}{\lambda} + \frac{1}{1 - \lambda}\Bigr)$}\\
	& \leq \frac{n}{n - 1}\Bigl(\frac{1}{\lambda} + \frac{1}{1 - \lambda}\Bigr) \times \sum_{i = 1}^{n} \Bigl(\sum_{k = 2}^{n_0} \frac{v_{k,i} u_{k,j}}{\sqrt{\lambda_k}} \Bigr)^2 \tag{using Cauchy-Schwarz} \\
	& \leq \frac{C \times n_{\max}(\Xi)}{\min(\lambda,1 - \lambda)} \times
	\begin{dcases}
		\log(n_0), & \quad \textrm{$d = 2$} \\
		1, & \quad \textrm{$d \geq 3$.}
	\end{dcases} \tag{same reasoning as in (iii).}
\end{align*}

(v) Each $U_{ij}$ satisfies the upper bound
\begin{align*}
	|U_{ij}| 
	& \leq \frac{1}{\min(\lambda,1 - \lambda)} \biggl| \sum_{k = 2}^{n_0} \frac{u_{k,j} v_{k,i}}{\sqrt{\lambda_k}} \biggr| \\
	& \leq \frac{C}{n_0\min(\lambda,1 - \lambda)} \sum_{k = 2}^{n_0} \frac{1}{\sqrt{\lambda_k}} \tag{using $\|u_k\|_{\infty},\|v_k\|_{\infty} \leq \sqrt{\frac{C}{n_0}}$} \\
	& \leq \frac{C}{n_0^{1 - 1/d}\min(\lambda,1 - \lambda)} \sum_{k = 2}^{n_0} \frac{1}{k^{1/d}} \tag{using $\lambda_k \geq c(k/n_0)^{2/d}$} \\
	& \leq \frac{C}{\min(\lambda,1 - \lambda)}.
\end{align*}
So we can apply Bernstein's inequality for permuted sums, recorded in~\eqref{eqn:bernstein-permuted}, with these upper bounds on variance and maximum absolute value, and with $\delta' = \delta/|E_{\Xi}| \geq \delta/(2^d n_0)$. Then taking a union bound over $j = 1,\ldots,|E_{\Xi}|$ gives the claimed result of the Lemma. \qed

\subsection{Step 3: Spread of sample points, and completing the proof of~\eqref{pf:graph-tv-ub-1}}
\label{subsec:step-3}

The bounds in Step 3 will be tight enough to established the desired result of~\eqref{pf:graph-tv-ub-1} as long as $\mc{Z}_{n}$ is spread out over $\Omega$ in the sense that both $n_{\min}(\Xi), n_{\max}(\Xi) = \Theta_P(\log n)$. We first state high probability bounds to that effect, and then complete the proof of~\eqref{pf:graph-tv-ub-1}.

\paragraph{Spread of sample points.}
Recall that $\varepsilon = C_1 (\log n/n)^{1/d}$ where $C_1 = 2\sqrt{d} \cdot (24B)^{1/d}$, and $\varepsilon_0 = \varepsilon/(2\sqrt{d}) = (24B\log n/n)^{1/d}$. By the assumption $(P,Q) \in \mc{P}^{\infty}(d)$, we have that the expected number of samples in any cube $\Delta \in \Xi$ is $O(\log n)$: specifically, for any $\Delta \in \Xi$,
$$
24 \log n = \frac{1}{B}n\varepsilon_0^d \leq \mathbb{E}[n(\Delta)] \leq B n \varepsilon_0^d = 24 B^2 \log n.
$$
Using a multiplicative Chernoff bound (see Section~\ref{subsec:concentration}) and a union bound, we have that
\begin{equation*}
\begin{aligned}
	& \Big|n(\Delta) - \E[n(\Delta)]\Big| \leq \frac{1}{2} \E[n(\Delta)] \; \textrm{for all $\Delta \in \Xi$,} \\
	& \textrm{with probability} \; \geq 1 - 2\varepsilon_0^{-d} \exp\Big(-\frac{\min_{\Delta} \E[n(\Delta)]}{12}\Big) \geq 1 - 2\varepsilon_0^{-d}\exp\Big(-2 \log n\Big) \geq 1 - \frac{2}{n} \geq 1 - \frac{\beta}{2}.
\end{aligned}
\end{equation*}
The last inequality follows from the assumption $\beta \geq 4/n$. In summary, with probability at least $1 - \beta/2$:
\begin{equation}
	\label{pf:graph-tv-typeI-step3-1}
	12 \log n \leq n_{\min}(\Xi) \leq n_{\max}(\Xi) \leq 36 B^2\log n.
\end{equation}
\paragraph{Proof of~\eqref{pf:graph-tv-ub-1}.}
Take $Z^n$ to be the set of possible $\mc{Z}_n$ for which $n_{\min}(\Xi)$ and $n_{\max}(\Xi)$ satisfy the inequalities of~\eqref{pf:graph-tv-typeI-step3-1}. We have just shown that $\mathbb{P}(\mc{Z}_n \in Z^n) \geq 1 - \beta/2$. Thus the first statement in~\eqref{pf:graph-tv-ub-1} is verified. Additionally, for any $\mc{Z}^n \in Z^n$, conditional on $\mc{Z}^n$ we have the following result when $d \geq 3$:
\begin{align*}
	& \max\Bigl\{\sum_{i = 1}^{n} \theta_i a_{\Pi(i)}: \DTV_{\varepsilon}(\theta) \leq 1 \Bigr\} \\
	& = \max\Bigl\{\frac{1}{n}\sum_{i = 1}^{n}\bigl(\theta_i - (P_{\Xi}\theta)_i \xi_i\bigr) + \frac{1}{n}\sum_{\Delta \in \Xi} (P^{\Xi}\theta)_\Delta W_\Delta: \DTV_{\varepsilon}(\theta) \leq 1 \Bigr\} \\
	& \leq \frac{1}{\min(\lambda,1 - \lambda)n} \times \frac{1}{n_{\min}(\Xi)} + \max\Bigl\{\frac{1}{n}\sum_{\Delta \in \Xi} (P^{\Xi}\theta)_\Delta W_\Delta : \DTV_{\varepsilon}(\theta) \leq 1 \Bigr\} \tag{by~\eqref{pf:graph-tv-typeI-step2-2}} \\
	& \leq \frac{1}{\min(\lambda,1 - \lambda)n} \times \frac{1}{n_{\min}(\Xi)} + \frac{1}{n[n_{\min}(\Xi)]^2} \times \max\Bigl\{\frac{1}{n}\sum_{\Delta \in \Xi} \gamma_\Delta W_\Delta : \DTV_{\Xi}(\gamma) \leq 1 \Bigr\} \tag{by~\eqref{pf:graph-tv-typeI-step2-3}} \\
	& \leq \frac{1}{\min(\lambda,1 - \lambda)n} \times \frac{1}{n_{\min}(\Xi)} + \frac{C}{n[n_{\min}(\Xi)]^2} \times \biggl(\frac{\log(n_0/\alpha)}{\min(\lambda,1 - \lambda)} + \sqrt{\frac{n_{\max}(\Xi) \log(n_0/\alpha)}{\min(\lambda,1 - \lambda)}}\biggr) \tag{with probability $\geq 1 - \alpha$, by Lemma~\ref{lem:grid-binomial-complexity}} \\
	& \leq \frac{C}{\min(\lambda,1 - \lambda)n \log n} + \frac{C}{n(\log n)^2} \times \biggl(\frac{\log(n)}{\min(\lambda,1 - \lambda)} + \sqrt{\frac{\log n \log(n)}{\min(\lambda,1 - \lambda)}}\biggr) \tag{by definition of $Z^n$}.
\end{align*}
In the last line we have also used the fact that $\log(1/\alpha) \leq \log n$. This implies that the second second statement in~\eqref{pf:graph-tv-ub-1} is correct, when the constant $C_3$ in the definition of the threshold $t_{n_1,n_2}$ is chosen to be sufficiently large. This completes the proof of~\eqref{pf:graph-tv-ub-1} for all $d \geq 3$, and hence establishes the claimed upper bound on the permutation critical value $t_{\DTV}(\mc{X}_{n_1},\mc{Y}_{n_2})$. When $d = 2$ the upper bound on the permutation complexity from Lemma~\ref{lem:grid-binomial-complexity} is different; otherwise the exact same steps give the claimed result.

\subsection{Lower bound on $\varepsilon_n$-Graph TV IPM}
\label{subsec:pf-graph-tv-lb}

In this section we prove the second part of~\eqref{pf:graph-tv-0}, by establishing the claimed lower bound on the $\varepsilon_n$-graph TV IPM. To do so, we proceed by identifying a witness of the population IPM and analyzing empirical functionals -- namely, difference in sample means and $\varepsilon_n$-graph TV -- of that witness. The following structural result is helpful for this purpose.
\begin{proposition}
	\label{prop:representation-population-tv-ipm}
	If $P,Q \in \mc{P}^{\infty}(d)$, then there exists a measurable set $A^{\ast} \subseteq \Rd$ with positive finite perimeter $0 < \per(A^{\ast}) < \infty$ such that the function 
	\begin{equation}
		\label{eqn:dualtv-witness}
		f^{\ast}(x) = \frac{\1(x \in A^{\ast})}{\mathrm{per}(A^{\ast})},
	\end{equation}
	is a \emph{witness} of the population TV IPM:
	\begin{equation*}
		\int f^{\ast}(x) \bigl(p(x)- q(x)\bigr) \,dx  = d_{\mathrm{BV}}(P,Q).
	\end{equation*}
	Moreover, there is a constant $C$ depending only on $d$ such that
	\begin{equation*}
		d_{\BV}(P,Q) \leq C\bigl(\nu(A^{\ast})\bigr)^{1/d}.
	\end{equation*}
\end{proposition} 
We delay the proof of Proposition~\ref{prop:representation-population-tv-ipm} to Section~\ref{subsec:pf-representation-population-tv-ipm}. Before that, we use the proposition to derive the desired lower bound on graph TV IPM. Let $A^{\ast}$ be a set that witnesses the population TV IPM in the sense that
$$
d_{\BV}(P,Q) = \frac{1}{\per(A^{\ast})}\int_{A^{\ast}} \bigl(p(x) - q(x)\bigr) \,dx,
$$ 
and let $f^{\ast}(x) = \1\{x \in A^{\ast}\}$. We have the following lower bound on graph TV IPM:
$$
d_{\DTV}(\mc{X}_{n_1},\mc{Y}_{n_2}) \geq \frac{1}{\DTV_{\varepsilon}(f^{\ast})} \Bigl(\frac{1}{n_1}\sum_{i = 1}^{n_1}f^{\ast}(X_i) - \frac{1}{n_2}\sum_{i = 1}^{n_2}f^{\ast}(Y_i)\Bigr).
$$
Applying Chebyshev's inequality to the difference in sample means, we see that with probability at least $1 - \delta$,
\begin{align*}
	\frac{1}{n_1}\sum_{i = 1}^{n_1}f^{\ast}(X_i) - \frac{1}{n_2}\sum_{i = 1}^{n_2}f^{\ast}(Y_i) 
	& \geq 
	\mathbb{P}_{P}(X \in A^{\ast}) -  \mathbb{P}_{Q}(Y \in A^{\ast}) - \sqrt{\frac{1}{n \delta}\biggl(\frac{\mathbb{P}_{P}(X \in A^{\ast})}{\lambda} + \frac{\mathbb{P}_{Q}(Y \in A^{\ast})}{(1 - \lambda)}\biggr)} \\
	& \geq 
	\mathbb{P}_{P}(X \in A^{\ast}) -  \mathbb{P}_{Q}(Y \in A^{\ast}) - \sqrt{\frac{B \nu(A^{\ast})}{n \delta \min(\lambda, 1 - \lambda)}}. \tag{by~\eqref{eqn:bounded-density}}
\end{align*}
On the other hand, high-probability upper bounds on the neighborhood graph TV in terms of continuum TV are known in the literature. In particular Lemma S.6 of~\citet{hu2022voronoigram} implies that
$$
\DTV_{\varepsilon}(\1_{A^{\ast}}) \leq \frac{C}{\delta} n^2 \varepsilon^{d + 1} \per(A^{\ast}),
$$
with probability at least $1 - \delta$. (Lemma S.6 of~\citet{hu2022voronoigram} deals with i.i.d data and so technically speaking does not apply to our two-sample setting, but basic modifications of the proof yield an unchanged result.)

Taking $\delta = \beta/4$, we have that with probability at least $1 - \beta/2$,
\begin{equation}
	\label{pf:graph-tv-lb-1}
	d_{\DTV}(\mc{X}_{n_1},\mc{Y}_{n_2}) 
	\geq 
	\frac{1}{C n^2 \varepsilon^{d  + 1}}\biggl( \frac{\beta}{2} \cdot d_{\BV}(P,Q) - \frac{1}{\per(A^{\ast})} \cdot \sqrt{\frac{2 B \nu(A^{\ast})}{n \beta \min(\lambda, 1 - \lambda)}}\biggr).
\end{equation}
Now we show that the error term in this lower bound is meaningfully smaller than $d_{\BV}(P,Q)$, using (i) an isoperimetric inequality -- recorded in~\eqref{eqn:isoperimetric-inequality} -- that lower bounds the perimeter of any set $A$ in terms of its volume, and (ii) the lower bound on the area of $A^{\ast}$ given by Proposition~\ref{prop:representation-population-tv-ipm}. In particular,
\begin{align*}
	\frac{1}{\per(A^{\ast})} \sqrt{ \frac{\nu(A^{\ast})}{n}} 
	& \leq C \frac{\bigl(\nu(A^{\ast})\bigr)^{1/d}}{\sqrt{n \cdot \nu(A^{\ast})}} \tag{by~\eqref{eqn:isoperimetric-inequality} }\\
	& \leq C \frac{d_{\BV}(P,Q)}{\sqrt{\bigl(d_{\BV}(P,Q)\bigr)^d \cdot n}} \tag{by Proposition~\ref{prop:representation-population-tv-ipm}}\\
	& \leq c \frac{\beta^{3/2} \cdot d_{\BV}(P,Q) \cdot \sqrt{\min(\lambda,1 - \lambda)}}{\sqrt{\log n}}.
\end{align*}
The last inequality holds from our assumed lower bound on $d_{\BV}(P,Q)$ in \eqref{eqn:detection-boundary-ub}. In this inequality, the constant $c \to 0$ as the constant $C \to \infty$ in \eqref{eqn:detection-boundary-ub}. Thus for an appropriately large choice of constant $C$ in~\eqref{eqn:detection-boundary-ub}:
\begin{equation*}
	d_{\DTV}(\mc{X}_{n_1},\mc{Y}_{n_2}) \geq \frac{d_{\BV}(P,Q) \beta}{C n^2 \varepsilon^{d  + 1}}\biggl(1 - \frac{c}{\sqrt{\log n}}\biggr) \geq \frac{d_{\BV}(P,Q)\beta}{C n^2 \varepsilon^{d  + 1}}.
\end{equation*}
Finally, it follows from the choice of radius $\varepsilon = C_1 (\log n/n)^{1/d}$ and the assumed lower bound on $d_{\BV}(P,Q)$ that, again with probability at least $1 - \beta/2$,
$$
d_{\DTV}(\mc{X}_{n_1},\mc{Y}_{n_2}) \geq \frac{C}{n \log n \bigl\{\min(\lambda,1 - \lambda)\bigr\}^{1/d}  } \times 
\begin{dcases}
	\sqrt{\log n}, & \textrm{$d = 2$} \\
	1, & \textrm{$d \geq 3$},
\end{dcases}
$$
which implies the desired claim.

\subsection{Proof of Proposition~\ref{prop:representation-population-tv-ipm}: representation of witness}
\label{subsec:pf-representation-population-tv-ipm}

In this section and the next (Section~\ref{subsec:pf-representation-population-tv-ipm-2}) we prove Proposition~\ref{prop:representation-population-tv-ipm}, starting in this section with the representation result, that there exists a witness $f^{\ast}$ of the population TV IPM which, up to normalization, is the indicator function of a set.

The claim is trivial if $P = Q$ -- any set $A^{\ast}$ with positive finite perimeter satisfies the claim -- and hereafter we assume $P \neq Q$. If $P \neq Q$ then there exists a witness $f^{\ast}$ of the TV IPM satisfying $\TV(f^{\ast}) > 0$ and $f^{\ast} \neq 0$ on a set of positive Lebesgue measure. This follows from the finiteness and characteristic property of $d_{\BV}$, and~\eqref{eqn:dual-tv-norm}.

We will prove the claim by showing that in fact there there exists a set $A^{\ast}$ with finite perimeter such that
\begin{equation}
	\label{pf:piecewise-constant-witness-0}
	\frac{1}{\per(A^{\ast})}\int_{A^{\ast}} \bigl(p(x) - q(x)\bigr) \,dx = \sup\Bigl\{\int_{\Rd} f(x) \bigl(p(x) - q(x)\bigr) \,dx: \mathrm{TV}(f) \leq 1 \Bigr\}.
\end{equation}
The idea will be to establish an equivalence between the variational problem in~\eqref{pf:piecewise-constant-witness-0}, which defines the population-level TV IPM, and a perimeter minimization problem over sets $A \subset \Rd$. The solution to this perimeter minimization problem is the $A^{\ast}$ in~\eqref{pf:piecewise-constant-witness-0}.

\paragraph{Equivalence between TV and perimeter minimization problems.}
We state here a result on an equivalence between perimeter and TV minimization problems as recorded in~\citet{chambolle2010introduction}. Consider the non-convex perimeter minimization problem
\begin{equation}
	\label{eqn:perimeter-minimization}
	\min\Bigl\{\lambda \hspace{2 pt} \per(A) - \int_{A} g(x) \,dx: \1_A \in \BV(\Rd)\Bigr\}.
\end{equation}
and its convex relaxation,
\begin{equation}
	\label{eqn:tv-minimization}
	\min\Bigl\{ \lambda \hspace{2 pt} \TV(f) - \int g(x) f(x) \,dx: \TV(f) \leq 1, 0 \leq f(x) \leq 1~~\textrm{for all $x \in \Rd$} \Bigr\}.
\end{equation}
It turns out that the for the solution $f^{\ast}$ to the TV minimization problem and any $s \in (0,1]$, the level set $\{f^{\ast} \geq 0\}$ is a solution to the perimeter minimization problem.

\paragraph{Proof of Proposition~\ref{prop:representation-population-tv-ipm}.} 

We rewrite~\eqref{pf:piecewise-constant-witness-0} as the minimization problem 
\begin{equation}
	\label{pf:piecewise-constant-witness-1}
	-\min \Bigl\{\int_{\Rd} f(x)\bigl(p(x) - q(x)\bigr): \TV(f) \leq 1\Bigr\}
\end{equation}
Consider the Lagrangian of the minimization in~\eqref{pf:piecewise-constant-witness-1}, 
\begin{equation}
	\label{pf:piecewise-constant-witness-2}
	\mc{L}(f,\lambda) := \int_{\Rd} f(x)\bigl(p(x) - q(x)\bigr) + \lambda \mathrm{TV}(f).
\end{equation}
At the specific value of the Lagrange multiplier $\lambda^{\ast} = d_{\BV}(P,Q)$, the minimum of the Lagrangian $\min_{f} \mc{L}(f,\lambda^{\ast}) = 0$ -- on the one hand clearly $\mc{L}(f,\lambda^{\ast}) \leq 0$ as the zero function is feasible, while on the other hand for any $f \in \BV(\Rd)$,
$$
d_{\BV}(P,Q) \TV(f) \geq \int_{\Rd} f(x)\bigl(p(x) - q(x)\bigr) \,dx \Longrightarrow \mc{L}(f,\lambda^{\ast}) \geq 0.
$$
We note that the same argument holds if the minimum is restricted to indicator functions of sets with finite perimeter; this will be useful shortly.

Now let $f$ be any function which achieves the minimum value $\mc{L}(f,\lambda^{\ast}) = 0$. By the coarea formula $\TV(f) = \TV(f^{+}) + \TV(f^{-})$ where $f^{+} = \max(f,0)$ and $f^{-} = \min(f,0)$. Additionally, if $\|f\|_{L^{\infty}(\Rd)} \neq 0$ then it follows from the $1$-homogeneity of $\TV$, i.e. $\TV(af) = a \TV(f)$, that $\mc{L}(g,\lambda^{\ast}) = 0$ for $g(x) = f(x)/\|f\|_{L^{\infty}(\Rd)}$. We conclude that
\begin{equation}
	\label{pf:piecewise-constant-witness-3}
	\min\Bigl\{\mc{L}(f,\lambda^{\ast}): f \in \BV(\Rd)\Bigr\} = \min\Bigl\{\mc{L}(f,\lambda^{\ast}): f \in \BV(\Rd), 0 \leq f(x) \leq 1~~ \textrm{for all $x \in \Rd$}\Bigr\}.
\end{equation}
Notice that the right hand side of~\eqref{pf:piecewise-constant-witness-3} is exactly the TV minimization problem~\eqref{eqn:tv-minimization}, with $\lambda = \lambda^{\ast}$. Notice additionally that for a witness $f^{\ast}$ of the TV IPM, the normalized function $g^{\ast} = f^{\ast}/\|f^{\ast}\|_{L^{\infty}(\Rd)}$ is a minimizer of~\eqref{pf:piecewise-constant-witness-3} for which $\TV(g^{\ast}) > 0$.

Now we use the equivalence between TV and perimeter minimization, which says that for any $s \in [0,1)$ it is the case that $\{g^{\ast} > s\}$ achieves the minimum of~\eqref{eqn:perimeter-minimization}. Moreover, there must exist some $s^{\ast} \in [0,1)$ for which $\per(\{g^{\ast} > s\}) > 0$, by the coarea formula. Take $A^{\ast} = \{g^{\ast} > s^{\ast}\}$. It follows that
$$
\lambda^{\ast} \per(A^{\ast}) - \int_{A^{\ast}} \bigl(p(x) - q(x)\bigr) \,dx = \min\Bigl\{\lambda^{\ast} \hspace{2 pt} \per(A) - \int_{A} g(x) \,dx: \1_A \in \BV(\Rd)\Bigr\} = 0.
$$ 
The second equality holds because $\min_{A} \mc{L}(\1_A,\lambda^{\ast}) = 0$, as previously argued. The previous display can be rearranged to read
$$
\frac{1}{\per(A^{\ast})}\int_{A^{\ast}} \bigl(p(x) - q(x)\bigr) \,dx = \lambda^{\ast} = d_{\BV}(P,Q),
$$
which is the desired claim.

\subsection{Proof of Proposition~\ref{prop:representation-population-tv-ipm}: estimates for witness $A^{\ast}$}
\label{subsec:pf-representation-population-tv-ipm-2}

Let $A^{c} = \Rd \setminus A$. We will assume without loss of generality that $\nu(A^{\ast}) \leq \nu((A^{\ast})^c)$, otherwise we can swap $p$ and $q$ and consider the set $\Rd \setminus A^{\ast}$. Then it follows from H\"{o}lder's inequality and the isoperimetric inequality~\eqref{eqn:isoperimetric-inequality} that
$$
d_{\BV}(P,Q) = \frac{1}{\per(A^{\ast})}\int_{A^{\ast}} \bigl(p(x) - q(x)\bigr) \,dx \leq \Bigl(\|p - q\|_{L^{\infty}(\Rd)}\Bigr)\frac{\nu(A^{\ast})}{\per(A^{\ast})} \leq C \bigl(\nu(A^{\ast})\bigr)^{1/d}.
$$

\section{Proof of Theorem~\ref{thm:lower-bound}}
The proof of Theorem~\ref{thm:lower-bound} proceeds in a classical way due to Ingster, in which the minimax risk is lower bounded by the probability of type II error in a two-point testing problem. We briefly review the general technique, and then give a specific construction that leads to Theorem~\ref{thm:lower-bound}.

\subsection{Review: reduction to two-point testing problem}
For collections of pairs of distributions $\mc{P}_0, \mc{P}_1$, define the minimax type II risk to be 
$$
\beta_{n_1,n_2}(\mc{P}_0;\mc{P}_1) := \inf_{\varphi} \sup_{(P,Q) \in \mc{P}_1} \mathbb{E}_{P,Q}[1 - \varphi],
$$
with the infimum being over all tests that are level-$\alpha$ for all $(P_0,P_0) \in \mc{P}_0$. Let $\Pi_0$ be a prior over $\mc{P}_0$, $\Pi_1$ be a prior over $\mc{P}_1$, and define
\begin{equation*}
	\begin{aligned}
		\nu_0\Big(\{x_1,\ldots,x_{n_1}\},\{y_1,\ldots,y_{n_2}\}\Big) & := \sum_{(P,P) \in \mc{P}_0} \Pi_0\big((P,P)\big) \cdot \prod_{i = 1}^{n_1} p(x_i) \cdot \prod_{i = 1}^{n_2} p(y_i)\\
		\nu_1\Big(\{x_1,\ldots,x_{n_1}\},\{y_1,\ldots,y_{n_2}\}\Big) & := \sum_{(P,Q) \in \mc{P}_1} \Pi_1\big((P,Q)\big) \cdot \prod_{i = 1}^{n_1} p(x_i) \cdot \prod_{i = 1}^{n_2} q(y_i).
	\end{aligned}
\end{equation*}
That is, $\nu_0$ is the density of $(\mc{X}_{n_1},\mc{Y}_{n_2})$ under $\Pi_0$, and $\nu_1$ is the density under $\Pi_1$. The likelihood ratio statistic for distinguishing $H_0: \nu = \nu_0$ versus $H_1: \nu = \nu_1$ is
\begin{equation*}
	R(\mc{X}_{n_1},\mc{Y}_{n_2}) := \frac{\nu_1(\mc{X}_{n_1},\mc{Y}_{n_2})}{\nu_0(\mc{X}_{n_1},\mc{Y}_{n_2})}.
\end{equation*}

The following result links the minimax type II risk to the first two moments of $R$. It is a direct extension to the two-sample setting of a result in goodness-of-testing due to Ingster, as recorded in~\citet{balakrishnan2019hypothesis}. 

\begin{lemma}
	\label{lem:likelihood-ratio}
	Let $0 < \beta < 1 - \alpha$. If 
	\begin{equation*}
		\Ebb_{\nu_0}[R(\mc{X}_{n_1},\mc{Y}_{n_2})^2] \leq 1 + 4(1 - \alpha - \beta)^2,
	\end{equation*}
	then $\beta_{n_1,n_2}(\mc{P}_0,\mc{P}_1) \geq \beta$. 
\end{lemma}

\subsection{Construction of alternatives}
\label{subsec:pf-lower-bound-construction}

Now we return to problem of testing for separation in the TV IPM, and construct collections $\mc{P}_0$ and $\mc{P}_1$ that lead to a tight lower bound. In the rest of this proof, for notational convenience we will assume that {$n_2 \leq n_1$}; of course, this is without loss of generality as we could otherwise relabel.

For the null, we will simply take $P_0$ to be the uniform distribution over $\Omega$ and let $\mc{P}_0$ be the singleton $ \{(P_0,P_0)\}$. Obviously the only possible prior $\Pi_0$ is the one which puts all its mass on the singleton.  

Each pair of alternatives $(P,Q) \in \mc{P}_1$ will be defined as a small, spatially localized perturbation of the null $(P_0,P_0)$. More specifically,  recall the collection of cubes $\Xi$ defined in the proof of Theorem~\ref{thm:graph-tv}, which had width equal to $\varepsilon_0 = \varepsilon_n/(2\sqrt{d})$. To construct the alternatives, we will use a collection of cubes $\Xi'$ defined in exactly the same way but with width given by $\varepsilon' = (\frac{\log n_2}{n_2})^{1/d}$. Split each $\Delta_j \in \Xi'$ into two equally sized rectangles $\Delta_j^L,\Delta_j^R$: 
$$
\Delta_j^L := \frac{j}{N'} + \Big[-\varepsilon',-\frac{\varepsilon'}{2}\Big)^d, \quad \Delta_j^R := \frac{j}{N'} + \Big[-\frac{\varepsilon'}{2},0\Big)^d, \quad \textrm{for $j \in [N']^d$,}
$$
and let $\phi_{\Delta}(x) = \1(x \in \Delta^L) - \1(x \in \Delta^R)$. Take $\mc{P}_1 = \{(P_{\Delta},Q_{\Delta}): \Delta \in \Xi'\}$, where for each $\Delta \in \Xi'$ the pair $(P_{\Delta},Q_{\Delta})$ have densities
$$
p_{\Delta}(x) = p_0(x), \quad q_{\Delta}(x) = p_0(x) + \frac{1}{2}\phi_{\Delta}(x),
$$
and let $\Pi_1$ be the uniform distribution over $\mc{P}_1$. Notice that by construction $\mc{P}_1 \subset \mc{P}^{\infty}(d)$. 

\subsection{Completing the proof of Theorem~\ref{thm:lower-bound}}
Now we complete the proof by analyzing the second moment of the likelihood ratio under the choices of $\mc{P}_0,\mc{P}_1$ given in Section~\ref{subsec:pf-lower-bound-construction},  and then applying Lemma~\ref{lem:likelihood-ratio}. 

When $\mc{P}_0 = \{(P_0,P_0)\}$ is a singleton, $P_0$ is the uniform distribution over $\Omega$, and $\Pi_1$ is a uniform prior over the collection $\mc{P}_1$, the likelihood ratio is
\begin{equation*}
	R(\mc{X}_{n_1},\mc{Y}_{n_2}) = \frac{1}{|\mc{P}_1|} \sum_{(P,Q) \in \mc{P}_1} \prod_{i = 1}^{n_1} p(X_i) \cdot \prod_{i = 1}^{n_2} q(Y_i).
\end{equation*}
For the specific choice of $\mc{P}_1$ made in Section~\ref{subsec:pf-lower-bound-construction}, this is
\begin{equation*}
	R(\mc{X}_{n_1},\mc{Y}_{n_2}) = \frac{1}{|\Xi'|} \sum_{\Delta \in \Xi'} \prod_{i = 1}^{n_2} \Big(1 + \frac{1}{2}\phi_{\Delta}(Y_i)\Big),
\end{equation*}
where $|\Xi'|$ is the number of cubes in the partition $\Xi'$. Then a computation yields
\begin{equation*}
	\begin{aligned}
		\mathbb{E}_{\nu_0}\Big[R(\mc{X}_{n_1},\mc{Y}_{n_2})^2\Big] & = \frac{|\Xi'| - 1}{|\Xi'|} + \frac{1}{|\Xi'|} \Bigl(\int_{\Omega} \big(1 + \frac{1}{2}\phi_{\Delta}(x)\big)^2\,dx\Big)^{n_2} \\
		& = \frac{|\Xi'| - 1}{|\Xi'|} + \frac{1}{|\Xi'|} \Big(1 + \frac{1}{2^{d + 1}|\Xi'|}\Big)^{n_2} \\
		& \leq 1 + \frac{1}{|\Xi'|} \exp\Big(\frac{n_2}{2^{d + 1}|\Xi'|}\Big) \\
		& \leq 1 + \frac{\log n_2}{n_2^{c_0}},
	\end{aligned}
\end{equation*}
where the last line follows as $|\Xi'| = \frac{n_2}{\log n_2}$ with $c_0 = 1 - (1/2)^{d + 1}$. Let $N$ solve $\log N/N^{c_0} = 4(1 - \alpha - \beta)^2$.  We conclude that for all $n_2 \geq N$, $\Ebb_{\nu_0}[R(\mc{X}_{n_1};\mc{Y}_{n_2})^2] \leq 1 + 4(1 - \alpha - \beta)^2$ and therefore $\beta_{n_1,n_2}(\mc{P}_0,\mc{P}_1) \geq \beta$ by Lemma~\ref{lem:likelihood-ratio}. 

Finally, for each $\Delta \in \Xi'$ we have 
\begin{align*}
	\TV(\phi_{\Delta}) & = 4d\Big(\frac{\varepsilon'}{2}\Big)^{d - 1}, \quad \textrm{and} \quad \|\phi_{\Delta}\|_{L^2(\Omega)}^2 = 2\Big(\frac{\varepsilon'}{2}\Big)^d.
\end{align*}
Therefore
\begin{equation*}
	d_{\BV}(P_\Delta,Q_\Delta) \geq \frac{1}{\TV(\phi_{\Delta})}\Big(\mathbb{E}_{Q_{\Delta}}[\phi_\Delta(Y)] - \mathbb{E}_{P_{\Delta}}[\phi_\Delta(X)]\Big) = \frac{(\varepsilon'/2)^d}{4d(\varepsilon'/2)^{d - 1}} = \frac{\varepsilon'}{8d} = \frac{(\log n_2/n_2)^{1/d}}{8d}.
\end{equation*}
We conclude that if $\rho = \frac{1}{8d}(\log n_2/n_2)^{1/d}$, then for any level-$\alpha$ test $\Risk_{n_1,n_2}(\rho,\varphi) \geq \beta_{n_1,n_2}(\mc{P}_1) \geq \beta$.

\section{Proof of Theorem~\ref{thm:lower-bound-chi-squared}}
To lower bound the risk of the $\varepsilon$-chi-squared test, we separately analyze the behavior of the chi-squared statistic under null and alternative. 
\begin{itemize}
	\item Under the alternative, we derive upper bounds on the mean and variance, and use Chebyshev's inequality to give an upper bound on the right tail.
	\item Under the null, we use estimates of the mean, variance and skewness of the chi-squared statistic along with Berry-Esseen to give a lower bound on the right tail.
\end{itemize}  
Comparing the right tails under null and alternative will give the claim of the theorem.

Throughout we assume $\varepsilon \leq \eta$, since otherwise $P(\Delta) - Q(\Delta) = 0$ for all $P,Q \in \mc{P}_{\eta}$, and the distribution of $K_{\varepsilon}(\mc{X}_{n_1},\mc{Y}_{n_2})$ is the same under both null and alternative. For simplicity we will further only consider $\varepsilon = \frac{1}{2^k}\eta$ for some $k \in \mathbb{N}$. The proof is easily extended to general $\varepsilon \leq \eta$.  

\subsection{Mean, variance, skewness}
From~\eqref{eqn:poissonized} we have that the count of $\mc{X}_{n_1}$ in each cell $\Delta \in \Xi_{\varepsilon}$ is distributed $n_1P_{n_1}(\Delta) \sim \mathrm{Pois}(n P(\Delta))$, and likewise $n_2Q_{n_2}(\Delta) \sim \mathrm{Pois}(n Q(\Delta))$. Additionally $P_{n_1}(\Delta)$ and $P_{n_1}(\Delta')$ are independent for $\Delta \neq \Delta'$. This can be used to derive general formulas for the mean and variance of the chi-squared statistic that hold for any $P,Q$, and an upper bound on the central third moment under the null $P = Q = P_0$. In particular the mean of the chi-squared statistic is
\begin{equation*}
	\mathbb{E}_{P,Q}\Big[\mc{K}_{\varepsilon}(\mc{X}_{n_1},\mc{Y}_{n_2})\Big] = 2n + n^2 \sum_{\Delta \in \Xi_{\varepsilon}} \Big(P(\Delta) - Q(\Delta)\Big)^2.
\end{equation*}
The variance of the chi-squared statistic is
\begin{align*}
	& \mathrm{Var}_{P,Q}\Big[\mc{K}_{\varepsilon}(\mc{X}_{n_1},\mc{Y}_{n_2})\Big] \\
	& = 2 n^2 \sum_{\Delta \in \Xi_{\varepsilon}} \Bigl\{P(\Delta)^2 + Q(\Delta)^2\Bigr\} + 8n + 8 n^3 \sum_{\Delta \in \Xi_{\varepsilon}} \Big(P(\Delta) - Q(\Delta)\Big)^2 \Big(P(\Delta) + Q(\Delta)\Big) + 2n^2\sum_{\Delta \in \Xi_{\varepsilon}} \Big(P(\Delta) - Q(\Delta)\Big)^2.
\end{align*}
For the central third moment, let $N_{\Delta} = n_1P_{n_1}(\Delta)$ and $M_{\Delta} = n_2Q_{n_2}(\Delta)$ (for notational convenience). Note that under the null $P = Q = P_0$ we have $\Ebb(N_{\Delta} - M_{\Delta})^2 = 2nP_0(\Delta)$. So, recalling the algebraic identity $(a + b)^p \leq 2^{p - 1}(a^b + b^p)$,
\begin{align*}
	\mathbb{E}_{P_0,P_0}\Bigl|\bigl(N_\Delta - M_{\Delta}\bigr)^2 - \Ebb\big[(N_\Delta - M_{\Delta})^2\big]\Bigr|^3 
	& \leq 4\Bigl(  \mathbb{E}\Big[\big(N_{\Delta} - M_{\Delta}\big)^6\Big] + \Big[\Ebb(N_{\Delta} - M_{\Delta})^2\Big]^3 \Bigr) \\
	& = 4\Bigl(  \mathbb{E}\Big[\big(N_{\Delta} - M_{\Delta}\big)^6\Big] + 8n^3P_0(\Delta)^3 \Bigr) \\
	& \leq 4\biggl(32\Bigl\{\mathbb{E}\Big[(N_{\Delta} - nP_0(\Delta))^6\Big] + \mathbb{E}\Big[\big(M_{\Delta} - nP_0(\Delta)\big)^6\Big]\Big\} + 8n^3P_0(\Delta)^3\biggr) \\
	& \leq C\max\{n^3P_0(\Delta)^3,nP_0(\Delta)\} \\
	& =  C\max\bigg\{\frac{n^3}{|\Xi_{\varepsilon}|^3},\frac{n}{|\Xi_{\varepsilon}|}\bigg\},
\end{align*}
for $C = 2^{15} + 8$. (No attempt has been made to keep the constant small.) 

\subsection{Anti-concentration under the null}
Under the null hypothesis $P = Q = P_0$ the first two moments of the chi-squared statistic are
\begin{equation*}
	\begin{aligned}
		\mathbb{E}_{P_0,P_0}\Big[\mc{K}_{\varepsilon}(\mc{X}_{n_1},\mc{Y}_{n_2})\Big] & = 2n\\
		\mathrm{Var}_{P_0,P_0}\Big[\mc{K}_{\varepsilon}(\mc{X}_{n_1},\mc{Y}_{n_2})\Big] & = \frac{4 n^2}{|\Xi_{\varepsilon}|} + 8n.
	\end{aligned}
\end{equation*}
For notational convenience write $\sigma_0$ for the standard deviation of $K_{\varepsilon}(\mc{X}_{n_1},\mc{Y}_{n_2})$. The chi-squared statistic is the sum of $|\Xi_{\varepsilon}|$ i.i.d random variables. Applying a Berry-Esseen bound for the rate at which the sum of independent random variables converges to a Normal, we conclude that for any $t > 0$,
\begin{align}
	\mathbb{P}_{P_0,P_0}\biggl(\mc{K}_{\varepsilon}(\mc{X}_{n_1},\mc{Y}_{n_2}) \geq 2n + t\sigma_0\biggr) & \geq 1 - \Phi(t) - \frac{C\max\big\{\frac{n^3}{|\Xi_{\varepsilon}|^3},\frac{n}{|\Xi_{\varepsilon}|}\big\}}{ \sigma_0^3} \nonumber \\
	& = 1 - \Phi(t) - C \max\big(|\Xi_{\varepsilon}|^{-\frac{1}{2}},n^{-\frac{1}{2}}\big) \nonumber \\
	& \geq  1 - \Phi(t) - C n^{-\frac{1}{3}} \label{pf:chisquared-lb-1},
\end{align}
with the last line following since $|\Xi_{\varepsilon}| \geq \varepsilon^{-d} \geq  \eta^{-d}$.

\subsection{Concentration under the alternative}
For any alternative $(P_0,Q_{\Delta}) \in \mc{P}_{\eta}$, and any $\Delta' \in \Xi_{\varepsilon}$, we have
$$
P_0(\Delta'), Q_{\Delta}(\Delta') \leq \frac{2}{|\Xi_{\varepsilon}|}, \quad \textrm{and} \quad  
\Big(P_0(\Delta') - Q_\Delta(\Delta')\Big)^2 =
\begin{dcases}
	\frac{1}{|\Xi_{\varepsilon}|^2},& \quad \Delta' \subset \Delta, \\
	0,& \quad \textrm{otherwise.}
\end{dcases}
$$
Notice that there are at most $(\eta/\varepsilon)^d$ cubes $\Delta' \in \Xi_{\varepsilon}$ that lie within $\Delta$. Using these facts, a computation gives the following upper bounds on the mean and variance of the chi-squared statistic:
\begin{equation*}
	\begin{aligned}
		\mathbb{E}_{P_0,Q_{\Delta}}\Big[\mc{K}_{\varepsilon}(\mc{X}_{n_1},\mc{Y}_{n_2})\Big] & \leq 2n + \frac{n^2}{|\Xi_{\varepsilon}|^2} \times \frac{\eta^d}{\varepsilon^d} := 2n + \mathbb{M}\sigma_0 \\
		\mathrm{Var}_{P_0,Q_{\Delta}}\Big[\mc{K}_{\varepsilon}(\mc{X}_{n_1},\mc{Y}_{n_2})\Big] & \leq \sigma_0^2 + \frac{C\eta^d}{\varepsilon^d} \times \max\Big\{\frac{n^3}{|\Xi_{\varepsilon}|^3},\frac{n^2}{|\Xi_{\varepsilon}|^2}\Big\} := \sigma_0^2(1 + \mathbb{V}).
	\end{aligned}
\end{equation*}
We can further bound $\mathbb{M}$ and $\mathbb{V}$, recalling that $\eta^d \leq n^{-2/3}$ and $|\Xi_{\varepsilon}| = \varepsilon^{-d} \geq \eta^{-d}$:
\begin{align}
	\mathbb{M} & = \frac{n^2 \eta^d}{|\Xi_\varepsilon| \sigma_0} \leq c \min\Big\{\frac{n^{1/3}}{|\Xi_{\varepsilon}|^{1/2}},\frac{n^{5/6}}{|\Xi_{\varepsilon}|}\Big\} \leq c_1 \label{pf:chisquared-lb-1.5}\\
	\mathbb{V} & = \frac{C\eta^d}{\max\{\frac{n^2}{|\Xi_{\varepsilon}|}, n\}} \times \max\Big\{\frac{n^3}{|\Xi|^2},\frac{n^2}{|\Xi|}\Big\} \leq C \eta^{d} \frac{n}{|\Xi_{\varepsilon}|} \leq C n^{-1/3}. \nonumber
\end{align}
The value of $c_1$ in~\eqref{pf:chisquared-lb-1.5} depends on the constant $c$ in~\eqref{eqn:lower-bound-chi-squared} with $c_1 \to 0$ as $c \to 0$.

We conclude from Chebyshev's inequality that 
\begin{equation}
	\label{pf:chisquared-lb-2}
	\mathbb{P}_{P_0,Q_{\Delta}}\biggl(\mc{K}_{\varepsilon}(\mc{X}_{n_1},\mc{Y}_{n_2})\leq 2n + (t + c_1)\sigma_0\biggr) \geq 1 - \frac{1 + Cn^{-1/3}}{t^2}.
\end{equation}

\subsection{Completing the proof}
We now complete the proof by comparing our bounds on the upper tails under null and alternative. Let $z^{\alpha}$ denote the $(1 - \alpha)$th upper quantile of the standard Normal distribution. 

For $\alpha_n = \alpha + Cn^{-1/3}$, we have from~\eqref{pf:chisquared-lb-1} that
\begin{equation*}
	\mathbb{P}_{P_0,P_0}\biggl(\mc{K}_{\varepsilon}(\mc{X}_{n_1},\mc{Y}_{n_2}) \geq 2n + z^{\alpha_n}\sigma_0\biggr) \geq \alpha.
\end{equation*}
Thus, the threshold $t_{\alpha}$ must be at least $2n + z^{\alpha_n}\sigma_0$ in order for $\varphi_{\csq}$ to be a level-$\alpha$ test. But then from~\eqref{pf:chisquared-lb-2},
\begin{equation*}
	\mathbb{P}_{P_0,Q_{\Delta}}\biggl(\mc{K}_{\varepsilon}(\mc{X}_{n_1},\mc{Y}_{n_2})\leq t_{\alpha}\biggr) \geq 1 - \frac{1 + Cn^{-1/3}}{(z^{\alpha_n} - c_1)^2}.
\end{equation*}
Noting that $\alpha_n \to \alpha$ as $n \to \infty$ and $c_1 \to 0$ as $c \to 0$ in~\eqref{eqn:spatially-localized-testing}, it follows that $\mathbb{E}_{P_0,Q_{\Delta}}{\varphi_{\csq}} \leq 2(1/z^{\alpha})^2$ for all $n$ sufficiently large and $c$ sufficiently small.

\section{Proof of Representation~\eqref{eqn:representation-graph-tv-ipm}}
\label{sec:pf-representation-graph-tv-ipm}
We recast the ratio optimization problem in a way that allows us to invoke a result of~\citet{hein2011beyond} on submodular optimization. Let $S_a(\theta) = |\theta^{\top} a|$ and notice that
\begin{equation*}
	\max_{\theta \in \Rn} R(\theta) = \max_{\theta \in \Rn} \frac{S_a(\theta)}{\|D_G\theta\|_1}.
\end{equation*}
Notice that $S_a(\theta)$ fulfills all the conditions of Theorem~3.1 of~\citet{hein2011beyond}, and therefore by that theorem
\begin{equation*}
	\min_{\theta \in \Rn} \frac{\|D_G\theta\|_1}{S_a(\theta)} = \min_{\theta \in \{0,1\}^n} \frac{\|D_G\theta\|_1}{S_a(\theta)}.
\end{equation*}
Taking the reciprocal of both sides yields the equality in~\eqref{eqn:representation-graph-tv-ipm}.

\section{Proof of Theorem~\ref{thm:graph-ks-consistency}}
\label{sec:pf-graph-ks-consistency}

Suppose $P = Q$. In Theorem~\ref{thm:graph-tv} we give finite-sample upper bounds which imply that $d_{\DTV}(\mc{X}_{n_1},\mc{Y}_{n_2}) \to 0$ as $n_1,n_2 \to \infty$. The bounds are with sufficiently high probability that an application of the Borel-Cantelli Lemma implies that the convergence is almost sure as $n_1,n_2 \to \infty$. 

Otherwise $P \neq Q$. In this case, as discussed in the main tet, our results will rely on a mode of variational convergence known as $\Gamma$-convergence, that has been used to analyze various graph functionals such as balanced cuts~\citep{garciatrillos2016consistency}. Indeed, as we discuss in Section~\ref{subsec:interpretations}, the graph TV IPM can be viewed as kind of balanced cut -- with a balancing term that takes into account the ``labels'' of the samples $\mc{X}_{n_1},\mc{Y}_{n_2}$ -- and the general structure of our proofs follows that of~\citep{garciatrillos2016consistency}.

We begin with a brief review of $\Gamma$-convergence in general, and $\Gamma$-convergence of graph total variation, since this will help set the stage for the intermediary results we need in order to prove Theorem~\ref{thm:graph-ks-consistency}.
\subsection{Review: $\Gamma$-convergence and convergence of minimizers}
Let $F_1,F_2,\ldots, F: T \to [0,\infty]$ be non-negative functionals defined on a metric space $T = (U,d)$. (Assume throughout that functionals are not identically equal to $\infty$.) Then the sequence $(F_N) = (F_N)_{N \in \mathbb{N}}$ is said to \emph{$\Gamma$-converge} to $F$, denoted $F_N \overset{\Gamma}{\to} F$, if the following two conditions are met:
\begin{itemize}
	\item {\bf Limsup inequality}. For every point $u \in T$, there exists $(u_N) \to u$ for which
	\begin{equation*}
		\limsup_{N \to \infty} F_N(u_N) \leq F(u).
	\end{equation*}
	\item {\bf Liminf inequality}. For every convergent sequence $(u_N) \to u$ in $T$, 
	\begin{equation*}
		\liminf_{N \to \infty} F_N(u_N) \geq F(u).
	\end{equation*}
\end{itemize}
$\Gamma$ convergence is fundamentally a variational form of convergence. Suppose that a sequence $F_N \overset{\Gamma}{\to} F$ additionally satisfies the following \emph{compactness} property: for every bounded sequence $u_N \in T$ for which 
$$\limsup_{N \to \infty} F_N(u_N) < \infty,
$$
the sequence $(u_N)$ is relatively compact in $T$, i.e. it convergences along subsequences. Then it follows that
\begin{equation*}
	\lim_{N \to \infty} \inf_{u \in T} F_N(u) = \min_{u \in T} F(u).
\end{equation*}
This is the \emph{fundamental theorem of $\Gamma$-convergence}~\citep{braides2006handbook}.

\subsection{Review: $\Gamma$-convergence of graph total variation}
\label{subsec:gamma-convergence-gtv}
Let $Z_1,\ldots,Z_n,Z_{n+1},\ldots$ be independent samples from a distribution $\mu$. Assume that $\mu$ supported on $\Omega$, absolutely continuous with respect to Lebesgue measure, with a continuous density $\omega$ satisfying
$$
\frac{1}{B} \leq \omega(x) \leq B, \quad \textrm{for all $x \in \Omega$,}
$$
for some $B \in [1,\infty)$. Consider the rescaled graph total variation,
$$
\GTV_{n,\varepsilon}(u) = \frac{\DTV_{n,\varepsilon}(u)}{\sigma n^2 \varepsilon^{d + 1}}.
$$ 
Under these conditions, \citet{garciatrillos2016continuum} establish that $\GTV_{n,\varepsilon}$ $\Gamma$-converges to a continuum weighted total variation. In order to make sense of this, they define (i) a common metric space in which one can compare functions defined on $\mc{Z}_n$ to functions defined over $\Omega$, and (ii) a notion of $\Gamma$-convergence for functionals involving random samples. We review both concepts. 

\paragraph{The metric space $\mathrm{TL}^1(\Omega)$.}

\begin{definition}
	The metric space $\mathrm{TL}^1(\Omega)$ consists of pairs $(\mu,u)$, where $\mu$ is a Borel probability measure and $u \in L^1(\Omega,\mu)$, and is equipped with the metric
	\begin{equation}
		\label{eqn:tl1}
		d_{\mathrm{TL}^1(\Omega)}\Bigl((\mu,u),(\mu',v)\Bigr) := \inf_{\gamma \in \Gamma(P,Q)} \int \int |x - y| + |u(x) - v(y)| \,d\pi(x,y),
	\end{equation}
	where $\Gamma$ is the set of couplings between $\mu$ and $\mu'$, that is, the set of probability measures on $\Omega \times \Omega$ for which the marginal in the first variable is given by $\mu$, and the marginal in the second variable is given by $\mu'$.
\end{definition}
Recall that a transportation map between $\mu$ and $\mu'$ is a Borel map $T: \Omega \to \Omega$ such that the \emph{push-forward} $T_{\sharp}\mu$, defined by
$$
(T_{\sharp}\mu)(S) := \mu(T^{-1}(S)), \textrm{for Borel sets $S \subseteq \Omega$}
$$
satisfies $T_{\sharp}\mu = \mu'$. A sequence of transportation maps $(T_N)$ is said to be \emph{stagnating} if $\sup_{x \in \Omega} |T_N(x) - x| \to 0$. When $\mu$ is absolutely continuous with respect to Lebesgue measure, \citet{garciatrillos2016continuum} derive the following results involving stagnating transportation maps.
\begin{itemize}
	\item A sequence $(\mu_N,u_N) \overset{\mathrm{TL}^1(\Omega)}{\to} (\mu,u)$ if and only if there exists a sequence of stagnating transportation maps $T_N$ between $\mu$ and $\mu_N$, for which additionally $\|u_N \circ T_N - u\|_{L^1(\Omega)} \to 0$.
	\item Suppose $Z_1,\ldots,Z_n$ are drawn i.i.d from a continuous density $\mu \in \mc{P}^{\infty}(d)$. Then with probability one, there exist a sequence of transportation maps $T_n$ from $\mu$ to $\mu_n$ satisfying
	\begin{equation*}
		\limsup_{n \to \infty} \frac{n^{1/d}\|\mathrm{Id} - T_n\|_{\infty}}{(\log n)^{q_d}} \leq C,
	\end{equation*}
	where we recall that $q_d = 3/2$ for $d = 2$ and $q_d = 1$ for $d \geq 3$.
\end{itemize}

\paragraph{$\Gamma$-convergence of random functionals.}
Let $\mu_n$ be the empirical measure of random samples $Z_1,\ldots,Z_n$, and suppose $F_n(\cdot)$ are functionals defined for $u$ such that $(u,\mu_n) \in \mathrm{TL}^1(\Omega)$. Then we say $F_n \overset{\Gamma}{\to} F$ if for $\mu$-almost every sequence $Z_1,\ldots,Z_{n},Z_{n + 1},\ldots$, it is the case that $F_n(\cdot) \overset{\Gamma}{\to} F$. 

\paragraph{$\Gamma$-convergence and compactness of $\GTV_{n,\varepsilon}$.}
Let $\omega$ be the density of $\mu$. Under the conditions mentioned above, \citet{garciatrillos2016continuum} show that $\GTV_{n,\varepsilon}(\cdot) \overset{\Gamma}{\to} \TV(\cdot;\omega^2)$, meaning that for $\mu$-almost every sequence $Z_1,\ldots,Z_{n},Z_{n + 1},\ldots$ the following two statements hold.
\begin{itemize}
	\item For every $u \in L^1(\Omega,\mu)$ there exists a sequence $u_n$ converging to $u$ in $\mathrm{TL}^1(\Omega)$ such that
	\begin{equation*}
		\limsup_{n \to \infty} \GTV_{n,\varepsilon}(u_n) \leq \TV(u;\omega^2).
	\end{equation*}
	\item For every $u_n \in L^1(\Omega,\mu_n)$ converging to $u \in L^1(\Omega,\mu)$ in $\mathrm{TL}^1(\Omega)$,
	\begin{equation*}
		\liminf_{n \to \infty} \GTV_{n,\varepsilon}(u_n) \geq \TV(u;\omega^2).
	\end{equation*}
\end{itemize}
In the above statements we have used ``$u_n$ converging to $u$'' to mean $(u_n,\mu_n) \overset{\mathrm{TL}^1(\Omega)}{\to} (u,\mu)$. 

Additionally \citet{garciatrillos2016continuum} show that $\GTV_{n,\varepsilon}$ satisfies the compactness property, meaning that for any bounded sequence $(u_n,\mu_n)$ that has bounded graph total variation, i.e.
\begin{equation}
	\label{eqn:gtv-compactness}
	\limsup_{n \to \infty} \|u_n\|_{L^1(\Omega;\mu_n)} < \infty, \quad \limsup_{n \to \infty} \GTV_{n,\varepsilon}(u_n) < \infty,
\end{equation}
it is the case that $(u_n,\mu_n)$ is relatively compact in $\mathrm{TL}^1(\Omega)$. 

\paragraph{Two-sample equivalent.}
The setting of~\citet{garciatrillos2016continuum} is that of a sequence of independent and identically distributed random samples, but the conclusions hold true if we suppose instead the conditions of Theorem~\ref{thm:graph-ks-consistency}. Specifically let $X_1,X_2,\ldots,X_{n_1},X_{n_1 + 1},\ldots$ be independently sampled from $P$ and $Y_1,Y_2,\ldots,Y_{n_2},Y_{n_2 + 1},\ldots$ be independently sampled from $Q$, where $(P,Q) \in \mc{P}^{\infty}(d)$, and define $\mu_n$ to be the empirical measure of $Z_{1n},\ldots,Z_{nn}$ where $Z_{in} = X_i$ for $i = 1,\ldots,n_1$ and $Z_{in} = Y_i$ for $i = n_1 + 1,\ldots,n$. Then $\GTV_{n,\varepsilon}$ still satisfies the compactness property. If additionally $\frac{n_1}{n} \to \lambda \in [0,1]$ then $\GTV_{n,\varepsilon}(\cdot) \overset{\Gamma}{\to} \TV(\cdot;\omega_{P,Q}^2)$. 

\subsection{$\Gamma$-convergence of criterion of graph TV IPM}
The definition of $\Gamma$-convergence requires functionals defined on a common metric space. In this section, we introduce functionals $E_n$ and $E$ which are related to the criteria of the graph TV IPM and density-weighted population TV IPM, but are defined on a common metric space. Then we prove $\Gamma$-convergence of $E_n$ to $E$, which ultimately leads to the convergence of the graph TV IPM.

\paragraph{The two-sample metric space $\mathrm{TL}_2^1(\Omega)$.}
For our purposes, it will be useful to introduce a two-sample analogue to $\mathrm{TL}^1(\Omega)$. In the following definition we say $u \in L^1(\Omega,\mu,\nu)$ if $u \in L^1(\Omega,\mu)$ and $u \in L^1(\Omega,\nu)$.
\begin{definition}
	The space $\mathrm{TL}_2^1(\Omega)$ consists of triplets $(\mu,\nu,u)$, where $\mu,\nu$ are Borel probability measures on $\Omega$, and $u \in L^1(\Omega,\mu,\nu)$. It is equipped with the metric
	\begin{equation}
		d_{\mathrm{TL}_2^1(\Omega)}\Bigl((\mu,\nu,u), (\mu',\nu',v)\Bigr) = d_{\mathrm{TL}^1(\Omega)}\Bigl((\mu,u),(\mu',v)\Bigr) + d_{\mathrm{TL}^1(\Omega)}\Bigl((\nu,u),(\nu',v)\Bigr).
	\end{equation}
\end{definition}
We make the following basic observations:
\begin{itemize}
	\item $(\mu_N,\nu_N,u_N) \overset{\mathrm{TL}_2^1(\Omega)}{\to} (\mu,\nu,u)$ if and only if $(\mu_N,u_N) \overset{\mathrm{TL}^1(\Omega)}{\to} (\mu,u)$ and $(\nu_N,u_N) \overset{\mathrm{TL}^1(\Omega)}{\to} (\nu,u)$.
	\item If $(\mu_N,\nu_N,u_N) \overset{\mathrm{TL}_2^1(\Omega)}{\to} (\mu,\nu,u)$ then $(\mu_N,u_N) \overset{\mathrm{TL}^1(\Omega)}{\to} (\mu,u)$ and $(\nu_N,u_N) \overset{\mathrm{TL}^1(\Omega)}{\to} (\nu,u)$.
	\item Suppose $\mu$ and $\nu$ are absolutely continuous with respect to Lebesgue measure, and suppose $(\mu_N,\nu_N,u_N) \overset{\mathrm{TL}_{2}^1(\Omega)}{\to} (\mu,\nu,u)$. Then there exist stagnating transportation maps $T_N^\mu$ between $\mu$ and $\mu_N$, and $T_N^\nu$ between $Q$ and $Q_N$. 
	%Additionally
	%\begin{equation*}
	%\Big|\mathbb{E}_{X \sim \mu_N}[u_N(X)] - \mathbb{E}_{X \sim \mu}[u(X)]\Big| %\leq \int_{\Omega} \Big|u \circ T_N^{\mu}(x) - u(x)\Big| \mu(x) \,dx \leq %\|u \circ T_N^{\mu} - u\|_{L^1(\Omega)} \|\mu\|_{L^{\infty}(\Omega)} \to 0;
	%\end{equation*}
	%and so $\mathbb{E}_{X \sim \mu_N}[u_N(X)] \to \mathbb{E}_{X \sim \mu}[u(X)]$. Likewise $\mathbb{E}_{Y \sim \nu_N}[u_N(Y)] \to \mathbb{E}_{Y \sim \nu}[u(Y)]$.
\end{itemize}

% In the special case where $\mu_N = P_{n_1}$ and $\nu_N = Q_{n_2}$ we will write ``$u_N$ converges to $u$ in $\mathrm{TL}_2^1(\Omega)$'' to mean $(\mu_N,\nu_N,u_N) \overset{\mathrm{TL}_2^1(\Omega)}{\to} (\mu,\nu,u)$, and write $\mu_n = \mu_{P_{n_1},Q_{n_2}}$ and $\mu = \mu_{P,Q}$.

\paragraph{Discrete and continuum functionals.}
We now define the functionals $E_n$ and $E$ mentioned above. 

To define the discrete functional, let $B_n(u) = |P_{n_1}(u) - Q_{n_2}(u)|$ for $(u,P_{n_1},Q_{n_2}) \in\mathrm{TL}_2^1(\Omega)$; this is equivalent to the balance term in the {ratio-based formulation} of graph TV IPM except defined for $u \in L^1(\Omega,P_{n_1},Q_{n_2})$ rather than $\theta \in \Rn$. Define $L_{B_n}^1(\Omega, P_{n_1},Q_{n_2})$ to be
\begin{equation}
	L_{B_n}^1(P_n,Q_n) := \Big\{u = \frac{v}{B_n(v)}:  v \in L^1(\Omega,P_n,Q_n), B_n(v) > 0\Big\}.
\end{equation}
The functional $E_n: \mathrm{TL}_2^1(\Omega) \to [0,\infty]$ is defined as
\begin{equation}
	E_n(u, P_{n_1}, Q_{n_2}) := 
	\begin{dcases}
		\GTV_{n,\varepsilon_n}(u),& \quad \textrm{$u_n \in L_{B_n}^1(P_{n_1}, Q_{n_2})$}, \\
		\infty,& \quad \textrm{otherwise.}
	\end{dcases}
\end{equation}
To define the continuum functional, let $B(u) = |P(u) - Q(u)|$ for $u \in L^1(\Omega,P,Q)$. Define $L_B^1(\Omega,P,Q)$ to be
\begin{equation}
	L_{B}^1(\Omega,P,Q) := \Big\{u = \frac{v}{B(v)}:  v \in L^1(\Omega,P,Q), B(v) > 0\Big\},
\end{equation}
and define the functional $E: \mathrm{TL}_2^1(\Omega) \to [0,\infty]$ to be
\begin{equation}
	E(u, P, Q) := 
	\begin{dcases}
		\mathrm{TV}(u;\omega_{P,Q}^2),& \quad \textrm{$u \in L_{B}^1(\Omega,P,Q)$}, \\
		\infty,& \quad \textrm{otherwise.}
	\end{dcases}
\end{equation}
We will write $E_n(u_n) = E_n(u_n,P_{n_1},Q_{n_2})$ and $E(u) = E(u,P,Q)$ when the measures are clear from context.

The definitions of $E_n,E$ may appear somewhat strange at first. The reason we define them this way -- that is, the reason we define them to be equal to $\GTV_{n,\varepsilon}(u)$ for functions $u$ that are normalized by a balance term $B_n(u)$, rather than in terms of the ratio $\GTV_{n,\varepsilon}(u)/B_n(u)$ -- is so that we can make use of the $\Gamma$-convergence of $\GTV_{n,\varepsilon}$ to $\TV(\cdot,\omega^2)$, and the compactness of $\GTV_{n,\varepsilon}$, in as straightforward a way as possible. A similar device was employed in~\citet{garciatrillos2016consistency} to analyze the convergence of various balanced cut functionals on graphs.

\paragraph{$\Gamma$-convergence of $E_n$.}

\begin{proposition}
	\label{prop:gamma-convergence}
	Under the conditions of Theorem~\ref{thm:graph-ks-consistency}, $E_n \overset{\Gamma}{\to} E$, meaning that for $P$-almost every $X_1,\ldots,X_{n_1},X_{n_1 + 1},\ldots$ and $Q$-almost every $Y_1,\ldots,Y_{n_2},Y_{n_2 + 1},\ldots$, the following two statements hold.
	\begin{itemize}
		\item For every $u \in L^1(\Omega,P,Q)$, there exists $u_n$ converging to $u$ in $\mathrm{TL}_2^1(\Omega)$, such that
		\begin{equation}
			\label{eqn:limsup}
			\limsup_{n \to \infty} E_n(u_n) \leq E(u).
		\end{equation}
		\item For every $u \in L^1(\Omega,P,Q)$, and all $u_n \in L^1(\Omega,P_{n_1},Q_{n_2})$ converging to $u$ in $\mathrm{TL}_2^1(\Omega)$, 
		\begin{equation}
			\label{eqn:liminf}
			\liminf_{n \to \infty} E_n(u_n) \geq E(u).
		\end{equation}
	\end{itemize} 
\end{proposition}
In the above statements we have used ``$u_n$ converging to $u$'' to mean $(u_n,P_{n_1},Q_{n_2}) \overset{\mathrm{TL}^1(\Omega)}{\to} (u,P,Q)$. 

\begin{proof}
	First we are going to prove the limsup inequality, then the liminf inequality. Throughout, we assume the existence of stagnating transportation maps $T_{n_1}^{P}$ from $P$ to $P_{n_1}$, and $T_{n_2}^{Q}$ from $Q$ to $Q_{n_2}$, keeping in mind that these exist with probability one.
	
	\textit{\underline{Limsup inequality.}} 
	As is standard, we are going to prove the limsup inequality~\eqref{eqn:limsup} for every $u$ in a dense subset of $L^1(\Omega,P,Q)$, in particular every Lipschitz function $u$. This implies the result holds for all $u \in L^1(\Omega,P,Q)$ by a diagonal argument.
	
	Suppose $u \not\in L_{B}^1(P, Q)$. Then taking $u_n$ to be the restriction of $u$ to $\mc{Z}_{n}$, we have
	\begin{equation*}
		\|u_n \circ T_n^P - u\|_{L^1(\Omega)} \leq C \max_{x \in \Omega} |T_n^P(x) - x|,
	\end{equation*}
	where $C$ depends on $\Omega$ and the Lipschitz constant of $u$. The same holds for $T_n^Q$. Since $T_n^P$ and $T_n^Q$ are stagnating, it follows that $u_n$ converges to $u$ in $\mathrm{TL}_2^1(\Omega)$. Obviously,
	\begin{equation*}
		\limsup_{n \to \infty} E_n(u_n) \leq \infty = E(u),
	\end{equation*}
	and this proves the claim for $u \not\in L_{B}^1(P,Q)$.
	
	On the other hand, if $u \in L_{B}^1(P,Q)$ then $u = v/B(v)$ for some $v \in L^1(\Omega,P,Q)$. Note that $v$ must also be a Lipschitz function, since $u$ is Lipschitz and
	\begin{equation*}
		B(v) \leq C\big(\|p\|_{L^\infty(\Omega)} + \|q\|_{L^{\infty}(\Omega)}\big)\|v\|_{L^1(\Omega)} < \infty.
	\end{equation*}
	Then, taking $v_n$ to be the restriction of $v$ to $Z_{1:n}$, the same argument as above shows that $v_n$ converges to $v$ in $\mathrm{TL}_{2}^1(\Omega)$. It follows that $P_{n_1}(v_n) \to P(v), Q_{n_2}(v_n) \to Q(v)$, $B_n(v_n) \to B(v)$, and
	\begin{equation*}
		\limsup_{n \to \infty} \mathrm{GTV}_{n,\varepsilon}(v_n) \leq \mathrm{TV}(v;\omega_{P,Q}^2).
	\end{equation*} 
	Now take $u_n = v_n/B_n(v_n)$, which is well-defined for all $n$ sufficiently large, since $B(v_n) \to B(v) > 0$. Then $u_n \in L_{B_n}^1(P_{n_1},Q_{n_2})$, $u_n$ converges to $u$ in $\mathrm{TL}_2^1(\Omega)$, and 
	\begin{equation*}
		\limsup_{n \to \infty} E_n(u_n) = \limsup_{n \to \infty} \frac{ \mathrm{GTV}_{n,\varepsilon}(v_n)}{B_n(v_n)} \leq \frac{\mathrm{TV}(v;\omega_{P,Q}^2)}{B(v)} = E(u).
	\end{equation*} 
	
	\textit{\underline{Liminf inequality.}} 
	To start, suppose $u \not\in L_{B}^1(P,Q)$. If $u_n$ converges to $u$ in $\mathrm{TL}_2^1(\Omega)$ then
	\begin{equation*}
		B_n(u_n) \to B(u) \neq 1,
	\end{equation*} 
	and so for all $n$ sufficiently large, $B_n(u_n) \neq 1$. Consequently,
	\begin{equation*}
		\liminf_{n \to \infty} E_n(u_n) = \infty = E(u).
	\end{equation*}
	Otherwise $u \in L_{B}^1(P,Q)$, and
	\begin{equation*}
		\liminf_{n \to \infty} E_n(u_n) = \liminf_{n \to \infty} \mathrm{GTV}_{n,\varepsilon_n}(u_n) \geq \mathrm{TV}(u;\omega_{P,Q}^2) = E(u).
	\end{equation*}
\end{proof}
\paragraph{Compactness of $E_n$.}
In what follows let $\mathrm{mean}_n(u) := \frac{1}{n}\sum_{i = 1}^{n} u_i$ be the sample average of $u$, and say $u$ is \emph{mean-zero} if $\mathrm{mean}_n(u) = 0$.
\begin{proposition}
	\label{prop:compactness}
	Under the conditions of Theorem~\ref{thm:graph-ks-consistency}, the sequence of functionals $(E_n)$ is \emph{precompact}, meaning that every mean-zero sequence $(u_n)$ for which $\sup_{n} E_n(u_n) < \infty$ is relatively compact in $\mathrm{TL}_1^2(\Omega)$.
\end{proposition}
\begin{proof}
	We will establish that any sequence $(u_n)$ satisfying the conditions of the theorem also satisfies~\eqref{eqn:gtv-compactness}, which implies the claim.
	
	First of all, since $E_n(u) \geq \GTV_{n,\varepsilon}(u)$ it follows that $\sup_{n} \GTV_{n,\varepsilon}(u) < \infty$.
	
	To upper bound the $L^1$-norm of $u$ we make use of the following pair of results involving the ratio cut functional. The first is a representation result due to~\cite{hein2011beyond}:
	\begin{equation*}
		\sup_u \; \frac{\GTV_{n,\varepsilon}(u)}{\|u - \mathrm{mean}_n(u)\|_{L^1(\Omega;\mu_n)}} = \sup_{S} \frac{\GTV_{n,\varepsilon}(\1_S)}{\|u - \mathrm{mean}_n(\1_S)\|_{L^1(\Omega;\mu_n)}}.
	\end{equation*}
	The right hand side of the above expression is the \emph{ratio cut}. \citet{garciatrillos2016continuum} show that the ratio cut is asymptotically consistent, under conditions analogous to those of Theorem~\ref{thm:graph-ks-consistency}:
	\begin{equation*}
		\inf_{S} \frac{\GTV_{n,\varepsilon}(\1_S)}{\|u - \mathrm{mean}_n(\1_S)\|_{L^1(\Omega;\mu_n)}} \to \inf_{A} \frac{\TV(\1_S;\omega_{P,Q}^2)}{2|A||A^c|}.
	\end{equation*}
	The results of~\citet{garciatrillos2016consistency} assume i.i.d samples, but are straightforwardly adapted to our two-sample setting, see the discussion in Section~\ref{subsec:gamma-convergence-gtv}.
	
	Taken together, we have that for any mean-zero sequence $u_n$ for which $\sup_{n} E_n(u_n) < \infty$,
	\begin{align}
		\limsup_{n \to \infty} \|u_n\|_{L^1(\Omega;\mu_n)} & \leq \limsup_{n \to \infty} \GTV_{n,\varepsilon}(u_n) \times \Big(\inf_{A} \frac{\TV(\1_S;\omega_{P,Q}^2)}{2|A||A^c|}\Big)^{-1} < \infty.
	\end{align}
	
\end{proof}

\subsection{Completing the proof of Theorem~\ref{thm:graph-ks-consistency}}
Let us make an explicit connection between the discrete functionals $E_n(\cdot)$ and $d_{\DTV}(\cdot,\cdot)$, and the continuum functionals $E(\cdot)$ and $d_{\BV}(\cdot,\cdot)$, by observing that
$$
\sigma n^2 \varepsilon^{d + 1} d_{\DTV}(\mc{X}_{n_1}, \mc{Y}_{n_2}) = \Big(\min E_n(u_n) \Big)^{-1}, \quad d_{\BV}(P,Q;\omega_{P,Q}^2) = \Big(\min E(u)\Big)^{-1}.
$$
Thus we can complete the proof of Theorem~\ref{thm:graph-ks-consistency} by establishing that $\min E_n(u_n) \to \min E(u)$, which we do using the results of Propositions~\ref{prop:gamma-convergence} and~\ref{prop:compactness}.

Let $f^{\ast}$ be a witness of $d_{\BV}(P,Q;\omega_{P,Q}^2)$, and let $u^{\ast} = f^{\ast}/B(f^{\ast})$. Then $u^{\ast}$ is a minimizer of $E(u)$. Observe that $E(u^{\ast}) < \infty$ as $d_{\BV}(P,Q;\omega_{P,Q}^2) > 0$. From the limsup inequality of Proposition~\ref{prop:gamma-convergence}, there exists a sequence $(u_n)$ for which
$$
\limsup_{n \to \infty} E_n(u_n) \leq E(u^{\ast}).
$$ 
Now, let $u_n^{\ast} = \argmin E_n(u_n)$. We may as well take $u_n^{\ast}$ to be mean-zero since $E_n$ is invariant under shifts, $E_n(u - c) = E_n(u)$ for any $c \in \Reals$. Since additionally
$$
\limsup_{n \to \infty} E_n(u_n^{\ast}) \leq \limsup_{n \to \infty} E_n(u_n) < E(u^{\ast}) < \infty,
$$
we conclude that $u_n^{\ast}$ is relatively compact in $\mathrm{TL}_2^1(\Omega)$ by Proposition~\ref{prop:compactness}. 

Let us assume without loss of generality that $u_n^{\ast} \to u$ in $\mathrm{TL}_2^1(\Omega)$ (otherwise, work along subsequences). In that case, it follows from the liminf inequality of Proposition~\ref{prop:gamma-convergence} that
$$
\liminf_{n \to \infty} E_n(u_n^{\ast}) \geq E(u) \geq E(u^{\ast}).
$$ 
So the limit of $E_n(u_n^{\ast})$ exists and is equal to $E(u^{\ast})$, which is the desired result.

\section{Technical results}
\label{sec:technical-results}

\subsection{Total variation and total variation IPM}
\label{subsec:total-variation}

In this section we collect a number of useful facts about TV and TV IPM. Throughout we assume that $(P,Q) \in \mc{P}^{\infty}(d)$ meaning in particular that densities $p,q$ exist and furthermore that $p,q \in L^{\infty}(\Rd)$.

\paragraph{Basic properties of total variation.}
Total variation satisfies the coarea formula
$$
\TV(f) = \int_{-\infty}^{\infty} \per(\{f > t\}) \,dt.
$$
Additionally, total variation is lower semi-continuous: if $f_k \to f$ weakly in $L^p(\Rd)$, for any $1 \leq p < \infty$, then
$$
\TV(f) \leq \liminf_{k \to \infty} \TV(f_k).
$$
Given a domain $\Omega \subseteq \Rd$, the total variation of $f \in L^1(\Omega)$ is defined as
$$
\TV(f;\Omega) := \sup\Bigl\{\int_{\Omega} f \cdot \mathrm{div}(\phi): \phi \in C_c^{\infty}(\Omega), 0 \leq \|\phi(x)\|_2 \leq 1~~\textrm{for all $x \in \Omega$}\Bigr\}. 
$$
It follows that $\TV(f;\Omega) \leq \TV(f)$.

\paragraph{Sobolev and isoperimetric inequalities.}
We record a Sobolev inequality for $f \in \BV(\Rd)$, as given in Theorem 14.33 of~\citet{leoni2017first}: there exists a constant $C$ depending only on $d$ such that
\begin{equation}
	\label{eqn:sobolev-inequality}
	\|f\|_{L^{d/(d - 1)}(\Rd)} \leq C \cdot \TV(f).
\end{equation}
This Sobolev inequality can be used to prove an isoperimetric inequality that we will also need. Let $A \subseteq \Rd$ be a set of finite perimeter, and take $A^c = \Rd \setminus A$. Then either $A$ or $A^c$ has finite Lebesgue measure, and there exists a constant $C$ depending only on $d$ such that
\begin{equation}
	\label{eqn:isoperimetric-inequality}
	\min\Bigl\{\nu(A), \nu(A^c)\Bigr\}^{(d - 1)/d} \leq C \cdot \per(A).
\end{equation}
See Theorem 14.44 of~\citet{leoni2017first}.

\paragraph{Finiteness of TV IPM.}
Here we show that the the population TV IPM is finite for $(P,Q) \in \mc{P}^{\infty}(d)$. Applying H\"{o}lder's inequality, the Sobolev inequality~\eqref{eqn:sobolev-inequality}, and then using the fact that $p \in L^{\infty}(\Rd)$ and $\int p(x) \,dx = 1$,
\begin{equation}
	\label{eqn:tv-ipm-finite}
	\begin{aligned}
		\int f(x) p(x) \,dx 
		& \leq \Bigl(\int f(x)^{d/(d - 1)}\,dx\Bigr)^{(d - 1)/d} \Bigl(\int p(x)^{d}\,dx\Bigr)^{1/d} \\
		& \leq C \hspace{1 pt} \TV(f) \Bigl(\int p(x)^{d}\,dx\Bigr)^{1/d} \\
		& \leq C \hspace{1 pt} \TV(f) \|p\|_{L^{\infty}(\Rd)}^{(d - 1)/d}.
	\end{aligned}
\end{equation}
The same inequality holds with respect to $q$.

\paragraph{Existence of maximizer of TV IPM.}
We show here that there exists a (not necessarily unique) function $f^{\ast}$ for which $\TV(f^{\ast}) \leq 1$ and which achieves the maximum value of the criterion of $d_{\BV}(P,Q)$, that is 
\begin{equation}
	\label{eqn:dual-tv-norm}
	\int \bigl(p(x) - q(x)\bigr) f^{\ast}(x) \,dx = \sup\Bigl\{ \int \bigl(p(x) - q(x)\bigr) f(x) \,dx: \TV(f) \leq 1 \Bigr\}.
\end{equation}
Consider a maximizing sequence of~\eqref{eqn:dual-tv-norm}: that is, a sequence of integrable $f_k$ for which $\TV(f_k) \leq 1, k = 1,2,\ldots$ and 
$$
\lim_{k \to \infty} \int f_k(x) \bigl(p(x) - q(x)\bigr)  \,dx = d_{\BV}(P,Q).
$$ 
It follows from $\TV(f_k) \leq 1$ and the Sobolev inequality~\eqref{eqn:sobolev-inequality} that $(f_k)_{k \in \mathbb{N}}$ is bounded in $L^{d/(d- 1)}(\Rd)$.  This means there exists a subsequence (also denoted by $f_k$) which converges weakly to an element $f \in L^{d/(d- 1)}(\Rd)$, meaning 
$$
\int f_k(x) g(x) \,dx \to \int f(x) g(x) \,dx, \quad \textrm{for all $g \in L^d(\Rd)$.}
$$
However by the lower semi-continuity of $\TV$ we have that $\liminf_{k \to \infty} \TV(f_k) \geq \TV(f)$. Noting that $p - q \in L^d(\Rd)$ (see~\eqref{eqn:tv-ipm-finite}), we have shown the desired result.

\paragraph{Characteristic property of TV IPM.}
The TV IPM is characteristic, meaning $d_{\BV}(P,Q) = 0$ only if $P = Q$. This can be seen by the following chain of implications:
\begin{align*}
	& \int \bigl(p(x) - q(x)\bigr)f(x) \,dx = 0, \quad \textrm{for all $f \in \BV(\Rd)$} \\
	& \Longrightarrow \int \bigl(p(x) - q(x)\bigr)f(x) \,dx = 0, \quad \textrm{for all $f \in C_c^{\infty}(\Rd)$} \\
	&  \Longrightarrow \int \bigl(p(x) - q(x)\bigr)f(x) \,dx = 0, \quad \textrm{for all $f \in C_c(\Rd)$},
\end{align*}
with the second implication following by density of $C_c^{\infty}(\Rd)$ in $C_c(\Rd)$. By the Riesz-Markov-Kakutani Theorem this identifies $P - Q$ with $0$.

\subsection{Concentration inequalities}
\label{subsec:concentration}

\paragraph{Hoeffding's inequality.}  Let $U_1,\ldots,U_n$ be independent random variables with mean $\mathbb{E}[U_i] = \mu$ and for which $|U_i| \leq b$ with probability $1$. Then for any $t \geq 0$,
\begin{equation*}
	\mathbb{P}\biggl(\Bigl|\sum_{i = 1}^{n}U_i - n\mu\Bigr| \geq t \biggr) \leq 2\exp\Bigl(-\frac{t^2}{2nb^2}\Bigr),
\end{equation*}
and it follows that with probability at least $1 - \delta$,
\begin{equation}
	\label{eqn:hoeffding}
	\Bigl|\sum_{i = 1}^{n}U_i - n\mu\Bigr| \leq b\sqrt{2 n \log(2/\delta)}.
\end{equation}

\paragraph{Multiplicative Chernoff bound.}
Let $U_1,\ldots,U_n$ be independent random variables with mean $\sum_{i = 1}^{n}\mathbb{E}[U_i] = n\mu$ and for which $|U_i| \leq 1$ with probability $1$. Then for any $0 \leq t \leq 1$,
\begin{equation*}
	\mathbb{P}\biggl(\Bigl|\sum_{i = 1}^{n}U_i - n\mu\Bigr| \geq t n \mu \biggr) \leq 2\exp\Bigl(-\frac{t^2n\mu}{3}\Bigr),
\end{equation*}
and it follows that with probability at least $1  - \delta$,
\begin{equation}
	\label{eqn:multiplicative-chernoff}
	\Bigl|\sum_{i = 1}^{n}U_i - n\mu\Bigr| \leq \sqrt{3 \log(2/\delta) n \mu}.
\end{equation}

\paragraph{Bernstein's inequality.} Let $U_1,\ldots,U_n$ be independent random variables with mean $\mathbb{E}[U_i] = 0$, variance $\sum_{i = 1}^{n} \Var[U_i] \leq \sigma^2$, and for which $|U_i| \leq b$ with probability $1$. Then for any $t \geq 0$,
\begin{equation*}
	\mathbb{P}\biggl(\Bigl|\sum_{i = 1}^{n}U_i\Bigr| \geq t \biggr) \leq 2\exp\Bigl(-\frac{t^2}{2(\sigma^2 + bt)}\Bigr)
\end{equation*}
and it follows that with probability at least $1 - \delta$,
\begin{equation}
	\label{eqn:bernstein}
	\Bigl|\sum_{i = 1}^{n}U_i\Bigr| \leq \sigma\sqrt{2\log(2/\delta)} + 2b \log(2/\delta)
\end{equation}	

\paragraph{Bernstein-type inequality for randomly permuted sums.} 
Let $\{b_{ij}\}_{i,j = 1}^{n}$ be an $n \times n$ array of numbers, and $\Pi \sim \mathrm{Unif}(S_n)$ be a randomly chosen permutation. Set $U_i = b_{i,\Pi(i)}$ for $i = 1,\ldots,n$, and suppose the mean $\mathbb{E}[U_i] = 0$, variance $\Var[\sum_{i = 1}^{n}U_i] \leq \sigma^2$ and $|U_i| \leq b$ with probability $1$. Then for any $t \geq 0$,
\begin{equation*}
	\mathbb{P}\biggl(\Bigl|\sum_{i = 1}^{n}U_i\Bigr| \geq t \biggr) \leq 16e^{1/16}\exp\Bigl(-\frac{t^2}{256(\sigma^2 + bt)}\Bigr),
\end{equation*}
and it follows that with probability at least $1 - \delta$,
\begin{equation}
	\label{eqn:bernstein-permuted}
	\Bigl|\sum_{i = 1}^{n}U_i\Bigr| \leq \sigma\sqrt{256\log(16e^{1/16}/\delta)} + 256 b \log(16e^{1/16}/\delta).
\end{equation}

\section{Numerical Experiments: Additional Details}
\label{sec:numerical-experiments-additional-details}

Here we give describe how the binned graph TV test is computed for the simulations in Section~\ref{subsec:simulations}. Recall that the binwidth of the binned graph TV test is taken to be $\varepsilon = 0.02$. Let $N = 1/\varepsilon = 50$. The domain $\Omega = (0,1)^2$ is partitioned into $N^2 = 2500$ bins of equal size. Denote the resulting partition by $\Xi_{\varepsilon}$. Let $G_{\Xi}'$ denote the torus graph over $\Xi_{\varepsilon}$,  and let $P_{\Xi}a$ denote the average of the normalized assignment vector $a$ over cells in $\Xi$. The binned graph TV IPM is then
$$
\max_{\|D_{G_{\Xi}'}\theta\|_1 \leq 1} \theta^{\top} P_{\Xi} a.
$$
The averaging operator $P_{\Xi}$ is defined in Section~\ref{subsec:piecewise-constant-approximation}. The two-dimensional torus graph $G_{\Xi}' = (\Xi,\ttE_{\Xi}')$, where $\{\Delta,\Delta'\} \in G_{\Xi}'$ if and only if $\Delta = \Delta_j, \Delta' = \Delta_{(j \pm e_i) \mod N}$ for $i \in \{1,2\}$. By using the torus graph $G_{\Xi}'$, as opposed to the grid graph $G_{\Xi}$, we avoid giving lower graph TV to vectors $\theta$ supported near the boundary of $\Omega$.

\end{document}